\documentclass[preprint,3p,sort&compress]{elsarticle}
\journal{{\tt arXiv.org}}

\usepackage[dvips]{epsfig}
\usepackage{graphicx} 
\usepackage{pdfpages}
\usepackage{latexsym}
\usepackage{verbatim}
\usepackage{amsmath,amsthm}
\usepackage[utf8]{inputenc}
\usepackage{amssymb}
\usepackage{bbm}
\usepackage{fonts}
\usepackage{enumerate}
\usepackage{placeins}
\usepackage{standalone}
\usepackage{paralist}
\usepackage{subcaption}

\usepackage[]{hyperref} 

\usepackage{pgfplots}
\pgfplotsset{compat=newest}       
\newcommand{\externaltikz}[2]{\includegraphics{Externals/#1}}

 \usepackage{relinput}
\newtheorem{theorem}{Theorem}[section]
\newtheorem{definition}[theorem]{Definition}
\newtheorem{remark}[theorem]{Remark}
\newtheorem{example}[theorem]{Example}
\newtheorem{assumption}[theorem]{Assumption}

\newtheorem{lemma}[theorem]{Lemma}

\newcounter{tikzsubfigcounter}[figure]
\renewcommand{\thetikzsubfigcounter}{\the\numexpr\value{figure}+1\relax\alph{tikzsubfigcounter}}

\newcounter{tikzsubfigcounterinvisible}[figure]
\renewcommand{\thetikzsubfigcounterinvisible}{\the\numexpr\value{figure}+1\relax\alph{tikzsubfigcounterinvisible}}

\newcommand{\settikzlabel}[1]{ %
	\refstepcounter{tikzsubfigcounterinvisible} \label{#1} 
}

\newtheorem{thm}{\bf Theorem}[section]

\newtheorem{rem}[thm]{\bf  Remark}

\numberwithin{equation}{section}

\title{A \hyperbolicity-preserving stochastic Galerkin approximation for uncertain hyperbolic systems of equations}

\author[ls]{Louisa Schlachter}
\address[ls]{Fachbereich Mathematik, TU Kaiserslautern, Erwin-Schr\"odinger-Str., 67663 Kaiserslautern, Germany, {\tt schlacht@mathematik.uni-kl.de}}

\author[fs]{Florian Schneider}
\address[fs]{Fachbereich Mathematik, TU Kaiserslautern, Erwin-Schr\"odinger-Str., 67663 Kaiserslautern, Germany, {\tt schneider@mathematik.uni-kl.de}}

\date{}

\definecolor{greenyellow}   {cmyk}{0.15, 0   , 0.69, 0   }
\definecolor{yellow}        {cmyk}{0   , 0   , 1   , 0   }
\definecolor{goldenrod}     {cmyk}{0   , 0.10, 0.84, 0   }
\definecolor{dandelion}     {cmyk}{0   , 0.29, 0.84, 0   }
\definecolor{apricot}       {cmyk}{0   , 0.32, 0.52, 0   }
\definecolor{peach}         {cmyk}{0   , 0.50, 0.70, 0   }
\definecolor{melon}         {cmyk}{0   , 0.46, 0.50, 0   }
\definecolor{yelloworange}  {cmyk}{0   , 0.42, 1   , 0   }
\definecolor{orange}        {cmyk}{0   , 0.61, 0.87, 0   }
\definecolor{burntorange}   {cmyk}{0   , 0.51, 1   , 0   }
\definecolor{bittersweet}   {cmyk}{0   , 0.75, 1   , 0.24}
\definecolor{redorange}     {cmyk}{0   , 0.77, 0.87, 0   }
\definecolor{mahogany}      {cmyk}{0   , 0.85, 0.87, 0.35}
\definecolor{maroon}        {cmyk}{0   , 0.87, 0.68, 0.32}
\definecolor{brickred}      {cmyk}{0   , 0.89, 0.94, 0.28}
\definecolor{red}           {cmyk}{0   , 1   , 1   , 0   }
\definecolor{orangered}     {cmyk}{0   , 1   , 0.50, 0   }
\definecolor{rubinered}     {cmyk}{0   , 1   , 0.13, 0   }
\definecolor{wildstrawberry}{cmyk}{0   , 0.96, 0.39, 0   }
\definecolor{salmon}        {cmyk}{0   , 0.53, 0.38, 0   }
\definecolor{carnationpink} {cmyk}{0   , 0.63, 0   , 0   }
\definecolor{magenta}       {cmyk}{0   , 1   , 0   , 0   }
\definecolor{violetred}     {cmyk}{0   , 0.81, 0   , 0   }
\definecolor{rhodamine}     {cmyk}{0   , 0.82, 0   , 0   }
\definecolor{mulberry}      {cmyk}{0.34, 0.90, 0   , 0.02}
\definecolor{redviolet}     {cmyk}{0.07, 0.90, 0   , 0.34}
\definecolor{fuchsia}       {cmyk}{0.47, 0.91, 0   , 0.08}
\definecolor{lavender}      {cmyk}{0   , 0.48, 0   , 0   }
\definecolor{thistle}       {cmyk}{0.12, 0.59, 0   , 0   }
\definecolor{orchid}        {cmyk}{0.32, 0.64, 0   , 0   }
\definecolor{darkorchid}    {cmyk}{0.40, 0.80, 0.20, 0   }
\definecolor{purple}        {cmyk}{0.45, 0.86, 0   , 0   }
\definecolor{plum}          {cmyk}{0.50, 1   , 0   , 0   }
\definecolor{violet}        {cmyk}{0.79, 0.88, 0   , 0   }
\definecolor{royalpurple}   {cmyk}{0.75, 0.90, 0   , 0   }
\definecolor{blueviolet}    {cmyk}{0.86, 0.91, 0   , 0.04}
\definecolor{periwinkle}    {cmyk}{0.57, 0.55, 0   , 0   }
\definecolor{cadetblue}     {cmyk}{0.62, 0.57, 0.23, 0   }
\definecolor{cornflowerblue}{cmyk}{0.65, 0.13, 0   , 0   }
\definecolor{midnightblue}  {cmyk}{0.98, 0.13, 0   , 0.43}
\definecolor{navyblue}      {cmyk}{0.94, 0.54, 0   , 0   }
\definecolor{royalblue}     {cmyk}{1   , 0.50, 0   , 0   }
\definecolor{blue}          {cmyk}{1   , 1   , 0   , 0   }
\definecolor{cerulean}      {cmyk}{0.94, 0.11, 0   , 0   }
\definecolor{cyan}          {cmyk}{1   , 0   , 0   , 0   }
\definecolor{processblue}   {cmyk}{0.96, 0   , 0   , 0   }
\definecolor{skyblue}       {cmyk}{0.62, 0   , 0.12, 0   }
\definecolor{turquoise}     {cmyk}{0.85, 0   , 0.20, 0   }
\definecolor{tealblue}      {cmyk}{0.86, 0   , 0.34, 0.02}
\definecolor{aquamarine}    {cmyk}{0.82, 0   , 0.30, 0   }
\definecolor{bluegreen}     {cmyk}{0.85, 0   , 0.33, 0   }
\definecolor{emerald}       {cmyk}{1   , 0   , 0.50, 0   }
\definecolor{junglegreen}   {cmyk}{0.99, 0   , 0.52, 0   }
\definecolor{seagreen}      {cmyk}{0.69, 0   , 0.50, 0   }
\definecolor{green}         {cmyk}{1   , 0   , 1   , 0   }
\definecolor{forestgreen}   {cmyk}{0.91, 0   , 0.88, 0.12}
\definecolor{pinegreen}     {cmyk}{0.92, 0   , 0.59, 0.25}
\definecolor{limegreen}     {cmyk}{0.50, 0   , 1   , 0   }
\definecolor{yellowgreen}   {cmyk}{0.44, 0   , 0.74, 0   }
\definecolor{springgreen}   {cmyk}{0.26, 0   , 0.76, 0   }
\definecolor{olivegreen}    {cmyk}{0.64, 0   , 0.95, 0.40}
\definecolor{rawsienna}     {cmyk}{0   , 0.72, 1   , 0.45}
\definecolor{sepia}         {cmyk}{0   , 0.83, 1   , 0.70}
\definecolor{brown}         {cmyk}{0   , 0.81, 1   , 0.60}
\definecolor{tan}           {cmyk}{0.14, 0.42, 0.56, 0   }
\definecolor{gray}          {cmyk}{0   , 0   , 0   , 0.50}
\definecolor{black}         {cmyk}{0   , 0   , 0   , 1   }
\definecolor{white}         {cmyk}{0   , 0   , 0   , 0   } 

\usepgfplotslibrary{groupplots}

\pgfplotsset{
colormap={myhot}{
rgb(0pt)=(1.000000,1.000000,1.000000);
rgb(1pt)=(1.000000,1.000000,0.996000);
rgb(2pt)=(1.000000,1.000000,0.992000);
rgb(3pt)=(1.000000,1.000000,0.988000);
rgb(4pt)=(1.000000,1.000000,0.984000);
rgb(5pt)=(1.000000,1.000000,0.980000);
rgb(6pt)=(1.000000,1.000000,0.976000);
rgb(7pt)=(1.000000,1.000000,0.972000);
rgb(8pt)=(1.000000,1.000000,0.968000);
rgb(9pt)=(1.000000,1.000000,0.964000);
rgb(10pt)=(1.000000,1.000000,0.960000);
rgb(11pt)=(1.000000,1.000000,0.956000);
rgb(12pt)=(1.000000,1.000000,0.952000);
rgb(13pt)=(1.000000,1.000000,0.948000);
rgb(14pt)=(1.000000,1.000000,0.944000);
rgb(15pt)=(1.000000,1.000000,0.940000);
rgb(16pt)=(1.000000,1.000000,0.936000);
rgb(17pt)=(1.000000,1.000000,0.932000);
rgb(18pt)=(1.000000,1.000000,0.928000);
rgb(19pt)=(1.000000,1.000000,0.924000);
rgb(20pt)=(1.000000,1.000000,0.920000);
rgb(21pt)=(1.000000,1.000000,0.916000);
rgb(22pt)=(1.000000,1.000000,0.912000);
rgb(23pt)=(1.000000,1.000000,0.908000);
rgb(24pt)=(1.000000,1.000000,0.904000);
rgb(25pt)=(1.000000,1.000000,0.900000);
rgb(26pt)=(1.000000,1.000000,0.896000);
rgb(27pt)=(1.000000,1.000000,0.892000);
rgb(28pt)=(1.000000,1.000000,0.888000);
rgb(29pt)=(1.000000,1.000000,0.884000);
rgb(30pt)=(1.000000,1.000000,0.880000);
rgb(31pt)=(1.000000,1.000000,0.876000);
rgb(32pt)=(1.000000,1.000000,0.872000);
rgb(33pt)=(1.000000,1.000000,0.868000);
rgb(34pt)=(1.000000,1.000000,0.864000);
rgb(35pt)=(1.000000,1.000000,0.860000);
rgb(36pt)=(1.000000,1.000000,0.856000);
rgb(37pt)=(1.000000,1.000000,0.852000);
rgb(38pt)=(1.000000,1.000000,0.848000);
rgb(39pt)=(1.000000,1.000000,0.844000);
rgb(40pt)=(1.000000,1.000000,0.840000);
rgb(41pt)=(1.000000,1.000000,0.836000);
rgb(42pt)=(1.000000,1.000000,0.832000);
rgb(43pt)=(1.000000,1.000000,0.828000);
rgb(44pt)=(1.000000,1.000000,0.824000);
rgb(45pt)=(1.000000,1.000000,0.820000);
rgb(46pt)=(1.000000,1.000000,0.816000);
rgb(47pt)=(1.000000,1.000000,0.812000);
rgb(48pt)=(1.000000,1.000000,0.808000);
rgb(49pt)=(1.000000,1.000000,0.804000);
rgb(50pt)=(1.000000,1.000000,0.800000);
rgb(51pt)=(1.000000,1.000000,0.796000);
rgb(52pt)=(1.000000,1.000000,0.792000);
rgb(53pt)=(1.000000,1.000000,0.788000);
rgb(54pt)=(1.000000,1.000000,0.784000);
rgb(55pt)=(1.000000,1.000000,0.780000);
rgb(56pt)=(1.000000,1.000000,0.776000);
rgb(57pt)=(1.000000,1.000000,0.772000);
rgb(58pt)=(1.000000,1.000000,0.768000);
rgb(59pt)=(1.000000,1.000000,0.764000);
rgb(60pt)=(1.000000,1.000000,0.760000);
rgb(61pt)=(1.000000,1.000000,0.756000);
rgb(62pt)=(1.000000,1.000000,0.752000);
rgb(63pt)=(1.000000,1.000000,0.748000);
rgb(64pt)=(1.000000,1.000000,0.744000);
rgb(65pt)=(1.000000,1.000000,0.740000);
rgb(66pt)=(1.000000,1.000000,0.736000);
rgb(67pt)=(1.000000,1.000000,0.732000);
rgb(68pt)=(1.000000,1.000000,0.728000);
rgb(69pt)=(1.000000,1.000000,0.724000);
rgb(70pt)=(1.000000,1.000000,0.720000);
rgb(71pt)=(1.000000,1.000000,0.716000);
rgb(72pt)=(1.000000,1.000000,0.712000);
rgb(73pt)=(1.000000,1.000000,0.708000);
rgb(74pt)=(1.000000,1.000000,0.704000);
rgb(75pt)=(1.000000,1.000000,0.700000);
rgb(76pt)=(1.000000,1.000000,0.696000);
rgb(77pt)=(1.000000,1.000000,0.692000);
rgb(78pt)=(1.000000,1.000000,0.688000);
rgb(79pt)=(1.000000,1.000000,0.684000);
rgb(80pt)=(1.000000,1.000000,0.680000);
rgb(81pt)=(1.000000,1.000000,0.676000);
rgb(82pt)=(1.000000,1.000000,0.672000);
rgb(83pt)=(1.000000,1.000000,0.668000);
rgb(84pt)=(1.000000,1.000000,0.664000);
rgb(85pt)=(1.000000,1.000000,0.660000);
rgb(86pt)=(1.000000,1.000000,0.656000);
rgb(87pt)=(1.000000,1.000000,0.652000);
rgb(88pt)=(1.000000,1.000000,0.648000);
rgb(89pt)=(1.000000,1.000000,0.644000);
rgb(90pt)=(1.000000,1.000000,0.640000);
rgb(91pt)=(1.000000,1.000000,0.636000);
rgb(92pt)=(1.000000,1.000000,0.632000);
rgb(93pt)=(1.000000,1.000000,0.628000);
rgb(94pt)=(1.000000,1.000000,0.624000);
rgb(95pt)=(1.000000,1.000000,0.620000);
rgb(96pt)=(1.000000,1.000000,0.616000);
rgb(97pt)=(1.000000,1.000000,0.612000);
rgb(98pt)=(1.000000,1.000000,0.608000);
rgb(99pt)=(1.000000,1.000000,0.604000);
rgb(100pt)=(1.000000,1.000000,0.600000);
rgb(101pt)=(1.000000,1.000000,0.596000);
rgb(102pt)=(1.000000,1.000000,0.592000);
rgb(103pt)=(1.000000,1.000000,0.588000);
rgb(104pt)=(1.000000,1.000000,0.584000);
rgb(105pt)=(1.000000,1.000000,0.580000);
rgb(106pt)=(1.000000,1.000000,0.576000);
rgb(107pt)=(1.000000,1.000000,0.572000);
rgb(108pt)=(1.000000,1.000000,0.568000);
rgb(109pt)=(1.000000,1.000000,0.564000);
rgb(110pt)=(1.000000,1.000000,0.560000);
rgb(111pt)=(1.000000,1.000000,0.556000);
rgb(112pt)=(1.000000,1.000000,0.552000);
rgb(113pt)=(1.000000,1.000000,0.548000);
rgb(114pt)=(1.000000,1.000000,0.544000);
rgb(115pt)=(1.000000,1.000000,0.540000);
rgb(116pt)=(1.000000,1.000000,0.536000);
rgb(117pt)=(1.000000,1.000000,0.532000);
rgb(118pt)=(1.000000,1.000000,0.528000);
rgb(119pt)=(1.000000,1.000000,0.524000);
rgb(120pt)=(1.000000,1.000000,0.520000);
rgb(121pt)=(1.000000,1.000000,0.516000);
rgb(122pt)=(1.000000,1.000000,0.512000);
rgb(123pt)=(1.000000,1.000000,0.508000);
rgb(124pt)=(1.000000,1.000000,0.504000);
rgb(125pt)=(1.000000,1.000000,0.500000);
rgb(126pt)=(1.000000,1.000000,0.496000);
rgb(127pt)=(1.000000,1.000000,0.492000);
rgb(128pt)=(1.000000,1.000000,0.488000);
rgb(129pt)=(1.000000,1.000000,0.484000);
rgb(130pt)=(1.000000,1.000000,0.480000);
rgb(131pt)=(1.000000,1.000000,0.476000);
rgb(132pt)=(1.000000,1.000000,0.472000);
rgb(133pt)=(1.000000,1.000000,0.468000);
rgb(134pt)=(1.000000,1.000000,0.464000);
rgb(135pt)=(1.000000,1.000000,0.460000);
rgb(136pt)=(1.000000,1.000000,0.456000);
rgb(137pt)=(1.000000,1.000000,0.452000);
rgb(138pt)=(1.000000,1.000000,0.448000);
rgb(139pt)=(1.000000,1.000000,0.444000);
rgb(140pt)=(1.000000,1.000000,0.440000);
rgb(141pt)=(1.000000,1.000000,0.436000);
rgb(142pt)=(1.000000,1.000000,0.432000);
rgb(143pt)=(1.000000,1.000000,0.428000);
rgb(144pt)=(1.000000,1.000000,0.424000);
rgb(145pt)=(1.000000,1.000000,0.420000);
rgb(146pt)=(1.000000,1.000000,0.416000);
rgb(147pt)=(1.000000,1.000000,0.412000);
rgb(148pt)=(1.000000,1.000000,0.408000);
rgb(149pt)=(1.000000,1.000000,0.404000);
rgb(150pt)=(1.000000,1.000000,0.400000);
rgb(151pt)=(1.000000,1.000000,0.396000);
rgb(152pt)=(1.000000,1.000000,0.392000);
rgb(153pt)=(1.000000,1.000000,0.388000);
rgb(154pt)=(1.000000,1.000000,0.384000);
rgb(155pt)=(1.000000,1.000000,0.380000);
rgb(156pt)=(1.000000,1.000000,0.376000);
rgb(157pt)=(1.000000,1.000000,0.372000);
rgb(158pt)=(1.000000,1.000000,0.368000);
rgb(159pt)=(1.000000,1.000000,0.364000);
rgb(160pt)=(1.000000,1.000000,0.360000);
rgb(161pt)=(1.000000,1.000000,0.356000);
rgb(162pt)=(1.000000,1.000000,0.352000);
rgb(163pt)=(1.000000,1.000000,0.348000);
rgb(164pt)=(1.000000,1.000000,0.344000);
rgb(165pt)=(1.000000,1.000000,0.340000);
rgb(166pt)=(1.000000,1.000000,0.336000);
rgb(167pt)=(1.000000,1.000000,0.332000);
rgb(168pt)=(1.000000,1.000000,0.328000);
rgb(169pt)=(1.000000,1.000000,0.324000);
rgb(170pt)=(1.000000,1.000000,0.320000);
rgb(171pt)=(1.000000,1.000000,0.316000);
rgb(172pt)=(1.000000,1.000000,0.312000);
rgb(173pt)=(1.000000,1.000000,0.308000);
rgb(174pt)=(1.000000,1.000000,0.304000);
rgb(175pt)=(1.000000,1.000000,0.300000);
rgb(176pt)=(1.000000,1.000000,0.296000);
rgb(177pt)=(1.000000,1.000000,0.292000);
rgb(178pt)=(1.000000,1.000000,0.288000);
rgb(179pt)=(1.000000,1.000000,0.284000);
rgb(180pt)=(1.000000,1.000000,0.280000);
rgb(181pt)=(1.000000,1.000000,0.276000);
rgb(182pt)=(1.000000,1.000000,0.272000);
rgb(183pt)=(1.000000,1.000000,0.268000);
rgb(184pt)=(1.000000,1.000000,0.264000);
rgb(185pt)=(1.000000,1.000000,0.260000);
rgb(186pt)=(1.000000,1.000000,0.256000);
rgb(187pt)=(1.000000,1.000000,0.252000);
rgb(188pt)=(1.000000,1.000000,0.248000);
rgb(189pt)=(1.000000,1.000000,0.244000);
rgb(190pt)=(1.000000,1.000000,0.240000);
rgb(191pt)=(1.000000,1.000000,0.236000);
rgb(192pt)=(1.000000,1.000000,0.232000);
rgb(193pt)=(1.000000,1.000000,0.228000);
rgb(194pt)=(1.000000,1.000000,0.224000);
rgb(195pt)=(1.000000,1.000000,0.220000);
rgb(196pt)=(1.000000,1.000000,0.216000);
rgb(197pt)=(1.000000,1.000000,0.212000);
rgb(198pt)=(1.000000,1.000000,0.208000);
rgb(199pt)=(1.000000,1.000000,0.204000);
rgb(200pt)=(1.000000,1.000000,0.200000);
rgb(201pt)=(1.000000,1.000000,0.196000);
rgb(202pt)=(1.000000,1.000000,0.192000);
rgb(203pt)=(1.000000,1.000000,0.188000);
rgb(204pt)=(1.000000,1.000000,0.184000);
rgb(205pt)=(1.000000,1.000000,0.180000);
rgb(206pt)=(1.000000,1.000000,0.176000);
rgb(207pt)=(1.000000,1.000000,0.172000);
rgb(208pt)=(1.000000,1.000000,0.168000);
rgb(209pt)=(1.000000,1.000000,0.164000);
rgb(210pt)=(1.000000,1.000000,0.160000);
rgb(211pt)=(1.000000,1.000000,0.156000);
rgb(212pt)=(1.000000,1.000000,0.152000);
rgb(213pt)=(1.000000,1.000000,0.148000);
rgb(214pt)=(1.000000,1.000000,0.144000);
rgb(215pt)=(1.000000,1.000000,0.140000);
rgb(216pt)=(1.000000,1.000000,0.136000);
rgb(217pt)=(1.000000,1.000000,0.132000);
rgb(218pt)=(1.000000,1.000000,0.128000);
rgb(219pt)=(1.000000,1.000000,0.124000);
rgb(220pt)=(1.000000,1.000000,0.120000);
rgb(221pt)=(1.000000,1.000000,0.116000);
rgb(222pt)=(1.000000,1.000000,0.112000);
rgb(223pt)=(1.000000,1.000000,0.108000);
rgb(224pt)=(1.000000,1.000000,0.104000);
rgb(225pt)=(1.000000,1.000000,0.100000);
rgb(226pt)=(1.000000,1.000000,0.096000);
rgb(227pt)=(1.000000,1.000000,0.092000);
rgb(228pt)=(1.000000,1.000000,0.088000);
rgb(229pt)=(1.000000,1.000000,0.084000);
rgb(230pt)=(1.000000,1.000000,0.080000);
rgb(231pt)=(1.000000,1.000000,0.076000);
rgb(232pt)=(1.000000,1.000000,0.072000);
rgb(233pt)=(1.000000,1.000000,0.068000);
rgb(234pt)=(1.000000,1.000000,0.064000);
rgb(235pt)=(1.000000,1.000000,0.060000);
rgb(236pt)=(1.000000,1.000000,0.056000);
rgb(237pt)=(1.000000,1.000000,0.052000);
rgb(238pt)=(1.000000,1.000000,0.048000);
rgb(239pt)=(1.000000,1.000000,0.044000);
rgb(240pt)=(1.000000,1.000000,0.040000);
rgb(241pt)=(1.000000,1.000000,0.036000);
rgb(242pt)=(1.000000,1.000000,0.032000);
rgb(243pt)=(1.000000,1.000000,0.028000);
rgb(244pt)=(1.000000,1.000000,0.024000);
rgb(245pt)=(1.000000,1.000000,0.020000);
rgb(246pt)=(1.000000,1.000000,0.016000);
rgb(247pt)=(1.000000,1.000000,0.012000);
rgb(248pt)=(1.000000,1.000000,0.008000);
rgb(249pt)=(1.000000,1.000000,0.004000);
rgb(250pt)=(1.000000,1.000000,0.000000);
rgb(251pt)=(1.000000,0.997333,0.000000);
rgb(252pt)=(1.000000,0.994667,0.000000);
rgb(253pt)=(1.000000,0.992000,0.000000);
rgb(254pt)=(1.000000,0.989333,0.000000);
rgb(255pt)=(1.000000,0.986667,0.000000);
rgb(256pt)=(1.000000,0.984000,0.000000);
rgb(257pt)=(1.000000,0.981333,0.000000);
rgb(258pt)=(1.000000,0.978667,0.000000);
rgb(259pt)=(1.000000,0.976000,0.000000);
rgb(260pt)=(1.000000,0.973333,0.000000);
rgb(261pt)=(1.000000,0.970667,0.000000);
rgb(262pt)=(1.000000,0.968000,0.000000);
rgb(263pt)=(1.000000,0.965333,0.000000);
rgb(264pt)=(1.000000,0.962667,0.000000);
rgb(265pt)=(1.000000,0.960000,0.000000);
rgb(266pt)=(1.000000,0.957333,0.000000);
rgb(267pt)=(1.000000,0.954667,0.000000);
rgb(268pt)=(1.000000,0.952000,0.000000);
rgb(269pt)=(1.000000,0.949333,0.000000);
rgb(270pt)=(1.000000,0.946667,0.000000);
rgb(271pt)=(1.000000,0.944000,0.000000);
rgb(272pt)=(1.000000,0.941333,0.000000);
rgb(273pt)=(1.000000,0.938667,0.000000);
rgb(274pt)=(1.000000,0.936000,0.000000);
rgb(275pt)=(1.000000,0.933333,0.000000);
rgb(276pt)=(1.000000,0.930667,0.000000);
rgb(277pt)=(1.000000,0.928000,0.000000);
rgb(278pt)=(1.000000,0.925333,0.000000);
rgb(279pt)=(1.000000,0.922667,0.000000);
rgb(280pt)=(1.000000,0.920000,0.000000);
rgb(281pt)=(1.000000,0.917333,0.000000);
rgb(282pt)=(1.000000,0.914667,0.000000);
rgb(283pt)=(1.000000,0.912000,0.000000);
rgb(284pt)=(1.000000,0.909333,0.000000);
rgb(285pt)=(1.000000,0.906667,0.000000);
rgb(286pt)=(1.000000,0.904000,0.000000);
rgb(287pt)=(1.000000,0.901333,0.000000);
rgb(288pt)=(1.000000,0.898667,0.000000);
rgb(289pt)=(1.000000,0.896000,0.000000);
rgb(290pt)=(1.000000,0.893333,0.000000);
rgb(291pt)=(1.000000,0.890667,0.000000);
rgb(292pt)=(1.000000,0.888000,0.000000);
rgb(293pt)=(1.000000,0.885333,0.000000);
rgb(294pt)=(1.000000,0.882667,0.000000);
rgb(295pt)=(1.000000,0.880000,0.000000);
rgb(296pt)=(1.000000,0.877333,0.000000);
rgb(297pt)=(1.000000,0.874667,0.000000);
rgb(298pt)=(1.000000,0.872000,0.000000);
rgb(299pt)=(1.000000,0.869333,0.000000);
rgb(300pt)=(1.000000,0.866667,0.000000);
rgb(301pt)=(1.000000,0.864000,0.000000);
rgb(302pt)=(1.000000,0.861333,0.000000);
rgb(303pt)=(1.000000,0.858667,0.000000);
rgb(304pt)=(1.000000,0.856000,0.000000);
rgb(305pt)=(1.000000,0.853333,0.000000);
rgb(306pt)=(1.000000,0.850667,0.000000);
rgb(307pt)=(1.000000,0.848000,0.000000);
rgb(308pt)=(1.000000,0.845333,0.000000);
rgb(309pt)=(1.000000,0.842667,0.000000);
rgb(310pt)=(1.000000,0.840000,0.000000);
rgb(311pt)=(1.000000,0.837333,0.000000);
rgb(312pt)=(1.000000,0.834667,0.000000);
rgb(313pt)=(1.000000,0.832000,0.000000);
rgb(314pt)=(1.000000,0.829333,0.000000);
rgb(315pt)=(1.000000,0.826667,0.000000);
rgb(316pt)=(1.000000,0.824000,0.000000);
rgb(317pt)=(1.000000,0.821333,0.000000);
rgb(318pt)=(1.000000,0.818667,0.000000);
rgb(319pt)=(1.000000,0.816000,0.000000);
rgb(320pt)=(1.000000,0.813333,0.000000);
rgb(321pt)=(1.000000,0.810667,0.000000);
rgb(322pt)=(1.000000,0.808000,0.000000);
rgb(323pt)=(1.000000,0.805333,0.000000);
rgb(324pt)=(1.000000,0.802667,0.000000);
rgb(325pt)=(1.000000,0.800000,0.000000);
rgb(326pt)=(1.000000,0.797333,0.000000);
rgb(327pt)=(1.000000,0.794667,0.000000);
rgb(328pt)=(1.000000,0.792000,0.000000);
rgb(329pt)=(1.000000,0.789333,0.000000);
rgb(330pt)=(1.000000,0.786667,0.000000);
rgb(331pt)=(1.000000,0.784000,0.000000);
rgb(332pt)=(1.000000,0.781333,0.000000);
rgb(333pt)=(1.000000,0.778667,0.000000);
rgb(334pt)=(1.000000,0.776000,0.000000);
rgb(335pt)=(1.000000,0.773333,0.000000);
rgb(336pt)=(1.000000,0.770667,0.000000);
rgb(337pt)=(1.000000,0.768000,0.000000);
rgb(338pt)=(1.000000,0.765333,0.000000);
rgb(339pt)=(1.000000,0.762667,0.000000);
rgb(340pt)=(1.000000,0.760000,0.000000);
rgb(341pt)=(1.000000,0.757333,0.000000);
rgb(342pt)=(1.000000,0.754667,0.000000);
rgb(343pt)=(1.000000,0.752000,0.000000);
rgb(344pt)=(1.000000,0.749333,0.000000);
rgb(345pt)=(1.000000,0.746667,0.000000);
rgb(346pt)=(1.000000,0.744000,0.000000);
rgb(347pt)=(1.000000,0.741333,0.000000);
rgb(348pt)=(1.000000,0.738667,0.000000);
rgb(349pt)=(1.000000,0.736000,0.000000);
rgb(350pt)=(1.000000,0.733333,0.000000);
rgb(351pt)=(1.000000,0.730667,0.000000);
rgb(352pt)=(1.000000,0.728000,0.000000);
rgb(353pt)=(1.000000,0.725333,0.000000);
rgb(354pt)=(1.000000,0.722667,0.000000);
rgb(355pt)=(1.000000,0.720000,0.000000);
rgb(356pt)=(1.000000,0.717333,0.000000);
rgb(357pt)=(1.000000,0.714667,0.000000);
rgb(358pt)=(1.000000,0.712000,0.000000);
rgb(359pt)=(1.000000,0.709333,0.000000);
rgb(360pt)=(1.000000,0.706667,0.000000);
rgb(361pt)=(1.000000,0.704000,0.000000);
rgb(362pt)=(1.000000,0.701333,0.000000);
rgb(363pt)=(1.000000,0.698667,0.000000);
rgb(364pt)=(1.000000,0.696000,0.000000);
rgb(365pt)=(1.000000,0.693333,0.000000);
rgb(366pt)=(1.000000,0.690667,0.000000);
rgb(367pt)=(1.000000,0.688000,0.000000);
rgb(368pt)=(1.000000,0.685333,0.000000);
rgb(369pt)=(1.000000,0.682667,0.000000);
rgb(370pt)=(1.000000,0.680000,0.000000);
rgb(371pt)=(1.000000,0.677333,0.000000);
rgb(372pt)=(1.000000,0.674667,0.000000);
rgb(373pt)=(1.000000,0.672000,0.000000);
rgb(374pt)=(1.000000,0.669333,0.000000);
rgb(375pt)=(1.000000,0.666667,0.000000);
rgb(376pt)=(1.000000,0.664000,0.000000);
rgb(377pt)=(1.000000,0.661333,0.000000);
rgb(378pt)=(1.000000,0.658667,0.000000);
rgb(379pt)=(1.000000,0.656000,0.000000);
rgb(380pt)=(1.000000,0.653333,0.000000);
rgb(381pt)=(1.000000,0.650667,0.000000);
rgb(382pt)=(1.000000,0.648000,0.000000);
rgb(383pt)=(1.000000,0.645333,0.000000);
rgb(384pt)=(1.000000,0.642667,0.000000);
rgb(385pt)=(1.000000,0.640000,0.000000);
rgb(386pt)=(1.000000,0.637333,0.000000);
rgb(387pt)=(1.000000,0.634667,0.000000);
rgb(388pt)=(1.000000,0.632000,0.000000);
rgb(389pt)=(1.000000,0.629333,0.000000);
rgb(390pt)=(1.000000,0.626667,0.000000);
rgb(391pt)=(1.000000,0.624000,0.000000);
rgb(392pt)=(1.000000,0.621333,0.000000);
rgb(393pt)=(1.000000,0.618667,0.000000);
rgb(394pt)=(1.000000,0.616000,0.000000);
rgb(395pt)=(1.000000,0.613333,0.000000);
rgb(396pt)=(1.000000,0.610667,0.000000);
rgb(397pt)=(1.000000,0.608000,0.000000);
rgb(398pt)=(1.000000,0.605333,0.000000);
rgb(399pt)=(1.000000,0.602667,0.000000);
rgb(400pt)=(1.000000,0.600000,0.000000);
rgb(401pt)=(1.000000,0.597333,0.000000);
rgb(402pt)=(1.000000,0.594667,0.000000);
rgb(403pt)=(1.000000,0.592000,0.000000);
rgb(404pt)=(1.000000,0.589333,0.000000);
rgb(405pt)=(1.000000,0.586667,0.000000);
rgb(406pt)=(1.000000,0.584000,0.000000);
rgb(407pt)=(1.000000,0.581333,0.000000);
rgb(408pt)=(1.000000,0.578667,0.000000);
rgb(409pt)=(1.000000,0.576000,0.000000);
rgb(410pt)=(1.000000,0.573333,0.000000);
rgb(411pt)=(1.000000,0.570667,0.000000);
rgb(412pt)=(1.000000,0.568000,0.000000);
rgb(413pt)=(1.000000,0.565333,0.000000);
rgb(414pt)=(1.000000,0.562667,0.000000);
rgb(415pt)=(1.000000,0.560000,0.000000);
rgb(416pt)=(1.000000,0.557333,0.000000);
rgb(417pt)=(1.000000,0.554667,0.000000);
rgb(418pt)=(1.000000,0.552000,0.000000);
rgb(419pt)=(1.000000,0.549333,0.000000);
rgb(420pt)=(1.000000,0.546667,0.000000);
rgb(421pt)=(1.000000,0.544000,0.000000);
rgb(422pt)=(1.000000,0.541333,0.000000);
rgb(423pt)=(1.000000,0.538667,0.000000);
rgb(424pt)=(1.000000,0.536000,0.000000);
rgb(425pt)=(1.000000,0.533333,0.000000);
rgb(426pt)=(1.000000,0.530667,0.000000);
rgb(427pt)=(1.000000,0.528000,0.000000);
rgb(428pt)=(1.000000,0.525333,0.000000);
rgb(429pt)=(1.000000,0.522667,0.000000);
rgb(430pt)=(1.000000,0.520000,0.000000);
rgb(431pt)=(1.000000,0.517333,0.000000);
rgb(432pt)=(1.000000,0.514667,0.000000);
rgb(433pt)=(1.000000,0.512000,0.000000);
rgb(434pt)=(1.000000,0.509333,0.000000);
rgb(435pt)=(1.000000,0.506667,0.000000);
rgb(436pt)=(1.000000,0.504000,0.000000);
rgb(437pt)=(1.000000,0.501333,0.000000);
rgb(438pt)=(1.000000,0.498667,0.000000);
rgb(439pt)=(1.000000,0.496000,0.000000);
rgb(440pt)=(1.000000,0.493333,0.000000);
rgb(441pt)=(1.000000,0.490667,0.000000);
rgb(442pt)=(1.000000,0.488000,0.000000);
rgb(443pt)=(1.000000,0.485333,0.000000);
rgb(444pt)=(1.000000,0.482667,0.000000);
rgb(445pt)=(1.000000,0.480000,0.000000);
rgb(446pt)=(1.000000,0.477333,0.000000);
rgb(447pt)=(1.000000,0.474667,0.000000);
rgb(448pt)=(1.000000,0.472000,0.000000);
rgb(449pt)=(1.000000,0.469333,0.000000);
rgb(450pt)=(1.000000,0.466667,0.000000);
rgb(451pt)=(1.000000,0.464000,0.000000);
rgb(452pt)=(1.000000,0.461333,0.000000);
rgb(453pt)=(1.000000,0.458667,0.000000);
rgb(454pt)=(1.000000,0.456000,0.000000);
rgb(455pt)=(1.000000,0.453333,0.000000);
rgb(456pt)=(1.000000,0.450667,0.000000);
rgb(457pt)=(1.000000,0.448000,0.000000);
rgb(458pt)=(1.000000,0.445333,0.000000);
rgb(459pt)=(1.000000,0.442667,0.000000);
rgb(460pt)=(1.000000,0.440000,0.000000);
rgb(461pt)=(1.000000,0.437333,0.000000);
rgb(462pt)=(1.000000,0.434667,0.000000);
rgb(463pt)=(1.000000,0.432000,0.000000);
rgb(464pt)=(1.000000,0.429333,0.000000);
rgb(465pt)=(1.000000,0.426667,0.000000);
rgb(466pt)=(1.000000,0.424000,0.000000);
rgb(467pt)=(1.000000,0.421333,0.000000);
rgb(468pt)=(1.000000,0.418667,0.000000);
rgb(469pt)=(1.000000,0.416000,0.000000);
rgb(470pt)=(1.000000,0.413333,0.000000);
rgb(471pt)=(1.000000,0.410667,0.000000);
rgb(472pt)=(1.000000,0.408000,0.000000);
rgb(473pt)=(1.000000,0.405333,0.000000);
rgb(474pt)=(1.000000,0.402667,0.000000);
rgb(475pt)=(1.000000,0.400000,0.000000);
rgb(476pt)=(1.000000,0.397333,0.000000);
rgb(477pt)=(1.000000,0.394667,0.000000);
rgb(478pt)=(1.000000,0.392000,0.000000);
rgb(479pt)=(1.000000,0.389333,0.000000);
rgb(480pt)=(1.000000,0.386667,0.000000);
rgb(481pt)=(1.000000,0.384000,0.000000);
rgb(482pt)=(1.000000,0.381333,0.000000);
rgb(483pt)=(1.000000,0.378667,0.000000);
rgb(484pt)=(1.000000,0.376000,0.000000);
rgb(485pt)=(1.000000,0.373333,0.000000);
rgb(486pt)=(1.000000,0.370667,0.000000);
rgb(487pt)=(1.000000,0.368000,0.000000);
rgb(488pt)=(1.000000,0.365333,0.000000);
rgb(489pt)=(1.000000,0.362667,0.000000);
rgb(490pt)=(1.000000,0.360000,0.000000);
rgb(491pt)=(1.000000,0.357333,0.000000);
rgb(492pt)=(1.000000,0.354667,0.000000);
rgb(493pt)=(1.000000,0.352000,0.000000);
rgb(494pt)=(1.000000,0.349333,0.000000);
rgb(495pt)=(1.000000,0.346667,0.000000);
rgb(496pt)=(1.000000,0.344000,0.000000);
rgb(497pt)=(1.000000,0.341333,0.000000);
rgb(498pt)=(1.000000,0.338667,0.000000);
rgb(499pt)=(1.000000,0.336000,0.000000);
rgb(500pt)=(1.000000,0.333333,0.000000);
rgb(501pt)=(1.000000,0.330667,0.000000);
rgb(502pt)=(1.000000,0.328000,0.000000);
rgb(503pt)=(1.000000,0.325333,0.000000);
rgb(504pt)=(1.000000,0.322667,0.000000);
rgb(505pt)=(1.000000,0.320000,0.000000);
rgb(506pt)=(1.000000,0.317333,0.000000);
rgb(507pt)=(1.000000,0.314667,0.000000);
rgb(508pt)=(1.000000,0.312000,0.000000);
rgb(509pt)=(1.000000,0.309333,0.000000);
rgb(510pt)=(1.000000,0.306667,0.000000);
rgb(511pt)=(1.000000,0.304000,0.000000);
rgb(512pt)=(1.000000,0.301333,0.000000);
rgb(513pt)=(1.000000,0.298667,0.000000);
rgb(514pt)=(1.000000,0.296000,0.000000);
rgb(515pt)=(1.000000,0.293333,0.000000);
rgb(516pt)=(1.000000,0.290667,0.000000);
rgb(517pt)=(1.000000,0.288000,0.000000);
rgb(518pt)=(1.000000,0.285333,0.000000);
rgb(519pt)=(1.000000,0.282667,0.000000);
rgb(520pt)=(1.000000,0.280000,0.000000);
rgb(521pt)=(1.000000,0.277333,0.000000);
rgb(522pt)=(1.000000,0.274667,0.000000);
rgb(523pt)=(1.000000,0.272000,0.000000);
rgb(524pt)=(1.000000,0.269333,0.000000);
rgb(525pt)=(1.000000,0.266667,0.000000);
rgb(526pt)=(1.000000,0.264000,0.000000);
rgb(527pt)=(1.000000,0.261333,0.000000);
rgb(528pt)=(1.000000,0.258667,0.000000);
rgb(529pt)=(1.000000,0.256000,0.000000);
rgb(530pt)=(1.000000,0.253333,0.000000);
rgb(531pt)=(1.000000,0.250667,0.000000);
rgb(532pt)=(1.000000,0.248000,0.000000);
rgb(533pt)=(1.000000,0.245333,0.000000);
rgb(534pt)=(1.000000,0.242667,0.000000);
rgb(535pt)=(1.000000,0.240000,0.000000);
rgb(536pt)=(1.000000,0.237333,0.000000);
rgb(537pt)=(1.000000,0.234667,0.000000);
rgb(538pt)=(1.000000,0.232000,0.000000);
rgb(539pt)=(1.000000,0.229333,0.000000);
rgb(540pt)=(1.000000,0.226667,0.000000);
rgb(541pt)=(1.000000,0.224000,0.000000);
rgb(542pt)=(1.000000,0.221333,0.000000);
rgb(543pt)=(1.000000,0.218667,0.000000);
rgb(544pt)=(1.000000,0.216000,0.000000);
rgb(545pt)=(1.000000,0.213333,0.000000);
rgb(546pt)=(1.000000,0.210667,0.000000);
rgb(547pt)=(1.000000,0.208000,0.000000);
rgb(548pt)=(1.000000,0.205333,0.000000);
rgb(549pt)=(1.000000,0.202667,0.000000);
rgb(550pt)=(1.000000,0.200000,0.000000);
rgb(551pt)=(1.000000,0.197333,0.000000);
rgb(552pt)=(1.000000,0.194667,0.000000);
rgb(553pt)=(1.000000,0.192000,0.000000);
rgb(554pt)=(1.000000,0.189333,0.000000);
rgb(555pt)=(1.000000,0.186667,0.000000);
rgb(556pt)=(1.000000,0.184000,0.000000);
rgb(557pt)=(1.000000,0.181333,0.000000);
rgb(558pt)=(1.000000,0.178667,0.000000);
rgb(559pt)=(1.000000,0.176000,0.000000);
rgb(560pt)=(1.000000,0.173333,0.000000);
rgb(561pt)=(1.000000,0.170667,0.000000);
rgb(562pt)=(1.000000,0.168000,0.000000);
rgb(563pt)=(1.000000,0.165333,0.000000);
rgb(564pt)=(1.000000,0.162667,0.000000);
rgb(565pt)=(1.000000,0.160000,0.000000);
rgb(566pt)=(1.000000,0.157333,0.000000);
rgb(567pt)=(1.000000,0.154667,0.000000);
rgb(568pt)=(1.000000,0.152000,0.000000);
rgb(569pt)=(1.000000,0.149333,0.000000);
rgb(570pt)=(1.000000,0.146667,0.000000);
rgb(571pt)=(1.000000,0.144000,0.000000);
rgb(572pt)=(1.000000,0.141333,0.000000);
rgb(573pt)=(1.000000,0.138667,0.000000);
rgb(574pt)=(1.000000,0.136000,0.000000);
rgb(575pt)=(1.000000,0.133333,0.000000);
rgb(576pt)=(1.000000,0.130667,0.000000);
rgb(577pt)=(1.000000,0.128000,0.000000);
rgb(578pt)=(1.000000,0.125333,0.000000);
rgb(579pt)=(1.000000,0.122667,0.000000);
rgb(580pt)=(1.000000,0.120000,0.000000);
rgb(581pt)=(1.000000,0.117333,0.000000);
rgb(582pt)=(1.000000,0.114667,0.000000);
rgb(583pt)=(1.000000,0.112000,0.000000);
rgb(584pt)=(1.000000,0.109333,0.000000);
rgb(585pt)=(1.000000,0.106667,0.000000);
rgb(586pt)=(1.000000,0.104000,0.000000);
rgb(587pt)=(1.000000,0.101333,0.000000);
rgb(588pt)=(1.000000,0.098667,0.000000);
rgb(589pt)=(1.000000,0.096000,0.000000);
rgb(590pt)=(1.000000,0.093333,0.000000);
rgb(591pt)=(1.000000,0.090667,0.000000);
rgb(592pt)=(1.000000,0.088000,0.000000);
rgb(593pt)=(1.000000,0.085333,0.000000);
rgb(594pt)=(1.000000,0.082667,0.000000);
rgb(595pt)=(1.000000,0.080000,0.000000);
rgb(596pt)=(1.000000,0.077333,0.000000);
rgb(597pt)=(1.000000,0.074667,0.000000);
rgb(598pt)=(1.000000,0.072000,0.000000);
rgb(599pt)=(1.000000,0.069333,0.000000);
rgb(600pt)=(1.000000,0.066667,0.000000);
rgb(601pt)=(1.000000,0.064000,0.000000);
rgb(602pt)=(1.000000,0.061333,0.000000);
rgb(603pt)=(1.000000,0.058667,0.000000);
rgb(604pt)=(1.000000,0.056000,0.000000);
rgb(605pt)=(1.000000,0.053333,0.000000);
rgb(606pt)=(1.000000,0.050667,0.000000);
rgb(607pt)=(1.000000,0.048000,0.000000);
rgb(608pt)=(1.000000,0.045333,0.000000);
rgb(609pt)=(1.000000,0.042667,0.000000);
rgb(610pt)=(1.000000,0.040000,0.000000);
rgb(611pt)=(1.000000,0.037333,0.000000);
rgb(612pt)=(1.000000,0.034667,0.000000);
rgb(613pt)=(1.000000,0.032000,0.000000);
rgb(614pt)=(1.000000,0.029333,0.000000);
rgb(615pt)=(1.000000,0.026667,0.000000);
rgb(616pt)=(1.000000,0.024000,0.000000);
rgb(617pt)=(1.000000,0.021333,0.000000);
rgb(618pt)=(1.000000,0.018667,0.000000);
rgb(619pt)=(1.000000,0.016000,0.000000);
rgb(620pt)=(1.000000,0.013333,0.000000);
rgb(621pt)=(1.000000,0.010667,0.000000);
rgb(622pt)=(1.000000,0.008000,0.000000);
rgb(623pt)=(1.000000,0.005333,0.000000);
rgb(624pt)=(1.000000,0.002667,0.000000);
rgb(625pt)=(1.000000,0.000000,0.000000);
rgb(626pt)=(0.997333,0.000000,0.000000);
rgb(627pt)=(0.994667,0.000000,0.000000);
rgb(628pt)=(0.992000,0.000000,0.000000);
rgb(629pt)=(0.989333,0.000000,0.000000);
rgb(630pt)=(0.986667,0.000000,0.000000);
rgb(631pt)=(0.984000,0.000000,0.000000);
rgb(632pt)=(0.981333,0.000000,0.000000);
rgb(633pt)=(0.978667,0.000000,0.000000);
rgb(634pt)=(0.976000,0.000000,0.000000);
rgb(635pt)=(0.973333,0.000000,0.000000);
rgb(636pt)=(0.970667,0.000000,0.000000);
rgb(637pt)=(0.968000,0.000000,0.000000);
rgb(638pt)=(0.965333,0.000000,0.000000);
rgb(639pt)=(0.962667,0.000000,0.000000);
rgb(640pt)=(0.960000,0.000000,0.000000);
rgb(641pt)=(0.957333,0.000000,0.000000);
rgb(642pt)=(0.954667,0.000000,0.000000);
rgb(643pt)=(0.952000,0.000000,0.000000);
rgb(644pt)=(0.949333,0.000000,0.000000);
rgb(645pt)=(0.946667,0.000000,0.000000);
rgb(646pt)=(0.944000,0.000000,0.000000);
rgb(647pt)=(0.941333,0.000000,0.000000);
rgb(648pt)=(0.938667,0.000000,0.000000);
rgb(649pt)=(0.936000,0.000000,0.000000);
rgb(650pt)=(0.933333,0.000000,0.000000);
rgb(651pt)=(0.930667,0.000000,0.000000);
rgb(652pt)=(0.928000,0.000000,0.000000);
rgb(653pt)=(0.925333,0.000000,0.000000);
rgb(654pt)=(0.922667,0.000000,0.000000);
rgb(655pt)=(0.920000,0.000000,0.000000);
rgb(656pt)=(0.917333,0.000000,0.000000);
rgb(657pt)=(0.914667,0.000000,0.000000);
rgb(658pt)=(0.912000,0.000000,0.000000);
rgb(659pt)=(0.909333,0.000000,0.000000);
rgb(660pt)=(0.906667,0.000000,0.000000);
rgb(661pt)=(0.904000,0.000000,0.000000);
rgb(662pt)=(0.901333,0.000000,0.000000);
rgb(663pt)=(0.898667,0.000000,0.000000);
rgb(664pt)=(0.896000,0.000000,0.000000);
rgb(665pt)=(0.893333,0.000000,0.000000);
rgb(666pt)=(0.890667,0.000000,0.000000);
rgb(667pt)=(0.888000,0.000000,0.000000);
rgb(668pt)=(0.885333,0.000000,0.000000);
rgb(669pt)=(0.882667,0.000000,0.000000);
rgb(670pt)=(0.880000,0.000000,0.000000);
rgb(671pt)=(0.877333,0.000000,0.000000);
rgb(672pt)=(0.874667,0.000000,0.000000);
rgb(673pt)=(0.872000,0.000000,0.000000);
rgb(674pt)=(0.869333,0.000000,0.000000);
rgb(675pt)=(0.866667,0.000000,0.000000);
rgb(676pt)=(0.864000,0.000000,0.000000);
rgb(677pt)=(0.861333,0.000000,0.000000);
rgb(678pt)=(0.858667,0.000000,0.000000);
rgb(679pt)=(0.856000,0.000000,0.000000);
rgb(680pt)=(0.853333,0.000000,0.000000);
rgb(681pt)=(0.850667,0.000000,0.000000);
rgb(682pt)=(0.848000,0.000000,0.000000);
rgb(683pt)=(0.845333,0.000000,0.000000);
rgb(684pt)=(0.842667,0.000000,0.000000);
rgb(685pt)=(0.840000,0.000000,0.000000);
rgb(686pt)=(0.837333,0.000000,0.000000);
rgb(687pt)=(0.834667,0.000000,0.000000);
rgb(688pt)=(0.832000,0.000000,0.000000);
rgb(689pt)=(0.829333,0.000000,0.000000);
rgb(690pt)=(0.826667,0.000000,0.000000);
rgb(691pt)=(0.824000,0.000000,0.000000);
rgb(692pt)=(0.821333,0.000000,0.000000);
rgb(693pt)=(0.818667,0.000000,0.000000);
rgb(694pt)=(0.816000,0.000000,0.000000);
rgb(695pt)=(0.813333,0.000000,0.000000);
rgb(696pt)=(0.810667,0.000000,0.000000);
rgb(697pt)=(0.808000,0.000000,0.000000);
rgb(698pt)=(0.805333,0.000000,0.000000);
rgb(699pt)=(0.802667,0.000000,0.000000);
rgb(700pt)=(0.800000,0.000000,0.000000);
rgb(701pt)=(0.797333,0.000000,0.000000);
rgb(702pt)=(0.794667,0.000000,0.000000);
rgb(703pt)=(0.792000,0.000000,0.000000);
rgb(704pt)=(0.789333,0.000000,0.000000);
rgb(705pt)=(0.786667,0.000000,0.000000);
rgb(706pt)=(0.784000,0.000000,0.000000);
rgb(707pt)=(0.781333,0.000000,0.000000);
rgb(708pt)=(0.778667,0.000000,0.000000);
rgb(709pt)=(0.776000,0.000000,0.000000);
rgb(710pt)=(0.773333,0.000000,0.000000);
rgb(711pt)=(0.770667,0.000000,0.000000);
rgb(712pt)=(0.768000,0.000000,0.000000);
rgb(713pt)=(0.765333,0.000000,0.000000);
rgb(714pt)=(0.762667,0.000000,0.000000);
rgb(715pt)=(0.760000,0.000000,0.000000);
rgb(716pt)=(0.757333,0.000000,0.000000);
rgb(717pt)=(0.754667,0.000000,0.000000);
rgb(718pt)=(0.752000,0.000000,0.000000);
rgb(719pt)=(0.749333,0.000000,0.000000);
rgb(720pt)=(0.746667,0.000000,0.000000);
rgb(721pt)=(0.744000,0.000000,0.000000);
rgb(722pt)=(0.741333,0.000000,0.000000);
rgb(723pt)=(0.738667,0.000000,0.000000);
rgb(724pt)=(0.736000,0.000000,0.000000);
rgb(725pt)=(0.733333,0.000000,0.000000);
rgb(726pt)=(0.730667,0.000000,0.000000);
rgb(727pt)=(0.728000,0.000000,0.000000);
rgb(728pt)=(0.725333,0.000000,0.000000);
rgb(729pt)=(0.722667,0.000000,0.000000);
rgb(730pt)=(0.720000,0.000000,0.000000);
rgb(731pt)=(0.717333,0.000000,0.000000);
rgb(732pt)=(0.714667,0.000000,0.000000);
rgb(733pt)=(0.712000,0.000000,0.000000);
rgb(734pt)=(0.709333,0.000000,0.000000);
rgb(735pt)=(0.706667,0.000000,0.000000);
rgb(736pt)=(0.704000,0.000000,0.000000);
rgb(737pt)=(0.701333,0.000000,0.000000);
rgb(738pt)=(0.698667,0.000000,0.000000);
rgb(739pt)=(0.696000,0.000000,0.000000);
rgb(740pt)=(0.693333,0.000000,0.000000);
rgb(741pt)=(0.690667,0.000000,0.000000);
rgb(742pt)=(0.688000,0.000000,0.000000);
rgb(743pt)=(0.685333,0.000000,0.000000);
rgb(744pt)=(0.682667,0.000000,0.000000);
rgb(745pt)=(0.680000,0.000000,0.000000);
rgb(746pt)=(0.677333,0.000000,0.000000);
rgb(747pt)=(0.674667,0.000000,0.000000);
rgb(748pt)=(0.672000,0.000000,0.000000);
rgb(749pt)=(0.669333,0.000000,0.000000);
rgb(750pt)=(0.666667,0.000000,0.000000);
rgb(751pt)=(0.664000,0.000000,0.000000);
rgb(752pt)=(0.661333,0.000000,0.000000);
rgb(753pt)=(0.658667,0.000000,0.000000);
rgb(754pt)=(0.656000,0.000000,0.000000);
rgb(755pt)=(0.653333,0.000000,0.000000);
rgb(756pt)=(0.650667,0.000000,0.000000);
rgb(757pt)=(0.648000,0.000000,0.000000);
rgb(758pt)=(0.645333,0.000000,0.000000);
rgb(759pt)=(0.642667,0.000000,0.000000);
rgb(760pt)=(0.640000,0.000000,0.000000);
rgb(761pt)=(0.637333,0.000000,0.000000);
rgb(762pt)=(0.634667,0.000000,0.000000);
rgb(763pt)=(0.632000,0.000000,0.000000);
rgb(764pt)=(0.629333,0.000000,0.000000);
rgb(765pt)=(0.626667,0.000000,0.000000);
rgb(766pt)=(0.624000,0.000000,0.000000);
rgb(767pt)=(0.621333,0.000000,0.000000);
rgb(768pt)=(0.618667,0.000000,0.000000);
rgb(769pt)=(0.616000,0.000000,0.000000);
rgb(770pt)=(0.613333,0.000000,0.000000);
rgb(771pt)=(0.610667,0.000000,0.000000);
rgb(772pt)=(0.608000,0.000000,0.000000);
rgb(773pt)=(0.605333,0.000000,0.000000);
rgb(774pt)=(0.602667,0.000000,0.000000);
rgb(775pt)=(0.600000,0.000000,0.000000);
rgb(776pt)=(0.597333,0.000000,0.000000);
rgb(777pt)=(0.594667,0.000000,0.000000);
rgb(778pt)=(0.592000,0.000000,0.000000);
rgb(779pt)=(0.589333,0.000000,0.000000);
rgb(780pt)=(0.586667,0.000000,0.000000);
rgb(781pt)=(0.584000,0.000000,0.000000);
rgb(782pt)=(0.581333,0.000000,0.000000);
rgb(783pt)=(0.578667,0.000000,0.000000);
rgb(784pt)=(0.576000,0.000000,0.000000);
rgb(785pt)=(0.573333,0.000000,0.000000);
rgb(786pt)=(0.570667,0.000000,0.000000);
rgb(787pt)=(0.568000,0.000000,0.000000);
rgb(788pt)=(0.565333,0.000000,0.000000);
rgb(789pt)=(0.562667,0.000000,0.000000);
rgb(790pt)=(0.560000,0.000000,0.000000);
rgb(791pt)=(0.557333,0.000000,0.000000);
rgb(792pt)=(0.554667,0.000000,0.000000);
rgb(793pt)=(0.552000,0.000000,0.000000);
rgb(794pt)=(0.549333,0.000000,0.000000);
rgb(795pt)=(0.546667,0.000000,0.000000);
rgb(796pt)=(0.544000,0.000000,0.000000);
rgb(797pt)=(0.541333,0.000000,0.000000);
rgb(798pt)=(0.538667,0.000000,0.000000);
rgb(799pt)=(0.536000,0.000000,0.000000);
rgb(800pt)=(0.533333,0.000000,0.000000);
rgb(801pt)=(0.530667,0.000000,0.000000);
rgb(802pt)=(0.528000,0.000000,0.000000);
rgb(803pt)=(0.525333,0.000000,0.000000);
rgb(804pt)=(0.522667,0.000000,0.000000);
rgb(805pt)=(0.520000,0.000000,0.000000);
rgb(806pt)=(0.517333,0.000000,0.000000);
rgb(807pt)=(0.514667,0.000000,0.000000);
rgb(808pt)=(0.512000,0.000000,0.000000);
rgb(809pt)=(0.509333,0.000000,0.000000);
rgb(810pt)=(0.506667,0.000000,0.000000);
rgb(811pt)=(0.504000,0.000000,0.000000);
rgb(812pt)=(0.501333,0.000000,0.000000);
rgb(813pt)=(0.498667,0.000000,0.000000);
rgb(814pt)=(0.496000,0.000000,0.000000);
rgb(815pt)=(0.493333,0.000000,0.000000);
rgb(816pt)=(0.490667,0.000000,0.000000);
rgb(817pt)=(0.488000,0.000000,0.000000);
rgb(818pt)=(0.485333,0.000000,0.000000);
rgb(819pt)=(0.482667,0.000000,0.000000);
rgb(820pt)=(0.480000,0.000000,0.000000);
rgb(821pt)=(0.477333,0.000000,0.000000);
rgb(822pt)=(0.474667,0.000000,0.000000);
rgb(823pt)=(0.472000,0.000000,0.000000);
rgb(824pt)=(0.469333,0.000000,0.000000);
rgb(825pt)=(0.466667,0.000000,0.000000);
rgb(826pt)=(0.464000,0.000000,0.000000);
rgb(827pt)=(0.461333,0.000000,0.000000);
rgb(828pt)=(0.458667,0.000000,0.000000);
rgb(829pt)=(0.456000,0.000000,0.000000);
rgb(830pt)=(0.453333,0.000000,0.000000);
rgb(831pt)=(0.450667,0.000000,0.000000);
rgb(832pt)=(0.448000,0.000000,0.000000);
rgb(833pt)=(0.445333,0.000000,0.000000);
rgb(834pt)=(0.442667,0.000000,0.000000);
rgb(835pt)=(0.440000,0.000000,0.000000);
rgb(836pt)=(0.437333,0.000000,0.000000);
rgb(837pt)=(0.434667,0.000000,0.000000);
rgb(838pt)=(0.432000,0.000000,0.000000);
rgb(839pt)=(0.429333,0.000000,0.000000);
rgb(840pt)=(0.426667,0.000000,0.000000);
rgb(841pt)=(0.424000,0.000000,0.000000);
rgb(842pt)=(0.421333,0.000000,0.000000);
rgb(843pt)=(0.418667,0.000000,0.000000);
rgb(844pt)=(0.416000,0.000000,0.000000);
rgb(845pt)=(0.413333,0.000000,0.000000);
rgb(846pt)=(0.410667,0.000000,0.000000);
rgb(847pt)=(0.408000,0.000000,0.000000);
rgb(848pt)=(0.405333,0.000000,0.000000);
rgb(849pt)=(0.402667,0.000000,0.000000);
rgb(850pt)=(0.400000,0.000000,0.000000);
rgb(851pt)=(0.397333,0.000000,0.000000);
rgb(852pt)=(0.394667,0.000000,0.000000);
rgb(853pt)=(0.392000,0.000000,0.000000);
rgb(854pt)=(0.389333,0.000000,0.000000);
rgb(855pt)=(0.386667,0.000000,0.000000);
rgb(856pt)=(0.384000,0.000000,0.000000);
rgb(857pt)=(0.381333,0.000000,0.000000);
rgb(858pt)=(0.378667,0.000000,0.000000);
rgb(859pt)=(0.376000,0.000000,0.000000);
rgb(860pt)=(0.373333,0.000000,0.000000);
rgb(861pt)=(0.370667,0.000000,0.000000);
rgb(862pt)=(0.368000,0.000000,0.000000);
rgb(863pt)=(0.365333,0.000000,0.000000);
rgb(864pt)=(0.362667,0.000000,0.000000);
rgb(865pt)=(0.360000,0.000000,0.000000);
rgb(866pt)=(0.357333,0.000000,0.000000);
rgb(867pt)=(0.354667,0.000000,0.000000);
rgb(868pt)=(0.352000,0.000000,0.000000);
rgb(869pt)=(0.349333,0.000000,0.000000);
rgb(870pt)=(0.346667,0.000000,0.000000);
rgb(871pt)=(0.344000,0.000000,0.000000);
rgb(872pt)=(0.341333,0.000000,0.000000);
rgb(873pt)=(0.338667,0.000000,0.000000);
rgb(874pt)=(0.336000,0.000000,0.000000);
rgb(875pt)=(0.333333,0.000000,0.000000);
rgb(876pt)=(0.330667,0.000000,0.000000);
rgb(877pt)=(0.328000,0.000000,0.000000);
rgb(878pt)=(0.325333,0.000000,0.000000);
rgb(879pt)=(0.322667,0.000000,0.000000);
rgb(880pt)=(0.320000,0.000000,0.000000);
rgb(881pt)=(0.317333,0.000000,0.000000);
rgb(882pt)=(0.314667,0.000000,0.000000);
rgb(883pt)=(0.312000,0.000000,0.000000);
rgb(884pt)=(0.309333,0.000000,0.000000);
rgb(885pt)=(0.306667,0.000000,0.000000);
rgb(886pt)=(0.304000,0.000000,0.000000);
rgb(887pt)=(0.301333,0.000000,0.000000);
rgb(888pt)=(0.298667,0.000000,0.000000);
rgb(889pt)=(0.296000,0.000000,0.000000);
rgb(890pt)=(0.293333,0.000000,0.000000);
rgb(891pt)=(0.290667,0.000000,0.000000);
rgb(892pt)=(0.288000,0.000000,0.000000);
rgb(893pt)=(0.285333,0.000000,0.000000);
rgb(894pt)=(0.282667,0.000000,0.000000);
rgb(895pt)=(0.280000,0.000000,0.000000);
rgb(896pt)=(0.277333,0.000000,0.000000);
rgb(897pt)=(0.274667,0.000000,0.000000);
rgb(898pt)=(0.272000,0.000000,0.000000);
rgb(899pt)=(0.269333,0.000000,0.000000);
rgb(900pt)=(0.266667,0.000000,0.000000);
rgb(901pt)=(0.264000,0.000000,0.000000);
rgb(902pt)=(0.261333,0.000000,0.000000);
rgb(903pt)=(0.258667,0.000000,0.000000);
rgb(904pt)=(0.256000,0.000000,0.000000);
rgb(905pt)=(0.253333,0.000000,0.000000);
rgb(906pt)=(0.250667,0.000000,0.000000);
rgb(907pt)=(0.248000,0.000000,0.000000);
rgb(908pt)=(0.245333,0.000000,0.000000);
rgb(909pt)=(0.242667,0.000000,0.000000);
rgb(910pt)=(0.240000,0.000000,0.000000);
rgb(911pt)=(0.237333,0.000000,0.000000);
rgb(912pt)=(0.234667,0.000000,0.000000);
rgb(913pt)=(0.232000,0.000000,0.000000);
rgb(914pt)=(0.229333,0.000000,0.000000);
rgb(915pt)=(0.226667,0.000000,0.000000);
rgb(916pt)=(0.224000,0.000000,0.000000);
rgb(917pt)=(0.221333,0.000000,0.000000);
rgb(918pt)=(0.218667,0.000000,0.000000);
rgb(919pt)=(0.216000,0.000000,0.000000);
rgb(920pt)=(0.213333,0.000000,0.000000);
rgb(921pt)=(0.210667,0.000000,0.000000);
rgb(922pt)=(0.208000,0.000000,0.000000);
rgb(923pt)=(0.205333,0.000000,0.000000);
rgb(924pt)=(0.202667,0.000000,0.000000);
rgb(925pt)=(0.200000,0.000000,0.000000);
rgb(926pt)=(0.197333,0.000000,0.000000);
rgb(927pt)=(0.194667,0.000000,0.000000);
rgb(928pt)=(0.192000,0.000000,0.000000);
rgb(929pt)=(0.189333,0.000000,0.000000);
rgb(930pt)=(0.186667,0.000000,0.000000);
rgb(931pt)=(0.184000,0.000000,0.000000);
rgb(932pt)=(0.181333,0.000000,0.000000);
rgb(933pt)=(0.178667,0.000000,0.000000);
rgb(934pt)=(0.176000,0.000000,0.000000);
rgb(935pt)=(0.173333,0.000000,0.000000);
rgb(936pt)=(0.170667,0.000000,0.000000);
rgb(937pt)=(0.168000,0.000000,0.000000);
rgb(938pt)=(0.165333,0.000000,0.000000);
rgb(939pt)=(0.162667,0.000000,0.000000);
rgb(940pt)=(0.160000,0.000000,0.000000);
rgb(941pt)=(0.157333,0.000000,0.000000);
rgb(942pt)=(0.154667,0.000000,0.000000);
rgb(943pt)=(0.152000,0.000000,0.000000);
rgb(944pt)=(0.149333,0.000000,0.000000);
rgb(945pt)=(0.146667,0.000000,0.000000);
rgb(946pt)=(0.144000,0.000000,0.000000);
rgb(947pt)=(0.141333,0.000000,0.000000);
rgb(948pt)=(0.138667,0.000000,0.000000);
rgb(949pt)=(0.136000,0.000000,0.000000);
rgb(950pt)=(0.133333,0.000000,0.000000);
rgb(951pt)=(0.130667,0.000000,0.000000);
rgb(952pt)=(0.128000,0.000000,0.000000);
rgb(953pt)=(0.125333,0.000000,0.000000);
rgb(954pt)=(0.122667,0.000000,0.000000);
rgb(955pt)=(0.120000,0.000000,0.000000);
rgb(956pt)=(0.117333,0.000000,0.000000);
rgb(957pt)=(0.114667,0.000000,0.000000);
rgb(958pt)=(0.112000,0.000000,0.000000);
rgb(959pt)=(0.109333,0.000000,0.000000);
rgb(960pt)=(0.106667,0.000000,0.000000);
rgb(961pt)=(0.104000,0.000000,0.000000);
rgb(962pt)=(0.101333,0.000000,0.000000);
rgb(963pt)=(0.098667,0.000000,0.000000);
rgb(964pt)=(0.096000,0.000000,0.000000);
rgb(965pt)=(0.093333,0.000000,0.000000);
rgb(966pt)=(0.090667,0.000000,0.000000);
rgb(967pt)=(0.088000,0.000000,0.000000);
rgb(968pt)=(0.085333,0.000000,0.000000);
rgb(969pt)=(0.082667,0.000000,0.000000);
rgb(970pt)=(0.080000,0.000000,0.000000);
rgb(971pt)=(0.077333,0.000000,0.000000);
rgb(972pt)=(0.074667,0.000000,0.000000);
rgb(973pt)=(0.072000,0.000000,0.000000);
rgb(974pt)=(0.069333,0.000000,0.000000);
rgb(975pt)=(0.066667,0.000000,0.000000);
rgb(976pt)=(0.064000,0.000000,0.000000);
rgb(977pt)=(0.061333,0.000000,0.000000);
rgb(978pt)=(0.058667,0.000000,0.000000);
rgb(979pt)=(0.056000,0.000000,0.000000);
rgb(980pt)=(0.053333,0.000000,0.000000);
rgb(981pt)=(0.050667,0.000000,0.000000);
rgb(982pt)=(0.048000,0.000000,0.000000);
rgb(983pt)=(0.045333,0.000000,0.000000);
rgb(984pt)=(0.042667,0.000000,0.000000);
rgb(985pt)=(0.040000,0.000000,0.000000);
rgb(986pt)=(0.037333,0.000000,0.000000);
rgb(987pt)=(0.034667,0.000000,0.000000);
rgb(988pt)=(0.032000,0.000000,0.000000);
rgb(989pt)=(0.029333,0.000000,0.000000);
rgb(990pt)=(0.026667,0.000000,0.000000);
rgb(991pt)=(0.024000,0.000000,0.000000);
rgb(992pt)=(0.021333,0.000000,0.000000);
rgb(993pt)=(0.018667,0.000000,0.000000);
rgb(994pt)=(0.016000,0.000000,0.000000);
rgb(995pt)=(0.013333,0.000000,0.000000);
rgb(996pt)=(0.010667,0.000000,0.000000);
rgb(997pt)=(0.008000,0.000000,0.000000);
rgb(998pt)=(0.005333,0.000000,0.000000);
rgb(999pt)=(0.002667,0.000000,0.000000);
}}

\newlength{\figureheight}
\newlength{\figurewidth}
\pgfplotsset{
	colormap={parula}{
rgb(0pt)=(0.2081,0.1663,0.5292);
rgb(1pt)=(0.208355,0.16778,0.532238);
rgb(2pt)=(0.208611,0.169261,0.535275);
rgb(3pt)=(0.208866,0.170741,0.538313);
rgb(4pt)=(0.209121,0.172222,0.54135);
rgb(5pt)=(0.209376,0.173702,0.544388);
rgb(6pt)=(0.209632,0.175183,0.547425);
rgb(7pt)=(0.209887,0.176663,0.550463);
rgb(8pt)=(0.210134,0.178144,0.553505);
rgb(9pt)=(0.210338,0.179624,0.556568);
rgb(10pt)=(0.210542,0.181105,0.559631);
rgb(11pt)=(0.210746,0.182585,0.562694);
rgb(12pt)=(0.210944,0.184066,0.565763);
rgb(13pt)=(0.211123,0.185546,0.568852);
rgb(14pt)=(0.211302,0.187027,0.57194);
rgb(15pt)=(0.21148,0.188507,0.575029);
rgb(16pt)=(0.211642,0.189996,0.578117);
rgb(17pt)=(0.21177,0.191502,0.581206);
rgb(18pt)=(0.211897,0.193008,0.584295);
rgb(19pt)=(0.212025,0.194514,0.587383);
rgb(20pt)=(0.212132,0.19602,0.590472);
rgb(21pt)=(0.212208,0.197526,0.59356);
rgb(22pt)=(0.212285,0.199032,0.596649);
rgb(23pt)=(0.212361,0.200538,0.599738);
rgb(24pt)=(0.212413,0.202044,0.602839);
rgb(25pt)=(0.212438,0.20355,0.605953);
rgb(26pt)=(0.212464,0.205056,0.609067);
rgb(27pt)=(0.212489,0.206562,0.612181);
rgb(28pt)=(0.212471,0.208083,0.61531);
rgb(29pt)=(0.21242,0.209614,0.61845);
rgb(30pt)=(0.212368,0.211146,0.621589);
rgb(31pt)=(0.212317,0.212677,0.624729);
rgb(32pt)=(0.212216,0.214209,0.627868);
rgb(33pt)=(0.212088,0.215741,0.631008);
rgb(34pt)=(0.211961,0.217272,0.634148);
rgb(35pt)=(0.211833,0.218804,0.637287);
rgb(36pt)=(0.211668,0.220354,0.640446);
rgb(37pt)=(0.211489,0.221911,0.643611);
rgb(38pt)=(0.21131,0.223468,0.646776);
rgb(39pt)=(0.211132,0.225025,0.649941);
rgb(40pt)=(0.210848,0.226603,0.653107);
rgb(41pt)=(0.210541,0.228186,0.656272);
rgb(42pt)=(0.210235,0.229768,0.659437);
rgb(43pt)=(0.209929,0.231351,0.662602);
rgb(44pt)=(0.209553,0.232934,0.665767);
rgb(45pt)=(0.20917,0.234516,0.668932);
rgb(46pt)=(0.208787,0.236099,0.672098);
rgb(47pt)=(0.208405,0.237681,0.675263);
rgb(48pt)=(0.20787,0.239289,0.678453);
rgb(49pt)=(0.207334,0.240897,0.681644);
rgb(50pt)=(0.206798,0.242505,0.684835);
rgb(51pt)=(0.206255,0.244114,0.688025);
rgb(52pt)=(0.205617,0.245722,0.691216);
rgb(53pt)=(0.204979,0.24733,0.694407);
rgb(54pt)=(0.204341,0.248938,0.697597);
rgb(55pt)=(0.203675,0.250554,0.700792);
rgb(56pt)=(0.202858,0.252213,0.704008);
rgb(57pt)=(0.202041,0.253872,0.707224);
rgb(58pt)=(0.201225,0.255531,0.710441);
rgb(59pt)=(0.200372,0.257184,0.713657);
rgb(60pt)=(0.199402,0.258818,0.716873);
rgb(61pt)=(0.198432,0.260452,0.720089);
rgb(62pt)=(0.197462,0.262085,0.723305);
rgb(63pt)=(0.196419,0.263735,0.726522);
rgb(64pt)=(0.195219,0.26542,0.729738);
rgb(65pt)=(0.19402,0.267105,0.732954);
rgb(66pt)=(0.19282,0.268789,0.73617);
rgb(67pt)=(0.191549,0.270474,0.739386);
rgb(68pt)=(0.19017,0.272159,0.742603);
rgb(69pt)=(0.188792,0.273843,0.745819);
rgb(70pt)=(0.187414,0.275528,0.749035);
rgb(71pt)=(0.1859,0.277237,0.752264);
rgb(72pt)=(0.184241,0.278973,0.755505);
rgb(73pt)=(0.182581,0.280709,0.758747);
rgb(74pt)=(0.180922,0.282444,0.761989);
rgb(75pt)=(0.179133,0.284209,0.765245);
rgb(76pt)=(0.177244,0.285996,0.768512);
rgb(77pt)=(0.175356,0.287783,0.77178);
rgb(78pt)=(0.173467,0.289569,0.775047);
rgb(79pt)=(0.171363,0.291406,0.778314);
rgb(80pt)=(0.169142,0.293269,0.781581);
rgb(81pt)=(0.166922,0.295132,0.784849);
rgb(82pt)=(0.164701,0.296996,0.788116);
rgb(83pt)=(0.162238,0.298934,0.791365);
rgb(84pt)=(0.159686,0.300899,0.794606);
rgb(85pt)=(0.157133,0.302865,0.797848);
rgb(86pt)=(0.15458,0.30483,0.80109);
rgb(87pt)=(0.151738,0.306858,0.804352);
rgb(88pt)=(0.148828,0.3089,0.80762);
rgb(89pt)=(0.145918,0.310942,0.810887);
rgb(90pt)=(0.143008,0.312984,0.814154);
rgb(91pt)=(0.139687,0.31514,0.81733);
rgb(92pt)=(0.136318,0.31731,0.820495);
rgb(93pt)=(0.132949,0.319479,0.82366);
rgb(94pt)=(0.129579,0.321649,0.826826);
rgb(95pt)=(0.125811,0.323918,0.829841);
rgb(96pt)=(0.122033,0.32619,0.832853);
rgb(97pt)=(0.118256,0.328462,0.835865);
rgb(98pt)=(0.114458,0.330737,0.838862);
rgb(99pt)=(0.110349,0.333059,0.841619);
rgb(100pt)=(0.106239,0.335382,0.844376);
rgb(101pt)=(0.102129,0.337705,0.847132);
rgb(102pt)=(0.0979874,0.340021,0.849835);
rgb(103pt)=(0.093648,0.342292,0.852209);
rgb(104pt)=(0.0893087,0.344564,0.854583);
rgb(105pt)=(0.0849694,0.346836,0.856957);
rgb(106pt)=(0.08063,0.349091,0.859234);
rgb(107pt)=(0.0762907,0.351286,0.861174);
rgb(108pt)=(0.0719514,0.353481,0.863114);
rgb(109pt)=(0.067612,0.355676,0.865053);
rgb(110pt)=(0.0633195,0.357817,0.866853);
rgb(111pt)=(0.0591333,0.359833,0.868333);
rgb(112pt)=(0.0549471,0.36185,0.869814);
rgb(113pt)=(0.050761,0.363866,0.871294);
rgb(114pt)=(0.0466838,0.365823,0.872626);
rgb(115pt)=(0.0427784,0.367687,0.873724);
rgb(116pt)=(0.038873,0.36955,0.874821);
rgb(117pt)=(0.0349676,0.371414,0.875919);
rgb(118pt)=(0.0315066,0.373217,0.876872);
rgb(119pt)=(0.0285456,0.374953,0.877664);
rgb(120pt)=(0.0255847,0.376688,0.878455);
rgb(121pt)=(0.0226237,0.378424,0.879246);
rgb(122pt)=(0.0202132,0.380061,0.879868);
rgb(123pt)=(0.0182477,0.381618,0.880353);
rgb(124pt)=(0.0162823,0.383175,0.880838);
rgb(125pt)=(0.0143168,0.384732,0.881323);
rgb(126pt)=(0.0127892,0.386241,0.881695);
rgb(127pt)=(0.0115129,0.387721,0.882001);
rgb(128pt)=(0.0102366,0.389202,0.882307);
rgb(129pt)=(0.00896036,0.390682,0.882614);
rgb(130pt)=(0.00812372,0.392089,0.88281);
rgb(131pt)=(0.00746006,0.393468,0.882963);
rgb(132pt)=(0.0067964,0.394846,0.883116);
rgb(133pt)=(0.00613273,0.396224,0.883269);
rgb(134pt)=(0.00581622,0.397562,0.88332);
rgb(135pt)=(0.00558649,0.398889,0.883346);
rgb(136pt)=(0.00535676,0.400217,0.883371);
rgb(137pt)=(0.00512703,0.401544,0.883397);
rgb(138pt)=(0.00516757,0.402804,0.883332);
rgb(139pt)=(0.00524414,0.404054,0.883256);
rgb(140pt)=(0.00532072,0.405305,0.883179);
rgb(141pt)=(0.0053973,0.406556,0.883103);
rgb(142pt)=(0.00572012,0.407757,0.882952);
rgb(143pt)=(0.00605195,0.408957,0.882799);
rgb(144pt)=(0.00638378,0.410157,0.882646);
rgb(145pt)=(0.00672643,0.411355,0.882489);
rgb(146pt)=(0.00728799,0.412529,0.882259);
rgb(147pt)=(0.00784955,0.413704,0.88203);
rgb(148pt)=(0.00841111,0.414878,0.8818);
rgb(149pt)=(0.00898919,0.416045,0.881564);
rgb(150pt)=(0.00967838,0.417168,0.881283);
rgb(151pt)=(0.0103676,0.418292,0.881002);
rgb(152pt)=(0.0110568,0.419415,0.880721);
rgb(153pt)=(0.011773,0.420532,0.880435);
rgb(154pt)=(0.0125898,0.42163,0.880129);
rgb(155pt)=(0.0134066,0.422728,0.879823);
rgb(156pt)=(0.0142234,0.423825,0.879516);
rgb(157pt)=(0.0150703,0.424915,0.879195);
rgb(158pt)=(0.0159892,0.425987,0.878838);
rgb(159pt)=(0.0169081,0.427059,0.87848);
rgb(160pt)=(0.017827,0.428132,0.878123);
rgb(161pt)=(0.0187748,0.429194,0.877746);
rgb(162pt)=(0.0197703,0.430241,0.877338);
rgb(163pt)=(0.0207658,0.431287,0.876929);
rgb(164pt)=(0.0217613,0.432334,0.876521);
rgb(165pt)=(0.0227802,0.43338,0.876113);
rgb(166pt)=(0.0238267,0.434427,0.875704);
rgb(167pt)=(0.0248733,0.435473,0.875296);
rgb(168pt)=(0.0259198,0.43652,0.874887);
rgb(169pt)=(0.0269802,0.437553,0.874451);
rgb(170pt)=(0.0280523,0.438574,0.873992);
rgb(171pt)=(0.0291243,0.439595,0.873532);
rgb(172pt)=(0.0301964,0.440616,0.873073);
rgb(173pt)=(0.0312844,0.441621,0.872614);
rgb(174pt)=(0.032382,0.442616,0.872154);
rgb(175pt)=(0.0334796,0.443612,0.871695);
rgb(176pt)=(0.0345772,0.444607,0.871235);
rgb(177pt)=(0.0357108,0.445603,0.870758);
rgb(178pt)=(0.0368595,0.446598,0.870273);
rgb(179pt)=(0.0380081,0.447594,0.869788);
rgb(180pt)=(0.0391568,0.448589,0.869303);
rgb(181pt)=(0.0402652,0.449565,0.868798);
rgb(182pt)=(0.0413628,0.450535,0.868287);
rgb(183pt)=(0.0424604,0.451505,0.867777);
rgb(184pt)=(0.043558,0.452474,0.867266);
rgb(185pt)=(0.0445889,0.453444,0.866756);
rgb(186pt)=(0.0456099,0.454414,0.866245);
rgb(187pt)=(0.0466309,0.455384,0.865735);
rgb(188pt)=(0.047652,0.456354,0.865224);
rgb(189pt)=(0.0486,0.457324,0.864714);
rgb(190pt)=(0.0495444,0.458294,0.864203);
rgb(191pt)=(0.0504889,0.459264,0.863692);
rgb(192pt)=(0.0514315,0.460234,0.863181);
rgb(193pt)=(0.0523249,0.461204,0.862645);
rgb(194pt)=(0.0532183,0.462174,0.862109);
rgb(195pt)=(0.0541117,0.463144,0.861573);
rgb(196pt)=(0.0549991,0.464111,0.861034);
rgb(197pt)=(0.0558414,0.465056,0.860472);
rgb(198pt)=(0.0566838,0.466,0.859911);
rgb(199pt)=(0.0575261,0.466944,0.859349);
rgb(200pt)=(0.0583532,0.467889,0.858793);
rgb(201pt)=(0.0591189,0.468833,0.858257);
rgb(202pt)=(0.0598847,0.469778,0.857721);
rgb(203pt)=(0.0606505,0.470722,0.857185);
rgb(204pt)=(0.0614018,0.471667,0.856641);
rgb(205pt)=(0.0621165,0.472611,0.85608);
rgb(206pt)=(0.0628312,0.473556,0.855518);
rgb(207pt)=(0.0635459,0.4745,0.854957);
rgb(208pt)=(0.064242,0.475444,0.854405);
rgb(209pt)=(0.0649057,0.476389,0.853868);
rgb(210pt)=(0.0655694,0.477333,0.853332);
rgb(211pt)=(0.066233,0.478278,0.852796);
rgb(212pt)=(0.0668625,0.479222,0.852249);
rgb(213pt)=(0.0674495,0.480167,0.851687);
rgb(214pt)=(0.0680366,0.481111,0.851126);
rgb(215pt)=(0.0686237,0.482056,0.850564);
rgb(216pt)=(0.0691838,0.483,0.850003);
rgb(217pt)=(0.0697198,0.483944,0.849441);
rgb(218pt)=(0.0702559,0.484889,0.84888);
rgb(219pt)=(0.0707919,0.485833,0.848318);
rgb(220pt)=(0.0712967,0.486778,0.847772);
rgb(221pt)=(0.0717817,0.487722,0.847236);
rgb(222pt)=(0.0722667,0.488667,0.8467);
rgb(223pt)=(0.0727517,0.489611,0.846164);
rgb(224pt)=(0.0732012,0.490573,0.845628);
rgb(225pt)=(0.0736351,0.491543,0.845092);
rgb(226pt)=(0.0740691,0.492513,0.844556);
rgb(227pt)=(0.074503,0.493483,0.84402);
rgb(228pt)=(0.0748973,0.494433,0.843484);
rgb(229pt)=(0.0752802,0.495378,0.842948);
rgb(230pt)=(0.0756631,0.496322,0.842412);
rgb(231pt)=(0.0760459,0.497267,0.841876);
rgb(232pt)=(0.0763631,0.498233,0.841362);
rgb(233pt)=(0.0766694,0.499203,0.840851);
rgb(234pt)=(0.0769757,0.500173,0.840341);
rgb(235pt)=(0.077282,0.501143,0.83983);
rgb(236pt)=(0.0775162,0.502137,0.83932);
rgb(237pt)=(0.0777459,0.503132,0.838809);
rgb(238pt)=(0.0779757,0.504128,0.838298);
rgb(239pt)=(0.0782042,0.505123,0.837789);
rgb(240pt)=(0.0783829,0.506093,0.837304);
rgb(241pt)=(0.0785616,0.507063,0.836819);
rgb(242pt)=(0.0787402,0.508033,0.836334);
rgb(243pt)=(0.0789135,0.509008,0.835851);
rgb(244pt)=(0.0790411,0.510029,0.835392);
rgb(245pt)=(0.0791688,0.51105,0.834932);
rgb(246pt)=(0.0792964,0.512071,0.834473);
rgb(247pt)=(0.0794048,0.513092,0.834018);
rgb(248pt)=(0.0794303,0.514113,0.833584);
rgb(249pt)=(0.0794559,0.515134,0.83315);
rgb(250pt)=(0.0794814,0.516155,0.832717);
rgb(251pt)=(0.0794862,0.517183,0.832289);
rgb(252pt)=(0.0794351,0.51823,0.831881);
rgb(253pt)=(0.0793841,0.519276,0.831473);
rgb(254pt)=(0.079333,0.520323,0.831064);
rgb(255pt)=(0.079255,0.521369,0.830665);
rgb(256pt)=(0.0791273,0.522416,0.830282);
rgb(257pt)=(0.0789997,0.523462,0.829899);
rgb(258pt)=(0.0788721,0.524509,0.829516);
rgb(259pt)=(0.0786889,0.525589,0.829156);
rgb(260pt)=(0.0784336,0.526712,0.828824);
rgb(261pt)=(0.0781784,0.527835,0.828492);
rgb(262pt)=(0.0779231,0.528958,0.82816);
rgb(263pt)=(0.077615,0.530081,0.827868);
rgb(264pt)=(0.0772577,0.531205,0.827613);
rgb(265pt)=(0.0769003,0.532328,0.827357);
rgb(266pt)=(0.0765429,0.533451,0.827102);
rgb(267pt)=(0.0761243,0.534589,0.826862);
rgb(268pt)=(0.0756649,0.535738,0.826632);
rgb(269pt)=(0.0752054,0.536886,0.826403);
rgb(270pt)=(0.0747459,0.538035,0.826173);
rgb(271pt)=(0.0742168,0.539219,0.825961);
rgb(272pt)=(0.0736553,0.540418,0.825756);
rgb(273pt)=(0.0730937,0.541618,0.825552);
rgb(274pt)=(0.0725321,0.542818,0.825348);
rgb(275pt)=(0.0718925,0.544037,0.825183);
rgb(276pt)=(0.0712288,0.545262,0.82503);
rgb(277pt)=(0.0705652,0.546487,0.824877);
rgb(278pt)=(0.0699015,0.547713,0.824723);
rgb(279pt)=(0.0691514,0.548938,0.824614);
rgb(280pt)=(0.0683856,0.550163,0.824511);
rgb(281pt)=(0.0676198,0.551388,0.824409);
rgb(282pt)=(0.0668541,0.552614,0.824307);
rgb(283pt)=(0.0660408,0.553886,0.824205);
rgb(284pt)=(0.065224,0.555162,0.824103);
rgb(285pt)=(0.0644072,0.556439,0.824001);
rgb(286pt)=(0.0635892,0.557715,0.823899);
rgb(287pt)=(0.0626703,0.558991,0.823848);
rgb(288pt)=(0.0617514,0.560268,0.823797);
rgb(289pt)=(0.0608324,0.561544,0.823746);
rgb(290pt)=(0.0599087,0.56282,0.823693);
rgb(291pt)=(0.0589387,0.564096,0.823616);
rgb(292pt)=(0.0579688,0.565373,0.82354);
rgb(293pt)=(0.0569988,0.566649,0.823463);
rgb(294pt)=(0.0560243,0.567925,0.823386);
rgb(295pt)=(0.0550288,0.569202,0.82331);
rgb(296pt)=(0.0540333,0.570478,0.823233);
rgb(297pt)=(0.0530378,0.571754,0.823157);
rgb(298pt)=(0.0520423,0.57303,0.82308);
rgb(299pt)=(0.0510468,0.574307,0.823004);
rgb(300pt)=(0.0500514,0.575583,0.822927);
rgb(301pt)=(0.0490559,0.576859,0.82285);
rgb(302pt)=(0.0480604,0.578127,0.822756);
rgb(303pt)=(0.0470649,0.579377,0.822629);
rgb(304pt)=(0.0460694,0.580628,0.822501);
rgb(305pt)=(0.0450739,0.581879,0.822374);
rgb(306pt)=(0.0441,0.583119,0.822235);
rgb(307pt)=(0.0431556,0.584344,0.822082);
rgb(308pt)=(0.0422111,0.585569,0.821929);
rgb(309pt)=(0.0412667,0.586795,0.821776);
rgb(310pt)=(0.0403351,0.58802,0.821597);
rgb(311pt)=(0.0394162,0.589245,0.821392);
rgb(312pt)=(0.0384973,0.59047,0.821188);
rgb(313pt)=(0.0375784,0.591695,0.820984);
rgb(314pt)=(0.0367495,0.592891,0.820735);
rgb(315pt)=(0.0359838,0.594065,0.820454);
rgb(316pt)=(0.035218,0.595239,0.820173);
rgb(317pt)=(0.0344523,0.596413,0.819892);
rgb(318pt)=(0.0337721,0.597553,0.819595);
rgb(319pt)=(0.0331339,0.598676,0.819288);
rgb(320pt)=(0.0324958,0.599799,0.818982);
rgb(321pt)=(0.0318577,0.600923,0.818676);
rgb(322pt)=(0.0312964,0.602026,0.818312);
rgb(323pt)=(0.0307604,0.603124,0.817929);
rgb(324pt)=(0.0302243,0.604222,0.817546);
rgb(325pt)=(0.0296883,0.605319,0.817163);
rgb(326pt)=(0.0292375,0.606395,0.816738);
rgb(327pt)=(0.0288036,0.607468,0.816304);
rgb(328pt)=(0.0283697,0.60854,0.81587);
rgb(329pt)=(0.0279357,0.609612,0.815436);
rgb(330pt)=(0.0275721,0.610637,0.814955);
rgb(331pt)=(0.0272147,0.611658,0.81447);
rgb(332pt)=(0.0268574,0.612679,0.813985);
rgb(333pt)=(0.0265,0.6137,0.8135);
rgb(334pt)=(0.0262447,0.614695,0.812964);
rgb(335pt)=(0.0259895,0.615691,0.812428);
rgb(336pt)=(0.0257342,0.616686,0.811892);
rgb(337pt)=(0.0254853,0.61768,0.811352);
rgb(338pt)=(0.0253066,0.61865,0.810765);
rgb(339pt)=(0.0251279,0.61962,0.810177);
rgb(340pt)=(0.0249492,0.62059,0.80959);
rgb(341pt)=(0.024779,0.621551,0.808995);
rgb(342pt)=(0.0246514,0.62247,0.808357);
rgb(343pt)=(0.0245237,0.623389,0.807719);
rgb(344pt)=(0.0243961,0.624308,0.80708);
rgb(345pt)=(0.0242748,0.625221,0.80643);
rgb(346pt)=(0.0241727,0.626114,0.805741);
rgb(347pt)=(0.0240706,0.627008,0.805051);
rgb(348pt)=(0.0239685,0.627901,0.804362);
rgb(349pt)=(0.0238832,0.628786,0.803656);
rgb(350pt)=(0.0238321,0.629654,0.802916);
rgb(351pt)=(0.0237811,0.630522,0.802176);
rgb(352pt)=(0.02373,0.631389,0.801435);
rgb(353pt)=(0.023679,0.632247,0.800685);
rgb(354pt)=(0.0236279,0.633089,0.799919);
rgb(355pt)=(0.0235769,0.633932,0.799153);
rgb(356pt)=(0.0235258,0.634774,0.798387);
rgb(357pt)=(0.0234748,0.635604,0.797596);
rgb(358pt)=(0.0234237,0.63642,0.79678);
rgb(359pt)=(0.0233727,0.637237,0.795963);
rgb(360pt)=(0.0233216,0.638054,0.795146);
rgb(361pt)=(0.0232706,0.638856,0.794329);
rgb(362pt)=(0.0232195,0.639647,0.793512);
rgb(363pt)=(0.0231685,0.640439,0.792695);
rgb(364pt)=(0.0231174,0.64123,0.791879);
rgb(365pt)=(0.0230832,0.642005,0.791011);
rgb(366pt)=(0.0230577,0.64277,0.790118);
rgb(367pt)=(0.0230321,0.643536,0.789225);
rgb(368pt)=(0.0230066,0.644302,0.788331);
rgb(369pt)=(0.0229811,0.645049,0.787438);
rgb(370pt)=(0.0229556,0.645789,0.786544);
rgb(371pt)=(0.02293,0.646529,0.785651);
rgb(372pt)=(0.0229045,0.647269,0.784758);
rgb(373pt)=(0.022858,0.64801,0.783843);
rgb(374pt)=(0.0228069,0.64875,0.782924);
rgb(375pt)=(0.0227559,0.64949,0.782005);
rgb(376pt)=(0.0227048,0.65023,0.781086);
rgb(377pt)=(0.0227,0.650947,0.780144);
rgb(378pt)=(0.0227,0.651662,0.7792);
rgb(379pt)=(0.0227,0.652377,0.778256);
rgb(380pt)=(0.0227,0.653092,0.777311);
rgb(381pt)=(0.0228261,0.653781,0.776341);
rgb(382pt)=(0.0229538,0.65447,0.775371);
rgb(383pt)=(0.0230814,0.655159,0.774402);
rgb(384pt)=(0.0232108,0.655849,0.77343);
rgb(385pt)=(0.023364,0.656538,0.772434);
rgb(386pt)=(0.0235171,0.657227,0.771439);
rgb(387pt)=(0.0236703,0.657916,0.770443);
rgb(388pt)=(0.0238312,0.658602,0.769444);
rgb(389pt)=(0.0240354,0.659265,0.768423);
rgb(390pt)=(0.0242396,0.659929,0.767402);
rgb(391pt)=(0.0244438,0.660592,0.766381);
rgb(392pt)=(0.0247021,0.661256,0.765354);
rgb(393pt)=(0.025136,0.66192,0.764307);
rgb(394pt)=(0.02557,0.662583,0.763261);
rgb(395pt)=(0.0260039,0.663247,0.762214);
rgb(396pt)=(0.0264541,0.663911,0.761168);
rgb(397pt)=(0.026939,0.664574,0.760121);
rgb(398pt)=(0.027424,0.665238,0.759074);
rgb(399pt)=(0.027909,0.665902,0.758028);
rgb(400pt)=(0.028445,0.666555,0.756971);
rgb(401pt)=(0.0290577,0.667193,0.755899);
rgb(402pt)=(0.0296703,0.667832,0.754827);
rgb(403pt)=(0.0302829,0.66847,0.753755);
rgb(404pt)=(0.030994,0.669095,0.752683);
rgb(405pt)=(0.0318108,0.669708,0.751611);
rgb(406pt)=(0.0326276,0.670321,0.750539);
rgb(407pt)=(0.0334444,0.670933,0.749467);
rgb(408pt)=(0.0343045,0.67156,0.748366);
rgb(409pt)=(0.0351979,0.672198,0.747243);
rgb(410pt)=(0.0360913,0.672837,0.74612);
rgb(411pt)=(0.0369847,0.673475,0.744996);
rgb(412pt)=(0.0380432,0.674096,0.743873);
rgb(413pt)=(0.0391919,0.674709,0.74275);
rgb(414pt)=(0.0403405,0.675322,0.741627);
rgb(415pt)=(0.0414892,0.675934,0.740504);
rgb(416pt)=(0.0427123,0.676528,0.739381);
rgb(417pt)=(0.0439631,0.677115,0.738258);
rgb(418pt)=(0.0452138,0.677702,0.737135);
rgb(419pt)=(0.0464646,0.678289,0.736011);
rgb(420pt)=(0.0477153,0.678897,0.734868);
rgb(421pt)=(0.0489661,0.67951,0.733719);
rgb(422pt)=(0.0502168,0.680123,0.73257);
rgb(423pt)=(0.0514676,0.680735,0.731422);
rgb(424pt)=(0.0529237,0.681325,0.73025);
rgb(425pt)=(0.0544042,0.681912,0.729076);
rgb(426pt)=(0.0558847,0.682499,0.727902);
rgb(427pt)=(0.0573652,0.683086,0.726728);
rgb(428pt)=(0.0587709,0.683673,0.725553);
rgb(429pt)=(0.0601748,0.68426,0.724379);
rgb(430pt)=(0.0615787,0.684847,0.723205);
rgb(431pt)=(0.0629946,0.685435,0.722028);
rgb(432pt)=(0.0646027,0.686022,0.720803);
rgb(433pt)=(0.0662108,0.686609,0.719577);
rgb(434pt)=(0.0678189,0.687196,0.718352);
rgb(435pt)=(0.069427,0.687779,0.717131);
rgb(436pt)=(0.0710351,0.688341,0.715931);
rgb(437pt)=(0.0726432,0.688902,0.714731);
rgb(438pt)=(0.0742514,0.689464,0.713532);
rgb(439pt)=(0.0758709,0.690026,0.712326);
rgb(440pt)=(0.07753,0.690587,0.711101);
rgb(441pt)=(0.0791892,0.691149,0.709876);
rgb(442pt)=(0.0808483,0.69171,0.70865);
rgb(443pt)=(0.0825387,0.692272,0.707417);
rgb(444pt)=(0.0843,0.692833,0.706167);
rgb(445pt)=(0.0860613,0.693395,0.704916);
rgb(446pt)=(0.0878225,0.693956,0.703665);
rgb(447pt)=(0.089564,0.694518,0.702405);
rgb(448pt)=(0.0912742,0.69508,0.701128);
rgb(449pt)=(0.0929844,0.695641,0.699852);
rgb(450pt)=(0.0946946,0.696203,0.698576);
rgb(451pt)=(0.0965009,0.696752,0.697299);
rgb(452pt)=(0.0984153,0.697288,0.696023);
rgb(453pt)=(0.10033,0.697824,0.694747);
rgb(454pt)=(0.102244,0.69836,0.693471);
rgb(455pt)=(0.10413,0.698896,0.69218);
rgb(456pt)=(0.105994,0.699432,0.690878);
rgb(457pt)=(0.107857,0.699968,0.689577);
rgb(458pt)=(0.10972,0.700505,0.688275);
rgb(459pt)=(0.111632,0.701041,0.686973);
rgb(460pt)=(0.113572,0.701577,0.685671);
rgb(461pt)=(0.115512,0.702113,0.684369);
rgb(462pt)=(0.117452,0.702649,0.683068);
rgb(463pt)=(0.119429,0.703185,0.681747);
rgb(464pt)=(0.12142,0.703721,0.68042);
rgb(465pt)=(0.123411,0.704257,0.679093);
rgb(466pt)=(0.125402,0.704793,0.677765);
rgb(467pt)=(0.127372,0.705308,0.676438);
rgb(468pt)=(0.129338,0.705819,0.675111);
rgb(469pt)=(0.131303,0.706329,0.673783);
rgb(470pt)=(0.133269,0.70684,0.672456);
rgb(471pt)=(0.135369,0.70735,0.671084);
rgb(472pt)=(0.137488,0.707861,0.669705);
rgb(473pt)=(0.139607,0.708371,0.668327);
rgb(474pt)=(0.141725,0.708882,0.666949);
rgb(475pt)=(0.143795,0.709392,0.665595);
rgb(476pt)=(0.145862,0.709903,0.664242);
rgb(477pt)=(0.14793,0.710414,0.662889);
rgb(478pt)=(0.150003,0.710924,0.661534);
rgb(479pt)=(0.152198,0.711435,0.66013);
rgb(480pt)=(0.154394,0.711945,0.658726);
rgb(481pt)=(0.156589,0.712456,0.657322);
rgb(482pt)=(0.158784,0.712963,0.655922);
rgb(483pt)=(0.160979,0.713448,0.654543);
rgb(484pt)=(0.163174,0.713933,0.653165);
rgb(485pt)=(0.16537,0.714418,0.651786);
rgb(486pt)=(0.16757,0.714908,0.650397);
rgb(487pt)=(0.169791,0.715419,0.648968);
rgb(488pt)=(0.172012,0.715929,0.647538);
rgb(489pt)=(0.174232,0.71644,0.646109);
rgb(490pt)=(0.176483,0.716935,0.64468);
rgb(491pt)=(0.178806,0.717395,0.64325);
rgb(492pt)=(0.181129,0.717854,0.641821);
rgb(493pt)=(0.183452,0.718314,0.640391);
rgb(494pt)=(0.185755,0.718783,0.638952);
rgb(495pt)=(0.188027,0.719268,0.637497);
rgb(496pt)=(0.190299,0.719753,0.636042);
rgb(497pt)=(0.192571,0.720238,0.634587);
rgb(498pt)=(0.194913,0.720711,0.633132);
rgb(499pt)=(0.197338,0.72117,0.631677);
rgb(500pt)=(0.199762,0.72163,0.630223);
rgb(501pt)=(0.202187,0.722089,0.628768);
rgb(502pt)=(0.204612,0.722549,0.627299);
rgb(503pt)=(0.207037,0.723008,0.625818);
rgb(504pt)=(0.209462,0.723468,0.624338);
rgb(505pt)=(0.211887,0.723927,0.622857);
rgb(506pt)=(0.214328,0.724386,0.621377);
rgb(507pt)=(0.216778,0.724846,0.619896);
rgb(508pt)=(0.219229,0.725305,0.618416);
rgb(509pt)=(0.221679,0.725765,0.616935);
rgb(510pt)=(0.224202,0.726188,0.615455);
rgb(511pt)=(0.226754,0.726597,0.613974);
rgb(512pt)=(0.229307,0.727005,0.612494);
rgb(513pt)=(0.231859,0.727414,0.611014);
rgb(514pt)=(0.234392,0.727842,0.609513);
rgb(515pt)=(0.236919,0.728276,0.608007);
rgb(516pt)=(0.239446,0.72871,0.606501);
rgb(517pt)=(0.241973,0.729144,0.604995);
rgb(518pt)=(0.244611,0.729556,0.603467);
rgb(519pt)=(0.247266,0.729964,0.601935);
rgb(520pt)=(0.24992,0.730372,0.600404);
rgb(521pt)=(0.252575,0.730781,0.598872);
rgb(522pt)=(0.25523,0.731189,0.597365);
rgb(523pt)=(0.257884,0.731598,0.595859);
rgb(524pt)=(0.260539,0.732006,0.594353);
rgb(525pt)=(0.263194,0.732414,0.592846);
rgb(526pt)=(0.265848,0.732796,0.591314);
rgb(527pt)=(0.268503,0.733179,0.589783);
rgb(528pt)=(0.271158,0.733562,0.588251);
rgb(529pt)=(0.27383,0.733945,0.58672);
rgb(530pt)=(0.276638,0.734328,0.585188);
rgb(531pt)=(0.279446,0.734711,0.583657);
rgb(532pt)=(0.282254,0.735094,0.582125);
rgb(533pt)=(0.285051,0.735471,0.580594);
rgb(534pt)=(0.287808,0.735829,0.579062);
rgb(535pt)=(0.290565,0.736186,0.577531);
rgb(536pt)=(0.293322,0.736544,0.575999);
rgb(537pt)=(0.2961,0.736894,0.574468);
rgb(538pt)=(0.298933,0.737226,0.572936);
rgb(539pt)=(0.301767,0.737557,0.571405);
rgb(540pt)=(0.3046,0.737889,0.569873);
rgb(541pt)=(0.307452,0.738221,0.568351);
rgb(542pt)=(0.310336,0.738553,0.566845);
rgb(543pt)=(0.313221,0.738885,0.565339);
rgb(544pt)=(0.316105,0.739217,0.563833);
rgb(545pt)=(0.318978,0.739537,0.562315);
rgb(546pt)=(0.321837,0.739843,0.560784);
rgb(547pt)=(0.324696,0.74015,0.559252);
rgb(548pt)=(0.327555,0.740456,0.557721);
rgb(549pt)=(0.330468,0.740749,0.556216);
rgb(550pt)=(0.333429,0.741029,0.554736);
rgb(551pt)=(0.336389,0.74131,0.553255);
rgb(552pt)=(0.33935,0.741591,0.551775);
rgb(553pt)=(0.342296,0.741872,0.550279);
rgb(554pt)=(0.345231,0.742153,0.548773);
rgb(555pt)=(0.348167,0.742433,0.547267);
rgb(556pt)=(0.351102,0.742714,0.545761);
rgb(557pt)=(0.354038,0.742977,0.54429);
rgb(558pt)=(0.356973,0.743232,0.542835);
rgb(559pt)=(0.359908,0.743488,0.54138);
rgb(560pt)=(0.362844,0.743743,0.539925);
rgb(561pt)=(0.365839,0.743959,0.53847);
rgb(562pt)=(0.368851,0.744163,0.537015);
rgb(563pt)=(0.371863,0.744367,0.53556);
rgb(564pt)=(0.374875,0.744571,0.534105);
rgb(565pt)=(0.377843,0.744775,0.532672);
rgb(566pt)=(0.380804,0.74498,0.531243);
rgb(567pt)=(0.383765,0.745184,0.529814);
rgb(568pt)=(0.386726,0.745388,0.528384);
rgb(569pt)=(0.389711,0.745568,0.527003);
rgb(570pt)=(0.392697,0.745747,0.525624);
rgb(571pt)=(0.395684,0.745926,0.524246);
rgb(572pt)=(0.39867,0.746104,0.522868);
rgb(573pt)=(0.401657,0.746257,0.521489);
rgb(574pt)=(0.404643,0.74641,0.520111);
rgb(575pt)=(0.40763,0.746563,0.518732);
rgb(576pt)=(0.410611,0.746716,0.517359);
rgb(577pt)=(0.413546,0.746869,0.516032);
rgb(578pt)=(0.416482,0.747023,0.514705);
rgb(579pt)=(0.419417,0.747176,0.513377);
rgb(580pt)=(0.422357,0.747319,0.512055);
rgb(581pt)=(0.425318,0.747421,0.510753);
rgb(582pt)=(0.428279,0.747523,0.509451);
rgb(583pt)=(0.43124,0.747626,0.50815);
rgb(584pt)=(0.43418,0.747735,0.506848);
rgb(585pt)=(0.437065,0.747862,0.505546);
rgb(586pt)=(0.439949,0.74799,0.504244);
rgb(587pt)=(0.442834,0.748117,0.502942);
rgb(588pt)=(0.445727,0.748227,0.501659);
rgb(589pt)=(0.448637,0.748304,0.500408);
rgb(590pt)=(0.451547,0.74838,0.499157);
rgb(591pt)=(0.454457,0.748457,0.497906);
rgb(592pt)=(0.457333,0.748522,0.496667);
rgb(593pt)=(0.460167,0.748573,0.495441);
rgb(594pt)=(0.463,0.748624,0.494216);
rgb(595pt)=(0.465833,0.748675,0.492991);
rgb(596pt)=(0.468667,0.748726,0.491779);
rgb(597pt)=(0.4715,0.748777,0.490579);
rgb(598pt)=(0.474333,0.748829,0.48938);
rgb(599pt)=(0.477167,0.74888,0.48818);
rgb(600pt)=(0.479969,0.748931,0.486995);
rgb(601pt)=(0.482752,0.748982,0.485821);
rgb(602pt)=(0.485534,0.749033,0.484647);
rgb(603pt)=(0.488316,0.749084,0.483473);
rgb(604pt)=(0.491081,0.7491,0.482316);
rgb(605pt)=(0.493838,0.7491,0.481168);
rgb(606pt)=(0.496595,0.7491,0.480019);
rgb(607pt)=(0.499351,0.7491,0.47887);
rgb(608pt)=(0.502069,0.74912,0.477722);
rgb(609pt)=(0.504775,0.749145,0.476573);
rgb(610pt)=(0.50748,0.749171,0.475424);
rgb(611pt)=(0.510186,0.749196,0.474276);
rgb(612pt)=(0.512892,0.7492,0.47317);
rgb(613pt)=(0.515598,0.7492,0.472073);
rgb(614pt)=(0.518303,0.7492,0.470975);
rgb(615pt)=(0.521009,0.7492,0.469877);
rgb(616pt)=(0.523644,0.749176,0.46878);
rgb(617pt)=(0.526273,0.749151,0.467682);
rgb(618pt)=(0.528902,0.749125,0.466585);
rgb(619pt)=(0.531531,0.7491,0.465487);
rgb(620pt)=(0.53416,0.749074,0.464415);
rgb(621pt)=(0.536789,0.749049,0.463343);
rgb(622pt)=(0.539418,0.749023,0.462271);
rgb(623pt)=(0.542043,0.748998,0.461199);
rgb(624pt)=(0.544621,0.748972,0.460127);
rgb(625pt)=(0.547199,0.748947,0.459055);
rgb(626pt)=(0.549777,0.748921,0.457983);
rgb(627pt)=(0.55235,0.748891,0.45692);
rgb(628pt)=(0.554903,0.74884,0.455899);
rgb(629pt)=(0.557456,0.748789,0.454878);
rgb(630pt)=(0.560008,0.748738,0.453857);
rgb(631pt)=(0.562554,0.748687,0.452829);
rgb(632pt)=(0.565081,0.748636,0.451783);
rgb(633pt)=(0.567608,0.748585,0.450736);
rgb(634pt)=(0.570135,0.748534,0.449689);
rgb(635pt)=(0.572653,0.748474,0.44866);
rgb(636pt)=(0.575155,0.748397,0.447665);
rgb(637pt)=(0.577656,0.748321,0.446669);
rgb(638pt)=(0.580158,0.748244,0.445674);
rgb(639pt)=(0.582649,0.748168,0.444678);
rgb(640pt)=(0.585125,0.748091,0.443683);
rgb(641pt)=(0.587601,0.748014,0.442687);
rgb(642pt)=(0.590077,0.747938,0.441692);
rgb(643pt)=(0.59254,0.747861,0.440709);
rgb(644pt)=(0.59499,0.747785,0.439739);
rgb(645pt)=(0.597441,0.747708,0.438769);
rgb(646pt)=(0.599891,0.747632,0.437799);
rgb(647pt)=(0.602311,0.747555,0.436814);
rgb(648pt)=(0.604711,0.747478,0.435819);
rgb(649pt)=(0.60711,0.747402,0.434823);
rgb(650pt)=(0.60951,0.747325,0.433828);
rgb(651pt)=(0.611909,0.747232,0.432867);
rgb(652pt)=(0.614308,0.747129,0.431922);
rgb(653pt)=(0.616708,0.747027,0.430978);
rgb(654pt)=(0.619107,0.746925,0.430033);
rgb(655pt)=(0.621487,0.746823,0.429089);
rgb(656pt)=(0.623861,0.746721,0.428144);
rgb(657pt)=(0.626235,0.746619,0.4272);
rgb(658pt)=(0.628609,0.746517,0.426256);
rgb(659pt)=(0.630962,0.746393,0.425311);
rgb(660pt)=(0.63331,0.746266,0.424367);
rgb(661pt)=(0.635658,0.746138,0.423422);
rgb(662pt)=(0.638007,0.746011,0.422478);
rgb(663pt)=(0.640332,0.745906,0.421557);
rgb(664pt)=(0.642654,0.745804,0.420638);
rgb(665pt)=(0.644977,0.745702,0.419719);
rgb(666pt)=(0.6473,0.7456,0.4188);
rgb(667pt)=(0.649623,0.745472,0.417881);
rgb(668pt)=(0.651946,0.745345,0.416962);
rgb(669pt)=(0.654268,0.745217,0.416043);
rgb(670pt)=(0.656587,0.745089,0.415124);
rgb(671pt)=(0.658859,0.744962,0.414205);
rgb(672pt)=(0.661131,0.744834,0.413286);
rgb(673pt)=(0.663402,0.744707,0.412368);
rgb(674pt)=(0.665674,0.744579,0.411453);
rgb(675pt)=(0.667946,0.744451,0.410559);
rgb(676pt)=(0.670218,0.744324,0.409666);
rgb(677pt)=(0.672489,0.744196,0.408773);
rgb(678pt)=(0.674755,0.744062,0.407879);
rgb(679pt)=(0.677001,0.743909,0.406986);
rgb(680pt)=(0.679247,0.743756,0.406092);
rgb(681pt)=(0.681494,0.743603,0.405199);
rgb(682pt)=(0.68374,0.743458,0.404306);
rgb(683pt)=(0.685986,0.74333,0.403412);
rgb(684pt)=(0.688232,0.743203,0.402519);
rgb(685pt)=(0.690479,0.743075,0.401626);
rgb(686pt)=(0.692704,0.742937,0.400732);
rgb(687pt)=(0.694899,0.742784,0.399839);
rgb(688pt)=(0.697094,0.742631,0.398945);
rgb(689pt)=(0.699289,0.742477,0.398052);
rgb(690pt)=(0.701497,0.742324,0.397171);
rgb(691pt)=(0.703718,0.742171,0.396303);
rgb(692pt)=(0.705939,0.742018,0.395435);
rgb(693pt)=(0.708159,0.741865,0.394568);
rgb(694pt)=(0.710351,0.741712,0.3937);
rgb(695pt)=(0.71252,0.741559,0.392832);
rgb(696pt)=(0.71469,0.741405,0.391964);
rgb(697pt)=(0.71686,0.741252,0.391096);
rgb(698pt)=(0.719029,0.741082,0.390228);
rgb(699pt)=(0.721199,0.740904,0.38936);
rgb(700pt)=(0.723369,0.740725,0.388492);
rgb(701pt)=(0.725538,0.740546,0.387625);
rgb(702pt)=(0.727708,0.740386,0.386757);
rgb(703pt)=(0.729878,0.740233,0.385889);
rgb(704pt)=(0.732047,0.74008,0.385021);
rgb(705pt)=(0.734217,0.739927,0.384153);
rgb(706pt)=(0.736366,0.739753,0.383285);
rgb(707pt)=(0.73851,0.739574,0.382417);
rgb(708pt)=(0.740654,0.739395,0.38155);
rgb(709pt)=(0.742798,0.739217,0.380682);
rgb(710pt)=(0.744919,0.739038,0.379837);
rgb(711pt)=(0.747038,0.738859,0.378995);
rgb(712pt)=(0.749156,0.738681,0.378152);
rgb(713pt)=(0.751275,0.738502,0.37731);
rgb(714pt)=(0.753394,0.738323,0.376442);
rgb(715pt)=(0.755512,0.738145,0.375574);
rgb(716pt)=(0.757631,0.737966,0.374707);
rgb(717pt)=(0.75975,0.737789,0.373841);
rgb(718pt)=(0.761868,0.737636,0.372998);
rgb(719pt)=(0.763987,0.737483,0.372156);
rgb(720pt)=(0.766105,0.73733,0.371314);
rgb(721pt)=(0.76822,0.737169,0.370471);
rgb(722pt)=(0.770313,0.736965,0.369629);
rgb(723pt)=(0.772406,0.73676,0.368786);
rgb(724pt)=(0.774499,0.736556,0.367944);
rgb(725pt)=(0.776592,0.736358,0.367096);
rgb(726pt)=(0.778686,0.736179,0.366228);
rgb(727pt)=(0.780779,0.736001,0.36536);
rgb(728pt)=(0.782872,0.735822,0.364492);
rgb(729pt)=(0.784957,0.735643,0.363632);
rgb(730pt)=(0.787024,0.735465,0.36279);
rgb(731pt)=(0.789092,0.735286,0.361948);
rgb(732pt)=(0.791159,0.735107,0.361105);
rgb(733pt)=(0.793227,0.734929,0.360263);
rgb(734pt)=(0.795295,0.73475,0.359421);
rgb(735pt)=(0.797362,0.734571,0.358578);
rgb(736pt)=(0.79943,0.734392,0.357736);
rgb(737pt)=(0.801485,0.734214,0.356881);
rgb(738pt)=(0.803527,0.734035,0.356014);
rgb(739pt)=(0.805569,0.733856,0.355146);
rgb(740pt)=(0.807611,0.733678,0.354278);
rgb(741pt)=(0.809668,0.733499,0.353424);
rgb(742pt)=(0.811735,0.73332,0.352582);
rgb(743pt)=(0.813803,0.733142,0.35174);
rgb(744pt)=(0.81587,0.732963,0.350897);
rgb(745pt)=(0.817921,0.732784,0.350038);
rgb(746pt)=(0.819963,0.732606,0.349171);
rgb(747pt)=(0.822005,0.732427,0.348303);
rgb(748pt)=(0.824047,0.732248,0.347435);
rgb(749pt)=(0.826071,0.73207,0.346567);
rgb(750pt)=(0.828087,0.731891,0.345699);
rgb(751pt)=(0.830104,0.731712,0.344831);
rgb(752pt)=(0.83212,0.731534,0.343963);
rgb(753pt)=(0.834158,0.731355,0.343095);
rgb(754pt)=(0.8362,0.731176,0.342228);
rgb(755pt)=(0.838242,0.730998,0.34136);
rgb(756pt)=(0.840284,0.730819,0.340492);
rgb(757pt)=(0.842303,0.73064,0.339624);
rgb(758pt)=(0.84432,0.730462,0.338756);
rgb(759pt)=(0.846336,0.730283,0.337888);
rgb(760pt)=(0.848353,0.730104,0.33702);
rgb(761pt)=(0.850369,0.729926,0.336153);
rgb(762pt)=(0.852386,0.729747,0.335285);
rgb(763pt)=(0.854402,0.729568,0.334417);
rgb(764pt)=(0.856419,0.729391,0.333546);
rgb(765pt)=(0.858435,0.729238,0.332627);
rgb(766pt)=(0.860452,0.729085,0.331708);
rgb(767pt)=(0.862468,0.728932,0.330789);
rgb(768pt)=(0.864481,0.728778,0.329874);
rgb(769pt)=(0.866472,0.728625,0.32898);
rgb(770pt)=(0.868463,0.728472,0.328087);
rgb(771pt)=(0.870454,0.728319,0.327194);
rgb(772pt)=(0.872445,0.728166,0.326295);
rgb(773pt)=(0.874436,0.728013,0.325376);
rgb(774pt)=(0.876427,0.727859,0.324457);
rgb(775pt)=(0.878418,0.727706,0.323538);
rgb(776pt)=(0.880417,0.727561,0.322619);
rgb(777pt)=(0.882433,0.727433,0.3217);
rgb(778pt)=(0.88445,0.727306,0.320781);
rgb(779pt)=(0.886466,0.727178,0.319862);
rgb(780pt)=(0.888463,0.72705,0.318933);
rgb(781pt)=(0.890429,0.726923,0.317989);
rgb(782pt)=(0.892394,0.726795,0.317044);
rgb(783pt)=(0.894359,0.726668,0.3161);
rgb(784pt)=(0.896337,0.726552,0.315132);
rgb(785pt)=(0.898328,0.72645,0.314136);
rgb(786pt)=(0.900319,0.726348,0.313141);
rgb(787pt)=(0.90231,0.726246,0.312145);
rgb(788pt)=(0.904301,0.726158,0.31115);
rgb(789pt)=(0.906292,0.726081,0.310154);
rgb(790pt)=(0.908283,0.726005,0.309159);
rgb(791pt)=(0.910274,0.725928,0.308163);
rgb(792pt)=(0.912249,0.725851,0.307151);
rgb(793pt)=(0.914214,0.725775,0.30613);
rgb(794pt)=(0.91618,0.725698,0.305109);
rgb(795pt)=(0.918145,0.725622,0.304088);
rgb(796pt)=(0.920111,0.7256,0.303031);
rgb(797pt)=(0.922076,0.7256,0.301959);
rgb(798pt)=(0.924041,0.7256,0.300886);
rgb(799pt)=(0.926007,0.7256,0.299814);
rgb(800pt)=(0.927972,0.7256,0.298722);
rgb(801pt)=(0.929938,0.7256,0.297624);
rgb(802pt)=(0.931903,0.7256,0.296527);
rgb(803pt)=(0.933869,0.7256,0.295429);
rgb(804pt)=(0.935812,0.725668,0.294264);
rgb(805pt)=(0.937752,0.725744,0.29309);
rgb(806pt)=(0.939692,0.725821,0.291916);
rgb(807pt)=(0.941632,0.725897,0.290741);
rgb(808pt)=(0.943571,0.726023,0.289518);
rgb(809pt)=(0.945511,0.726151,0.288293);
rgb(810pt)=(0.947451,0.726278,0.287068);
rgb(811pt)=(0.949389,0.726411,0.285839);
rgb(812pt)=(0.951278,0.726641,0.284537);
rgb(813pt)=(0.953167,0.72687,0.283235);
rgb(814pt)=(0.955056,0.7271,0.281933);
rgb(815pt)=(0.956938,0.72734,0.280622);
rgb(816pt)=(0.958776,0.727646,0.279243);
rgb(817pt)=(0.960614,0.727952,0.277865);
rgb(818pt)=(0.962451,0.728259,0.276486);
rgb(819pt)=(0.964273,0.728597,0.275086);
rgb(820pt)=(0.966034,0.729057,0.273606);
rgb(821pt)=(0.967795,0.729516,0.272126);
rgb(822pt)=(0.969557,0.729976,0.270645);
rgb(823pt)=(0.971288,0.730473,0.269135);
rgb(824pt)=(0.972947,0.73106,0.267552);
rgb(825pt)=(0.974606,0.731647,0.265969);
rgb(826pt)=(0.976265,0.732234,0.264387);
rgb(827pt)=(0.977857,0.732879,0.262785);
rgb(828pt)=(0.979338,0.733619,0.261151);
rgb(829pt)=(0.980818,0.734359,0.259518);
rgb(830pt)=(0.982299,0.735099,0.257884);
rgb(831pt)=(0.983697,0.73591,0.256227);
rgb(832pt)=(0.984999,0.736803,0.254542);
rgb(833pt)=(0.986301,0.737697,0.252858);
rgb(834pt)=(0.987603,0.73859,0.251173);
rgb(835pt)=(0.988753,0.739566,0.249474);
rgb(836pt)=(0.989774,0.740613,0.247764);
rgb(837pt)=(0.990795,0.741659,0.246054);
rgb(838pt)=(0.991816,0.742706,0.244344);
rgb(839pt)=(0.992677,0.743816,0.242681);
rgb(840pt)=(0.993443,0.744965,0.241048);
rgb(841pt)=(0.994209,0.746114,0.239414);
rgb(842pt)=(0.994975,0.747262,0.23778);
rgb(843pt)=(0.995578,0.748465,0.236165);
rgb(844pt)=(0.996114,0.74969,0.234557);
rgb(845pt)=(0.99665,0.750915,0.232949);
rgb(846pt)=(0.997186,0.752141,0.231341);
rgb(847pt)=(0.997562,0.753386,0.229813);
rgb(848pt)=(0.997893,0.754637,0.228307);
rgb(849pt)=(0.998225,0.755887,0.226801);
rgb(850pt)=(0.998557,0.757138,0.225295);
rgb(851pt)=(0.998711,0.758433,0.223856);
rgb(852pt)=(0.998839,0.759735,0.222426);
rgb(853pt)=(0.998966,0.761037,0.220997);
rgb(854pt)=(0.999094,0.762339,0.219567);
rgb(855pt)=(0.999076,0.763641,0.218186);
rgb(856pt)=(0.99905,0.764942,0.216808);
rgb(857pt)=(0.999025,0.766244,0.21543);
rgb(858pt)=(0.998995,0.767546,0.214054);
rgb(859pt)=(0.998868,0.768848,0.212752);
rgb(860pt)=(0.99874,0.77015,0.21145);
rgb(861pt)=(0.998613,0.771451,0.210149);
rgb(862pt)=(0.998473,0.772756,0.208856);
rgb(863pt)=(0.998243,0.774083,0.207631);
rgb(864pt)=(0.998014,0.775411,0.206405);
rgb(865pt)=(0.997784,0.776738,0.20518);
rgb(866pt)=(0.997539,0.77806,0.20396);
rgb(867pt)=(0.997232,0.779362,0.20276);
rgb(868pt)=(0.996926,0.780664,0.201561);
rgb(869pt)=(0.99662,0.781966,0.200361);
rgb(870pt)=(0.996299,0.783268,0.199168);
rgb(871pt)=(0.995942,0.784569,0.197994);
rgb(872pt)=(0.995584,0.785871,0.19682);
rgb(873pt)=(0.995227,0.787173,0.195646);
rgb(874pt)=(0.994842,0.788475,0.19449);
rgb(875pt)=(0.994408,0.789777,0.193367);
rgb(876pt)=(0.993974,0.791078,0.192244);
rgb(877pt)=(0.99354,0.79238,0.191121);
rgb(878pt)=(0.993083,0.793671,0.190021);
rgb(879pt)=(0.992598,0.794947,0.188949);
rgb(880pt)=(0.992113,0.796223,0.187877);
rgb(881pt)=(0.991628,0.797499,0.186805);
rgb(882pt)=(0.99113,0.798789,0.185732);
rgb(883pt)=(0.990619,0.800091,0.18466);
rgb(884pt)=(0.990109,0.801393,0.183588);
rgb(885pt)=(0.989598,0.802695,0.182516);
rgb(886pt)=(0.989072,0.803996,0.18146);
rgb(887pt)=(0.988536,0.805298,0.180413);
rgb(888pt)=(0.988,0.8066,0.179367);
rgb(889pt)=(0.987464,0.807902,0.17832);
rgb(890pt)=(0.98691,0.809186,0.177291);
rgb(891pt)=(0.986349,0.810462,0.17627);
rgb(892pt)=(0.985787,0.811738,0.175249);
rgb(893pt)=(0.985226,0.813015,0.174228);
rgb(894pt)=(0.984644,0.814311,0.173207);
rgb(895pt)=(0.984057,0.815613,0.172186);
rgb(896pt)=(0.98347,0.816914,0.171165);
rgb(897pt)=(0.982883,0.818216,0.170144);
rgb(898pt)=(0.982296,0.819518,0.169145);
rgb(899pt)=(0.981709,0.82082,0.16815);
rgb(900pt)=(0.981122,0.822122,0.167154);
rgb(901pt)=(0.980535,0.823423,0.166159);
rgb(902pt)=(0.979947,0.824725,0.165163);
rgb(903pt)=(0.97936,0.826027,0.164168);
rgb(904pt)=(0.978773,0.827329,0.163172);
rgb(905pt)=(0.978186,0.828631,0.162177);
rgb(906pt)=(0.977599,0.829932,0.161181);
rgb(907pt)=(0.977012,0.831234,0.160186);
rgb(908pt)=(0.976425,0.832536,0.15919);
rgb(909pt)=(0.975838,0.833841,0.158195);
rgb(910pt)=(0.975251,0.835168,0.157199);
rgb(911pt)=(0.974664,0.836495,0.156204);
rgb(912pt)=(0.974077,0.837823,0.155208);
rgb(913pt)=(0.973489,0.83915,0.154213);
rgb(914pt)=(0.972902,0.840477,0.153217);
rgb(915pt)=(0.972315,0.841805,0.152222);
rgb(916pt)=(0.971728,0.843132,0.151226);
rgb(917pt)=(0.971155,0.844466,0.150224);
rgb(918pt)=(0.970619,0.845819,0.149203);
rgb(919pt)=(0.970083,0.847172,0.148182);
rgb(920pt)=(0.969547,0.848525,0.147161);
rgb(921pt)=(0.96902,0.849886,0.14614);
rgb(922pt)=(0.968509,0.851265,0.145119);
rgb(923pt)=(0.967999,0.852643,0.144098);
rgb(924pt)=(0.967488,0.854022,0.143077);
rgb(925pt)=(0.967,0.855411,0.142056);
rgb(926pt)=(0.966541,0.856815,0.141035);
rgb(927pt)=(0.966081,0.858219,0.140014);
rgb(928pt)=(0.965622,0.859623,0.138992);
rgb(929pt)=(0.965189,0.86104,0.137945);
rgb(930pt)=(0.96478,0.862469,0.136873);
rgb(931pt)=(0.964372,0.863899,0.135801);
rgb(932pt)=(0.963963,0.865328,0.134729);
rgb(933pt)=(0.96357,0.866773,0.133657);
rgb(934pt)=(0.963187,0.868228,0.132585);
rgb(935pt)=(0.962805,0.869683,0.131513);
rgb(936pt)=(0.962422,0.871138,0.130441);
rgb(937pt)=(0.962091,0.87261,0.129351);
rgb(938pt)=(0.961785,0.874091,0.128253);
rgb(939pt)=(0.961478,0.875571,0.127156);
rgb(940pt)=(0.961172,0.877052,0.126058);
rgb(941pt)=(0.960885,0.878571,0.124961);
rgb(942pt)=(0.960605,0.880103,0.123863);
rgb(943pt)=(0.960324,0.881634,0.122765);
rgb(944pt)=(0.960043,0.883166,0.121668);
rgb(945pt)=(0.959849,0.884719,0.120549);
rgb(946pt)=(0.95967,0.886276,0.119426);
rgb(947pt)=(0.959491,0.887833,0.118302);
rgb(948pt)=(0.959313,0.88939,0.117179);
rgb(949pt)=(0.959181,0.890995,0.116032);
rgb(950pt)=(0.959054,0.892603,0.114884);
rgb(951pt)=(0.958926,0.894211,0.113735);
rgb(952pt)=(0.958799,0.895819,0.112587);
rgb(953pt)=(0.958748,0.897453,0.111464);
rgb(954pt)=(0.958697,0.899086,0.110341);
rgb(955pt)=(0.958646,0.90072,0.109217);
rgb(956pt)=(0.958602,0.902359,0.108089);
rgb(957pt)=(0.958628,0.904043,0.106915);
rgb(958pt)=(0.958653,0.905728,0.105741);
rgb(959pt)=(0.958679,0.907413,0.104567);
rgb(960pt)=(0.958718,0.909102,0.103393);
rgb(961pt)=(0.95882,0.910812,0.102219);
rgb(962pt)=(0.958922,0.912522,0.101044);
rgb(963pt)=(0.959024,0.914232,0.0998703);
rgb(964pt)=(0.959153,0.915962,0.0986961);
rgb(965pt)=(0.959357,0.917749,0.0975219);
rgb(966pt)=(0.959561,0.919536,0.0963477);
rgb(967pt)=(0.959765,0.921323,0.0951736);
rgb(968pt)=(0.959996,0.923118,0.0939907);
rgb(969pt)=(0.960277,0.924931,0.092791);
rgb(970pt)=(0.960557,0.926743,0.0915913);
rgb(971pt)=(0.960838,0.928555,0.0903916);
rgb(972pt)=(0.961151,0.930378,0.0891919);
rgb(973pt)=(0.961509,0.932216,0.0879922);
rgb(974pt)=(0.961866,0.934054,0.0867925);
rgb(975pt)=(0.962223,0.935892,0.0855928);
rgb(976pt)=(0.96262,0.937768,0.0843802);
rgb(977pt)=(0.963053,0.939683,0.083155);
rgb(978pt)=(0.963487,0.941597,0.0819297);
rgb(979pt)=(0.963921,0.943512,0.0807045);
rgb(980pt)=(0.964415,0.945426,0.0794643);
rgb(981pt)=(0.964951,0.947341,0.0782135);
rgb(982pt)=(0.965487,0.949255,0.0769628);
rgb(983pt)=(0.966023,0.951169,0.075712);
rgb(984pt)=(0.966594,0.953118,0.0744441);
rgb(985pt)=(0.967181,0.955083,0.0731679);
rgb(986pt)=(0.967768,0.957049,0.0718916);
rgb(987pt)=(0.968355,0.959014,0.0706153);
rgb(988pt)=(0.96898,0.960999,0.0693006);
rgb(989pt)=(0.969619,0.96299,0.0679733);
rgb(990pt)=(0.970257,0.964981,0.0666459);
rgb(991pt)=(0.970895,0.966972,0.0653186);
rgb(992pt)=(0.971554,0.968984,0.0639486);
rgb(993pt)=(0.972218,0.971001,0.0625703);
rgb(994pt)=(0.972882,0.973017,0.0611919);
rgb(995pt)=(0.973545,0.975034,0.0598135);
rgb(996pt)=(0.974232,0.97705,0.058318);
rgb(997pt)=(0.974922,0.979067,0.056812);
rgb(998pt)=(0.975611,0.981083,0.055306);
rgb(999pt)=(0.9763,0.9831,0.0538)
}
}

\usetikzlibrary{arrows,shapes,positioning}
\usetikzlibrary{decorations.markings}
\tikzstyle arrowstyle=[scale=1]
\tikzstyle directed=[postaction={decorate,decoration={markings,
		mark=at position .65 with {\arrow[arrowstyle]{stealth}}}}]
\tikzstyle reverse directed=[postaction={decorate,decoration={markings,
		mark=at position .65 with {\arrowreversed[arrowstyle]{stealth};}}}]

\newcommand{\secref}[1]{Section~\ref{#1}}
\newcommand{\thmref}[1]{Theorem~\ref{#1}}
\newcommand{\exref}[1]{Example~\ref{#1}}

\newcommand{\lemref}[1]{Lemma~\ref{#1}}

\newcommand{\rmref}[1]{Remark~\ref{#1}}

\newcommand{\figref}[1]{Figure~\ref{#1}}
\newcommand{\tabref}[1]{Table~\ref{#1}}

\newcommand{\assref}[1]{Assumption~\ref{#1}}

\newcommand{\abs}[1]{\ensuremath{\left| #1 \right|}}

\newcommand{\R}{\mathbb{R}}
\newcommand{\Rhoch}[2]{\ensuremath{\R^{{#1} \times {#2}}}}

\newcommand{\scattering}{\ensuremath{\sigma_s}}
\newcommand{\absorption}{\ensuremath{\sigma_a}}

\newcommand{\distribution}[1][ ]{\ensuremath{\psi_{#1}}}

\newcommand{\basisind}{\ensuremath{i}} 

\newcommand{\moments}[1]{\ensuremath{u^{ ({#1}) }}} 

\newcommand{\multiplierscomp}[1]{\ensuremath{\alpha_{#1}}} 

\newcommand{\Domain}{\ensuremath{\mathcal{X}}} 
\newcommand{\dimension}{\ensuremath{d}} 
\newcommand{\timevar}{\ensuremath{t}} 

\newcommand{\PN}[1][\momentorder]{\ensuremath{\text{P}_{#1}}}
\newcommand{\MN}[1][\momentorder]{\ensuremath{\text{M}_{#1}}}

\newcommand{\flux}{\ensuremath{\mathbf{f}}}

\newcommand{\eigenvalue}{\ensuremath{\lambda}}

\newcommand{\entropy}{\ensuremath{h}} 

\newcommand{\x}{\ensuremath{x}}

\newcommand{\dx}{\partial_{\x}}

\newcommand{\dt}{\partial_\timevar}
\newcommand{\de}{\partial}

\newcommand{\dtstepsize}{\ensuremath{\Delta \timevar}}

\newcommand{\ncells}{\ensuremath{I}}

\newcommand{\timeind}{\ensuremath{n}}

\newcommand{\cellind}{\ensuremath{i}}

\newcommand{\cell}[1]{\ensuremath{C_{#1}}}

\newcommand{\limitervariable}{\ensuremath{\theta}}
\newcommand{\limitertolerance}{\ensuremath{\varepsilon}}

\newcommand{\density}{\ensuremath{\rho}}

\newcommand{\velocity}{\ensuremath{v}}
\newcommand{\energy}{\ensuremath{E}}
\newcommand{\pressure}{\ensuremath{p}}
\newcommand{\eulergamma}{\ensuremath{\gamma}}

\newcommand{\momentum}{\ensuremath{m}}

\newcommand{\uncertainty}{\ensuremath{\xi}}
\newcommand{\xiPDF}{\ensuremath{f_\Xi}}
\newcommand{\xiBasisPoly}[1]{\ensuremath{\phi_{#1}}}
\newcommand{\SGsumIndex}{\ensuremath{k}}
\newcommand{\SGeqIndex}{\ensuremath{j}}
\newcommand{\SGapproach}{\ensuremath{\sum_{\SGsumIndex=0}^\SGtruncorder \solution_\SGsumIndex \xiBasisPoly{\SGsumIndex}}}
\newcommand{\xiPDFdxi}{\ensuremath{\xiPDF \mathrm{d}\uncertainty}}
\newcommand{\intRS}{\ensuremath{\int_{\randomSpace}}}
\newcommand{\SGtruncorder}{\ensuremath{K}}
\newcommand{\nbxiBasisPoly}{\ensuremath{\SGtruncorder+1}}

\newcommand{\nbxiQuadNodes}{\ensuremath{Q}}
\newcommand{\xiQuadWeights}[1]{\ensuremath{w_{#1}}}

\newcommand{\xiQuadIndex}{\ensuremath{q}}

\newcommand{\SGmomentvec}{\ensuremath{\bU}}
\newcommand{\SGflux}{\ensuremath{\bF}}

\newcommand{\realizableSet}{\ensuremath{\mathcal{R}}}
\newcommand{\numRealizableSet}{\ensuremath{\realizableSet^{Q}}}
\newcommand{\realizablePara}{\ensuremath{b}}
\newcommand{\limitersolution}[1][\theta]{\ensuremath{\solution^{#1}}}

\newcommand{\sampleSpace}{\ensuremath{\Omega}}
\newcommand{\sigmaAlgebra}{\ensuremath{\mathcal{F}}}
\newcommand{\probabilityMeasure}{\ensuremath{\mathcal{P}}}
\newcommand{\randomSpace}{\ensuremath{\Gamma}}

\newcommand{\solution}{\ensuremath{\bu}}

\newcommand{\splittingpara}{\ensuremath{a}}
\newcommand{\xr}{\ensuremath{\x_R}}
\newcommand{\xl}{\ensuremath{\x_L}}
\newcommand{\densityl}{\ensuremath{\density_L}}
\newcommand{\momentuml}{\ensuremath{\momentum_L}}
\newcommand{\energyl}{\ensuremath{\energy_L}}
\newcommand{\densityr}{\ensuremath{\density_R}}
\newcommand{\momentumr}{\ensuremath{\momentum_R}}
\newcommand{\energyr}{\ensuremath{\energy_R}}
\newcommand{\energydiff}{\ensuremath{\epsilon}}
\newcommand{\cflnb}{\ensuremath{c}}

\newcommand{\entropicVar}{\ensuremath{\boldsymbol{\lambda}}}
\newcommand{\entropicVarSum}{\ensuremath{\boldsymbol{\Lambda}}}
\newcommand{\InvGradEntropy}{\ensuremath{\left(\nabla \entropy\right)^{-1}}}
\newcommand{\IPMMtruncorder}{\ensuremath{K}}
\newcommand{\IPMMsumIndex}{\ensuremath{k}}
\newcommand{\IPMMapproach}{\ensuremath{\sum_{\IPMMsumIndex=0}^\IPMMtruncorder \entropicVar_{\IPMMsumIndex} \xiBasisPoly{\IPMMsumIndex}}}
\newcommand{\InvVar}{\ensuremath{\Lambda}}
\newcommand{\InvVara}{\ensuremath{\Lambda_0}}
\newcommand{\InvVarb}{\ensuremath{\Lambda_1}}
\newcommand{\InvVarc}{\ensuremath{\Lambda_2}}

\newcommand{\nbsamples}{\ensuremath{K}}

\newcommand{\hyperbolicity}{hyperbolicity}
\newcommand{\hyperbolic}{admissible}
\newcommand{\hypset}{hyperbolicity set}

\newcommand{\hSG}{hSG}

\pgfplotscreateplotcyclelist{color parula}{%
	royalblue!70,every mark/.append style={solid,line width = 0pt,fill=royalblue!60!black},mark=ball\\%
	goldenrod!50!yelloworange,every mark/.append style={solid,fill=goldenrod!30!black},mark=*\\%
	green,every mark/.append style={black,fill=yellowgreen},mark=diamond*\\%
	bluegreen,every mark/.append style={fill=black},mark=triangle*\\%
	royalblue!70,densely dashed,every mark/.append style={solid,line width = 0pt,fill=royalblue!60!black},mark=ball\\%
	goldenrod!50!yelloworange,densely dashed,every mark/.append style={solid,black,fill=goldenrod},mark=diamond*\\%
	yellowgreen,densely dashed,every mark/.append style={solid,fill=yellowgreen!60!royalblue},mark=square*, mark size= 1pt\\%
	bluegreen,densely dashed,mark size = 1.5pt,every mark/.append style={solid,fill=white},mark=otimes*\\
	royalblue,densely dashed,mark size = 1.5pt,every mark/.append style={solid,fill=royalblue!30!black},mark=pentagon*\\
	goldenrod!40!yelloworange,every mark/.append style={solid,fill=goldenrod!30!black},mark=|\\%
	}

\pgfplotscreateplotcyclelist{color louisa}{%
	blue,every mark/.append style={solid,line width = 0pt,fill=royalblue!60!black},mark=ball\\%
	red,every mark/.append style={solid,fill=orangered!30!black},mark=*\\%
	green,every mark/.append style={black,fill=yellowgreen},mark=diamond*\\%
	goldenrod!50!yelloworange,every mark/.append style={fill=black},mark=triangle*\\%
	royalblue!70,densely dashed,every mark/.append style={solid,line width = 0pt,fill=royalblue!60!black},mark=ball\\%
	goldenrod!50!yelloworange,densely dashed,every mark/.append style={solid,black,fill=goldenrod},mark=diamond*\\%
	yellowgreen,densely dashed,every mark/.append style={solid,fill=yellowgreen!60!royalblue},mark=square*, mark size= 1pt\\%
	bluegreen,densely dashed,mark size = 1.5pt,every mark/.append style={solid,fill=white},mark=otimes*\\
	royalblue,densely dashed,mark size = 1.5pt,every mark/.append style={solid,fill=royalblue!30!black},mark=pentagon*\\
	goldenrod!40!yelloworange,every mark/.append style={solid,fill=goldenrod!30!black},mark=|\\%
}

\usepackage[colorinlistoftodos,obeyFinal]{todonotes}

\usepackage{media9}

\begin{document}

\begin{abstract}
Uncertainty Quantification through stochastic spectral methods is rising in popularity. We derive a modification of the classical stochastic Galerkin method, that ensures the \hyperbolicity{} of the underlying hyperbolic system of partial differential equations. The modification is done using a suitable ``slope'' limiter, based on similar ideas in the context of kinetic moment models. We apply the resulting modified stochastic Galerkin method to the compressible Euler equations and the $\MN[1]$ model of radiative transfer. Our numerical results show that it can compete with other UQ methods like the intrusive polynomial moment method while being computationally inexpensive and easy to implement.
\end{abstract}
\maketitle

\noindent


\section{Introduction}
High-order schemes in space and time have recently gained attention in the context of hyperbolic systems of conservation laws. The deterministic physical system (e.g., the shallow water equations for tsunami propagation \cite{vreugdenhil2013numerical} or the minimum entropy models in the context of radiative transfer \cite{BruHol01,Bru02,Schneider2014,schneider2016moment,Chidyagwai2017}) has to be solved accurately to ensure that especially waves in the far field are well captured.

On the other hand, non-deterministic effects may influence the validity of the accurate approximation of the deterministic system. Such effects can result from measurement uncertainties for physical parameters in the equations or the initial state of the (deterministic) system (e.g., due to inaccuracies of measuring devices or the need to obtain a large amount of data in real time (tsunami propagation)). 
Uncertainty Quantification (UQ) methods predict the behavior of physical systems with non-deterministic inputs, when the model parameters, boundary or initial conditions are not available exactly.
Especially in the context of hyperbolic systems of equations, whose solutions often develop discontinuities, non-deterministic effects have a huge impact on the behavior of the solution \cite{Poette2009,Kusch2015a}. 

Another side effect of the non-smooth nature of hyperbolic systems is that classical stochastic approaches to the solution of the non-deterministic system, like Monte Carlo, are less efficient for Uncertainty Quantification \cite{Poette2009}. Many variants of Polynomial Chaos (PC) methods \cite{ghanem2003stochastic,Wiener1938} have been successfully applied to various applications (see, e.g., \cite{acharjee2006uncertainty,lucor2007stochastic,xiu2002stochastic}). However, the naive usage of the stochastic Galerkin (SG) approach for more complex problems typically fails \cite{Poette2009,Abgrall2017} since the polynomial expansion of discontinuous data leads to huge oscillations (also known as Gibbs phenomenon). In some cases, the resulting SG system is not even hyperbolic. 

A huge amount of work has been spent to remove these disadvantages. One rather novel approach \cite{Chertock2015} aims at removing the loss of hyperbolicity with an operator splitting approach, applying the SG method to a sequence of linear systems and scalar nonlinear equations, for which the SG method is known to be hyperbolic. Unfortunately, the resulting approximation still does not maintain the hyperbolicity of the original system, as we will discuss later. \\
Another approach, the intrusive polynomial moment method (IPMM), bounds the oscillations of the Gibbs phenomenon by expanding the stochastic solution not in the conserved variables but in so-called entropic variables \cite{Poette2009}, which is well known in the radiative transfer community as minimum entropy models. The resulting generalized Polynomial Chaos (gPC) \cite{Wiener1938} system is hyperbolic and has good approximation properties but requires to solve (typically) expensive nonlinear systems in every space-time cell. Furthermore, it is necessary that the system possesses a strictly convex entropy function, which has to be known beforehand to define the entropic variables. 

For a more detailed overview of the various UQ methods, we refer to \cite{Abgrall2017,Poette2009} and references therein.

The mechanism that causes the classical stochastic Galerkin method to loose hyperbolicity has been observed before in the context of high-order discontinuous-Galerkin schemes for hyperbolic systems, especially for moment systems (see, e.g., \cite{Alldredge2015,Schneider2015b,Chidyagwai2017,Olbrant2012,Zhang2010,Zhang2011b,Schneider2016a}). We use a similar technique, a ``slope limiter'', to ``dampen'' the Gibbs oscillations in the stochastic expansion in such a way that the resulting system is always hyperbolic. Additionally, the method can be applied whenever the domain of hyperbolicity is explicitly available, even when there is no known entropy for the system.

The rest of the paper is organized as follows. In \secref{sec:model} we describe our model problem, the classical stochastic Galerkin approach and define the domain of \hyperbolicity{}. Our modification of SG is stated in \secref{sec:limiter}, where we also prove its \hyperbolicity{} preservation. In \secref{sec:systems} we give an overview of the investigated hyperbolic systems of conservation laws, namely the compressible Euler equations and the $\MN[1]$ model of radiative transfer, and embed them into our framework. \secref{sec:othermethods} is devoted to formulating the SG operator splitting approach and the intrusive polynomial moment method for our two benchmark systems, which we extensively investigate numerically in \secref{sec:results}.
\section{Modeling Uncertainties}\label{sec:model}
We consider a system of hyperbolic conservation laws of the form
\begin{equation}\label{conslaw}\dt \solution + \dx \flux (\solution) = 0,\end{equation}
with flux function $\flux({\solution}): \R^{\dimension} \rightarrow \R^{\dimension}$ in one spatial dimension $\x\in\Domain\subset\R$. Additionally, the solution 
$$\solution = \solution(\timevar,\x,\uncertainty):\R_+ \times \R \times \sampleSpace \rightarrow \R^{\dimension}$$
is depending on a one-dimensional random variable $\uncertainty$ with probability space $(\sampleSpace,\sigmaAlgebra,\probabilityMeasure)$.  We denote the random space of this uncertainty by $\randomSpace := \uncertainty(\sampleSpace)$ and its probability density function by $\xiPDF(\uncertainty) : \randomSpace \rightarrow \R_+$.

Moreover, we assume that the uncertainty is introduced via the initial conditions, namely
$$\solution(\timevar=0,\x,\uncertainty) = \solution^0(\x,\uncertainty).$$
Such a situation may arise for every deterministic system of equations where the initial state, e.g., of an experiment, cannot be determined exactly.

The system \eqref{conslaw} is solved using the generalized Polynomial Chaos (gPC) theory \cite{Wiener1938} which will be explained in the following stochastic Galerkin approach.

\subsection{Stochastic Galerkin}
The idea of stochastic Galerkin (SG) is to discretize the probability space $\sampleSpace$ of $\uncertainty$. According to the theory of gPC in \cite{Wiener1938} we can write $\solution$ as the following expansion
\begin{equation}\label{SGinfinitesum}\solution(\timevar,\x,\uncertainty) = \sum_{\SGsumIndex=0}^{\infty} \solution_{\SGsumIndex}(\timevar,\x) \xiBasisPoly{\SGsumIndex}(\uncertainty(\omega)),\end{equation}
with deterministic coefficients $\solution_\SGsumIndex$ and where $ \xiBasisPoly{\SGsumIndex}(\uncertainty(\cdot)): \sampleSpace \rightarrow \R$ describe orthonormal polynomials with respect to the inner product of the underlying distribution
\begin{equation} \label{eq:scalarproduct}\langle \mathbf{h}(\uncertainty),\,\bg(\uncertainty)\rangle = \int_{\omega\in\sampleSpace} \! \mathbf{h}(\uncertainty(\omega))\bg(\uncertainty(\omega)) \mathrm{d}\probabilityMeasure(\omega) = \int_{\uncertainty\in\randomSpace}\!\mathbf{h}(\uncertainty)\bg(\uncertainty)\xiPDF(\uncertainty)\mathrm{d}{\uncertainty}.\end{equation}
Hence $(\xiBasisPoly{\SGsumIndex}(\uncertainty))_{\SGsumIndex=0,\ldots,\infty}$ form a basis of $L_2(\sampleSpace)$.
\begin{example}\label{Legendrebasis}
	For a uniformly distributed random variable  $\uncertainty\sim\mathcal{U}(-1,1)$ we have $\xiPDF = \frac{1}{2}$ and $\xiBasisPoly{\SGsumIndex}$ is given by the $\SGsumIndex$-th normalized Legendre polynomial of order $\SGsumIndex$. If the distribution is Gaussian, we instead use normalized Hermite polynomials and adapt $\xiPDF$ to the corresponding probability density.
\end{example}
The stochastic Galerkin approach now approximates the solution $\solution$ of \eqref{conslaw} by truncating the infinite sum \eqref{SGinfinitesum} at finite order $\SGtruncorder$, i.e.,
\begin{equation}\label{SGapproach}\solution \approx \sum_{\SGsumIndex=0}^{\SGtruncorder} \solution_{\SGsumIndex}(\timevar,\x) \xiBasisPoly{\SGsumIndex}(\uncertainty),\end{equation}
which is converging to \eqref{SGinfinitesum} as $\SGtruncorder\rightarrow\infty$ by the Cameron-Martin theorem \cite{Cameron1947}.
The polynomial moments $\solution_{\SGsumIndex}$ of $\solution$ are given by a Galerkin projection onto the random space
\begin{equation}\label{SGmoment}\solution_{\SGsumIndex} (\timevar,\x) = \int_{\randomSpace} \solution(\timevar,\x,\uncertainty) \xiBasisPoly{\SGsumIndex}(\uncertainty) \xiPDF \mathrm{d}\uncertainty.\end{equation}
We plug the ansatz \eqref{SGapproach} into the system of conservation laws \eqref{conslaw} and project the result onto the space spanned by the basis polynomials up to order $\SGtruncorder$. Then we obtain 
$$\dt \intRS \left(\SGapproach\right)\!\xiBasisPoly{\SGeqIndex} \xiPDFdxi + \dx \intRS \flux\!\left(\SGapproach\right) \!\xiBasisPoly{\SGeqIndex}\xiPDFdxi = 0, \qquad \SGeqIndex = 0,\ldots,\SGtruncorder.$$
Using the orthonormality of the basis functions yields the following stochastic Galerkin system
\begin{equation}\label{SGsystem}
\dt \underbrace{\begin{pmatrix}
\solution_0\\ \vdots\\ \solution_\SGtruncorder
\end{pmatrix}}_{=:\SGmomentvec}
+ \dx \underbrace{\begin{pmatrix}
\intRS \flux\! \left( \SGapproach \right)\! \xiBasisPoly{0} \xiPDFdxi\\ \vdots\\ \intRS \flux\! \left( \SGapproach \right)\! \xiBasisPoly{\SGtruncorder} \xiPDFdxi
\end{pmatrix}}_{=:\SGflux\left(\SGmomentvec\right)}
= 0,
\end{equation}
with $\SGmomentvec,~\SGflux\!\left(\SGmomentvec\right)\in\R^{\dimension(\nbxiBasisPoly)}$. The Jacobian matrix of this model reads
\begin{equation}\label{SGJacobian}\frac{\de \SGflux}{\de \SGmomentvec} = \begin{pmatrix}
\hat{\bF}_{00} & \cdots & \hat{\bF}_{0\SGtruncorder}\\
\vdots & & \vdots\\
\hat{\bF}_{\SGtruncorder 0} & \cdots & \hat{\bF}_{\SGtruncorder\SGtruncorder}
\end{pmatrix} ~ \in \Rhoch{\dimension(\nbxiBasisPoly)}{\dimension(\nbxiBasisPoly)},
\end{equation}
where
$$\hat{\bF}_{ji} =\intRS \frac{\de \flux}{\de \solution}\! \left(\SGapproach\right)\! \xiBasisPoly{j}\xiBasisPoly{i}\xiPDFdxi ~\in\Rhoch{\dimension}{\dimension}.$$
The stochastic Galerkin system \eqref{SGsystem} can then be solved by any finite volume method. However, as shown in \cite{Poette2009}, the system \eqref{SGsystem} is not necessarily hyperbolic anymore, making the straight-forward implementation of classical finite volume methods difficult. 
\begin{rem}\label{rem:symhyp}
	For symmetric hyperbolic systems \eqref{conslaw}, the Jacobian \eqref{SGJacobian} is also symmetric and the stochastic Galerkin system \eqref{SGsystem} thus is hyperbolic as well.
\end{rem}
We introduce a ``slope-limited'' version of the stochastic Galerkin scheme, which maintains hyperbolicity, in the next section.

The expected value of $\solution$ is given by its moment of first order. We assume $\xiBasisPoly{0}(\uncertainty)=1$ since $\int_{\randomSpace}\! \xiPDFdxi=1$ and obtain the expected value and standard deviation using the orthonormality of the basis polynomials
\begin{align}
\mathbb{E}(\solution) &\approx  \int_{\randomSpace}  \SGapproach \, \xiPDFdxi 
= \sum_{\SGsumIndex=0}^\SGtruncorder \solution_\SGsumIndex  \int_{\randomSpace} \!\xiBasisPoly{\SGsumIndex}\xiBasisPoly{0}\, \xiPDFdxi
= \solution_0, \label{SGEV}\\[0.2cm]
s(\solution) &\approx  \sqrt{\int_{\randomSpace} \! \left(  \SGapproach\right)^2  \! \!\xiPDFdxi - \solution_0^2}
= \sqrt{\sum_{\SGsumIndex=0}^\SGtruncorder \solution_\SGsumIndex^2 \int_\randomSpace \!  \xiBasisPoly{\SGsumIndex}^2  \, \xiPDFdxi- \solution_0^2 \int_{\randomSpace} \!  \xiBasisPoly{0}^2  \, \xiPDFdxi} = \sqrt{\sum_{\SGsumIndex=1}^\SGtruncorder \solution_\SGsumIndex^2}.\label{SGSD}
\end{align} 

\subsection{Hyperbolicity}
Usually, the solution of our system of equations \eqref{conslaw} has to fulfill certain physical properties. For example, a density should always be nonnegative. For hierarchical models like the Euler equations, the system normally looses hyperbolicity for unphysical states.

\begin{definition}
	We call the set 
	\begin{align*}
		\realizableSet := \left\{\solution\in\R^\dimension~\big|~ \eigenvalue_\basisind\left(\solution\right)\in\R~\forall\,\basisind=1,\ldots,\dimension,~\text{ where } \eigenvalue_\basisind \text{ is the } \basisind\text{-th eigenvalue of } \frac{\partial\flux(\solution)}{\partial\solution}\right\}
	\end{align*}
	the \textbf{\hypset{}}. We call every solution vector $\solution\in\realizableSet$ \textbf{\hyperbolic{}}.
\end{definition}
\begin{assumption}
	\label{ass:RealizableSet}
	In the following we always assume that the \hypset{} $\realizableSet$ is open and convex.
\end{assumption}
Throughout our numerical analysis, we approximate the integrals in the SG system \eqref{SGsystem} by a $\nbxiQuadNodes$-point Gauss-Legendre quadrature with respect to the uncertainty $\uncertainty$ and the inner product \eqref{eq:scalarproduct}. Hence we write 
\begin{equation}\label{quadrature}\intRS \bg(\uncertainty) \xiPDF(\uncertainty) \mathrm{d}\uncertainty = \sum_{\xiQuadIndex = 1}^{\nbxiQuadNodes} 
\bg(\uncertainty_\xiQuadIndex) \xiPDF (\uncertainty_{\xiQuadIndex}) \xiQuadWeights{\xiQuadIndex},\end{equation}
where $\nbxiQuadNodes$ defines the number of quadrature nodes and $\xiQuadWeights{\xiQuadIndex}$ the quadrature weights. In the stochastic Galerkin approach we approximate $\solution\approx\SGapproach$. Evaluating this expression at the different quadrature nodes, we obtain the stochastic numerically \hypset{}
\begin{equation}
\label{numRealSet} \numRealizableSet := \left\{\SGmomentvec\in\R^{\dimension(\nbxiBasisPoly)} ~\bigg|~\sum_{\SGsumIndex=0}^{\SGtruncorder} \solution_\SGsumIndex \xiBasisPoly{\SGsumIndex}(\uncertainty_\xiQuadIndex) \in \realizableSet ~~ \forall\,\xiQuadIndex= 1,\ldots,\nbxiQuadNodes \right\},
\end{equation}
consisting of those moment vectors that lead to an \hyperbolic{} stochastic Galerkin approximation of the solution for each quadrature node $\uncertainty_\xiQuadIndex$~\footnote{If no quadrature rule would be needed, the definition of the stochastic \hypset{} would be the same except that $\sum_{\SGsumIndex=0}^{\SGtruncorder} \solution_\SGsumIndex \xiBasisPoly{\SGsumIndex}(\uncertainty)\in\realizableSet$ for all $\uncertainty\in\randomSpace$.}.

\section{Hyperbolicity-preserving  stochastic Galerkin scheme (\hSG{})}\label{sec:limiter}
In this section, we develop a \hyperbolicity{}-preserving variant of the stochastic Galerkin scheme, applying a slope limiter to the SG polynomial \eqref{SGapproach} that point-wisely shifts the solution into the numerically \hypset{}. For simplicity, we use the classical Lax-Friedrichs scheme \cite{Toro2009} for the space-time discretization and show that it
 preserves \hyperbolicity{} of the stochastic Galerkin system \eqref{SGsystem} under a CFL-type condition\footnote{Similar results can be shown for other monotone first-order schemes.}.

\subsection{Hyperbolicity preservation of the Lax-Friedrichs scheme}
At first, we determine under which conditions the Lax-Friedrichs scheme preserves \hyperbolicity{} of the zeroth moment of the solution (basically its deterministic part). Here, we assume that the result of the previous time step is lying in $\numRealizableSet$ as this will be ensured by the \hyperbolicity{} limiter.
In order to apply a finite volume scheme, we divide the domain $\Domain=[\x_L,\x_R]\subset\R$ into $\ncells$ cells $\cell{\cellind}= [\x_{\cellind-\frac{1}{2}},\x_{\cellind+\frac{1}{2}}]$ with $\x_{\cellind\pm\frac{1}{2}} = \x_{\cellind}\pm\frac{\Delta\x }{2}$ and $\Delta\x = \frac{\x_R-\x_L}{\ncells}$. We denote the current time step by $\timevar_\timeind$. Then, one time step of the Lax-Friedrichs method for each moment reads 
\begin{align*}\solution_\SGeqIndex(\timevar_{\timeind+1},\x_{\cellind}) 
&= \frac{1}{2}\big(\solution_\SGeqIndex(\timevar_\timeind, \x_{\cellind+1}) + \solution_\SGeqIndex(\timevar_\timeind, \x_{\cellind-1}) \big) \\
&- \frac{\dtstepsize}{2\Delta\x}  \intRS \flux\!\left(\sum_{\SGsumIndex=0}^\SGtruncorder \solution_\SGsumIndex(\timevar_\timeind, \x_{\cellind+1}) \xiBasisPoly{\SGsumIndex}(\uncertainty)	\right) \!\xiBasisPoly{\SGeqIndex}(\uncertainty) \xiPDF(\uncertainty) \mathrm{d}\uncertainty\\
&+ \frac{\dtstepsize}{2\Delta\x} \intRS \flux\!\left(\sum_{\SGsumIndex=0}^\SGtruncorder \solution_\SGsumIndex(\timevar_\timeind, \x_{\cellind-1}) \xiBasisPoly{\SGsumIndex}(\uncertainty)	\right)\! \xiBasisPoly{\SGeqIndex}(\uncertainty) \xiPDF(\uncertainty) \mathrm{d}\uncertainty, \qquad \SGeqIndex=0,\ldots,\SGtruncorder.
\end{align*}
Using the quadrature defined in \eqref{quadrature}, we obtain
\begin{equation}\label{LaxFriedrichs}
\left.
\begin{alignedat}{1}
\hspace*{1.2cm}
\solution_\SGeqIndex(\timevar_{\timeind+1},\x_{\cellind}) &=\frac{1}{2}\big(\solution_\SGeqIndex(\timevar_\timeind, \x_{\cellind+1}) + \solution_\SGeqIndex(\timevar_\timeind, \x_{\cellind-1}) \big)\\
&- \frac{\dtstepsize}{2\Delta\x} \sum_{\xiQuadIndex = 1}^\nbxiQuadNodes  \flux\!\left(\sum_{\SGsumIndex=0}^\SGtruncorder \solution_\SGsumIndex(\timevar_\timeind, \x_{\cellind+1}) \xiBasisPoly{\SGsumIndex}(\uncertainty_\xiQuadIndex)	\right)\! \xiBasisPoly{\SGeqIndex}(\uncertainty_\xiQuadIndex) \xiPDF(\uncertainty_\xiQuadIndex) \xiQuadWeights{\xiQuadIndex}\hspace*{1.2cm}\\
&+ \frac{\dtstepsize}{2\Delta\x} \sum_{\xiQuadIndex = 1}^\nbxiQuadNodes  \flux\!\left(\sum_{\SGsumIndex=0}^\SGtruncorder \solution_\SGsumIndex(\timevar_\timeind, \x_{\cellind-1}) \xiBasisPoly{\SGsumIndex}(\uncertainty_\xiQuadIndex)	\right)\! \xiBasisPoly{\SGeqIndex}(\uncertainty_\xiQuadIndex) \xiPDF(\uncertainty_\xiQuadIndex) \xiQuadWeights{\xiQuadIndex} . 
\end{alignedat}
\right\}
\end{equation}
The CFL condition for this scheme reads
$$\Delta \timevar = \cflnb \,\frac{\Delta\x}{\eigenvalue_{\max}},$$
where $\cflnb\leq1$ is the CFL number and $\eigenvalue_{\max}$ describes the absolute maximal eigenvalue of the Jacobian \eqref{SGJacobian}.

In order to formulate the desired theorem, we need the following assumption.
\begin{assumption}
	There exists a constant $0<\realizablePara\in\R$ such that 
	\begin{equation}
	\label{realizationPara}
	\realizablePara = \sup\left\{ \tilde{\realizablePara}>0 ~\big|~ \solution\pm\tilde{\realizablePara}\flux(\solution)\in \realizableSet~~\forall \,\solution\in\realizableSet \right\}. 
	\end{equation}
\end{assumption}
\begin{remark}
	In practice, we can calculate the value $\realizablePara$ for every time step $\timevar_\timeind$ from the set 
	\begin{equation}
	\label{eq:realizationPara2}
	\left\{ \tilde{\realizablePara}>0 ~\Big|~ \solution\pm\tilde{\realizablePara}\flux(\solution)\in \realizableSet,~~ \realizableSet\ni\solution = \sum_{\SGsumIndex=0}^\SGtruncorder \solution_\SGsumIndex(\timevar_\timeind, \x_{\cellind}) \xiBasisPoly{\SGsumIndex}(\uncertainty_\xiQuadIndex),~\cellind=1,\ldots,\ncells,~\xiQuadIndex=1,\ldots,\nbxiQuadNodes \right\}.
	\end{equation}
	Under the assumption that all the given point-values of $\solution$ are \hyperbolic{}, the existence of such a maximal $\realizablePara$ follows immediately from the convexity and openness of $\realizableSet$ (see \assref{ass:RealizableSet}).
\end{remark}

We will investigate the calculation of this parameter for two model equations in the next section. 

With this at hand, we conclude the following theorem.

\begin{theorem}\label{LFrealizprev}	
	Let 
	\begin{align}
	\label{eq:assLFrealizprev}
	\SGmomentvec(\timevar_\timeind,\x_{\cellind})\in\numRealizableSet,\quad \forall\,\cellind=1,\ldots,\ncells,
	\end{align}
	and 
	\begin{align}
	\label{eq:CFL}
	\Delta\timevar < \min\left( \realizablePara\Delta\x ,\, \frac{\Delta \x}{\eigenvalue_{\max}}\right),
	\end{align}
	where $\realizablePara$ is determined from \eqref{eq:realizationPara2}.	
	Then, one time step of the Lax-Friedrichs scheme \eqref{LaxFriedrichs} preserves the \hyperbolicity{} of the zeroth moment $\solution_0$, i.e. $\solution_0(\timevar_{\timeind+1},\x_{\cellind}) \in \realizableSet,~\forall\, \cellind=1,\ldots,\ncells.$
\end{theorem}
\begin{proof}
	At first, we rewrite $\solution_0$ using \eqref{SGmoment}, quadrature and the stochastic Galerkin approach \eqref{SGapproach}
	\begin{align}
	\label{eq:u0quadrature}
	\solution_0(\timevar,\x)= \sum_{\xiQuadIndex = 1}^\nbxiQuadNodes \left( \sum_{\SGsumIndex=0}^\SGtruncorder \solution_{\SGsumIndex}(\timevar,\x) \xiBasisPoly{\SGsumIndex}(\uncertainty_\xiQuadIndex)\right) \! \xiBasisPoly{0}(\uncertainty_\xiQuadIndex) \xiPDF(\uncertainty_\xiQuadIndex) \xiQuadWeights{\xiQuadIndex}.
	\end{align}
	Inserting this into the Lax-Friedrichs method \eqref{LaxFriedrichs} and reordering some terms yields
	\begin{align*}
	\solution_0(\timevar_{\timeind+1},\x_{\cellind}) 
    &=\frac{1}{2} \sum_{\xiQuadIndex = 1}^\nbxiQuadNodes \left( \sum_{\SGsumIndex=0}^\SGtruncorder \big(\solution_{\SGsumIndex}(\timevar_\timeind,\x_{\cellind+1}) + \solution_{\SGsumIndex}(\timevar_\timeind,\x_{\cellind-1}) \big) \xiBasisPoly{\SGsumIndex}(\uncertainty_\xiQuadIndex)\right) \! \xiBasisPoly{0}(\uncertainty_\xiQuadIndex) \xiPDF(\uncertainty_\xiQuadIndex) \xiQuadWeights{\xiQuadIndex} \\
	&- \frac{\dtstepsize}{2\Delta\x} \sum_{\xiQuadIndex = 1}^\nbxiQuadNodes  \flux\!\left(\sum_{\SGsumIndex=0}^\SGtruncorder \solution_\SGsumIndex(\timevar_\timeind, \x_{\cellind+1}) \xiBasisPoly{\SGsumIndex}(\uncertainty_\xiQuadIndex)	\right)\! \xiBasisPoly{0}(\uncertainty_\xiQuadIndex) \xiPDF(\uncertainty_\xiQuadIndex) \xiQuadWeights{\xiQuadIndex}\\
    &+ \frac{\dtstepsize}{2\Delta\x} \sum_{\xiQuadIndex = 1}^\nbxiQuadNodes  \flux\!\left(\sum_{\SGsumIndex=0}^\SGtruncorder \solution_\SGsumIndex(\timevar_\timeind, \x_{\cellind-1}) \xiBasisPoly{\SGsumIndex}(\uncertainty_\xiQuadIndex)	\right)\! \xiBasisPoly{0}(\uncertainty_\xiQuadIndex) \xiPDF(\uncertainty_\xiQuadIndex) \xiQuadWeights{\xiQuadIndex},\\[0.2cm]
	&= \frac{1}{2} \sum_{\xiQuadIndex = 1}^\nbxiQuadNodes \left( \sum_{\SGsumIndex=0}^\SGtruncorder \solution_\SGsumIndex(\timevar_\timeind, \x_{\cellind+1}) \xiBasisPoly{\SGsumIndex}(\uncertainty_\xiQuadIndex) - \frac{\Delta \timevar}{\Delta \x} \flux\!\left(\sum_{\SGsumIndex=0}^\SGtruncorder \solution_\SGsumIndex(\timevar_\timeind, \x_{\cellind+1}) \xiBasisPoly{\SGsumIndex}(\uncertainty_\xiQuadIndex)\right)\right) \!\xiBasisPoly{0}(\uncertainty_\xiQuadIndex) \xiPDF(\uncertainty_\xiQuadIndex) \xiQuadWeights{\xiQuadIndex}\\
	&+ \frac{1}{2} \sum_{\xiQuadIndex = 1}^\nbxiQuadNodes \left( \sum_{\SGsumIndex=0}^\SGtruncorder \solution_\SGsumIndex(\timevar_\timeind, \x_{\cellind-1}) \xiBasisPoly{\SGsumIndex}(\uncertainty_\xiQuadIndex) + \frac{\Delta \timevar}{\Delta \x} \flux\!\left(\sum_{\SGsumIndex=0}^\SGtruncorder \solution_\SGsumIndex(\timevar_\timeind, \x_{\cellind-1}) \xiBasisPoly{\SGsumIndex}(\uncertainty_\xiQuadIndex)\right)\right) \!\xiBasisPoly{0}(\uncertainty_\xiQuadIndex) \xiPDF(\uncertainty_\xiQuadIndex) \xiQuadWeights{\xiQuadIndex}.
	\end{align*}
	By assumption \eqref{eq:assLFrealizprev}, we have
	$$\sum_{\SGsumIndex=0}^\SGtruncorder \solution_\SGsumIndex(\timevar_\timeind, \x_{\cellind\pm1}) \xiBasisPoly{\SGsumIndex}(\uncertainty_\xiQuadIndex)\in\realizableSet, \quad \forall \,\xiQuadIndex=1,\ldots,\nbxiQuadNodes.$$
	Under the CFL condition \eqref{eq:CFL}, we see that $\frac{\Delta\timevar}{\Delta\x}< \realizablePara$, such that \eqref{eq:realizationPara2} guarantees
	$$\sum_{\SGsumIndex=0}^\SGtruncorder \solution_\SGsumIndex(\timevar_\timeind, \x_{\cellind\pm1}) \xiBasisPoly{\SGsumIndex}(\uncertainty_\xiQuadIndex) \mp \frac{\Delta \timevar}{\Delta \x} \flux\!\left(\sum_{\SGsumIndex=0}^\SGtruncorder \solution_\SGsumIndex(\timevar_\timeind, \x_{\cellind\pm1}) \xiBasisPoly{\SGsumIndex}(\uncertainty_\xiQuadIndex)\right) \in \realizableSet, \quad \forall \,\xiQuadIndex=1,\ldots,\nbxiQuadNodes.$$
	Hence, $\solution_0(\timevar_{\timeind+1},\x_{\cellind})$ is a convex combination of \hyperbolic{} quantities and therefore \hyperbolic{} by \assref{ass:RealizableSet}~(convexity of $\realizableSet$, positivity of $\xiBasisPoly{0}$ and $\xiPDF$).
\end{proof}
\begin{remark}\label{rem:CFLmod}
	We need the strictly smaller sign in \eqref{eq:CFL} since the supremum in \eqref{realizationPara} could place $\solution\pm\realizablePara\flux(\solution)$ onto the boundary of $\realizableSet$, which is assumed to be open. In our computations we set
	\begin{align*}
		\Delta\timevar = \cflnb \cdot \min\left(\realizablePara\Delta\x , \, \frac{\Delta \x}{\eigenvalue_{\max}}\right),
	\end{align*}
	with a CFL number $\cflnb = 0.95<1$.
\end{remark}

\subsection{Limiting}
To obtain a \hyperbolicity{}-preserving numerical scheme, we need to ensure assumption \eqref{eq:assLFrealizprev} in \thmref{LFrealizprev}. This can be done using a \hyperbolicity{} limiter which is based on the ideas in \cite{Zhang2010,Alldredge2015,Schneider2016a}.

We define the slope-limited SG polynomial as
$$\limitersolution(\timevar_\timeind,\x_{\cellind},\uncertainty_\xiQuadIndex) := \limitervariable \,\solution_0(\timevar_\timeind,\x_{\cellind}) + (1-\limitervariable) \sum_{\SGsumIndex=0}^\SGtruncorder \solution_\SGsumIndex(\timevar_\timeind, \x_{\cellind}) \xiBasisPoly{\SGsumIndex}(\uncertainty_\xiQuadIndex) 
= \solution_0(\timevar_\timeind,\x_{\cellind}) + (1-\limitervariable) \sum_{\SGsumIndex=1}^\SGtruncorder \solution_\SGsumIndex(\timevar_\timeind, \x_{\cellind}) \xiBasisPoly{\SGsumIndex}(\uncertainty_\xiQuadIndex).$$
The variable $\limitervariable$ limits the SG polynomial towards the (assumed to be) \hyperbolic{} zeroth moment $\solution_0$. 
The case $\limitervariable=0$ coincides with the unlimited solution and for $\limitervariable=1$ we have
$$\limitersolution[\limitervariable=1](\timevar_\timeind,\x_{\cellind},\uncertainty_\xiQuadIndex) = \solution_0(\timevar_\timeind,\x_{\cellind}),$$
which is supposed to be \hyperbolic{}. Because of this property and since $\realizableSet$ is convex, we can choose
$$\hat{\limitervariable}(\timevar_\timeind,\x_{\cellind}) := \inf\left\{ \tilde{\limitervariable}\in [0,1] ~\big|~ \limitersolution[\tilde{\limitervariable}](\timevar_\timeind,\x_{\cellind},\uncertainty_\xiQuadIndex)  \in \realizableSet~~ \forall\, \xiQuadIndex=1,\ldots,\nbxiQuadNodes\right\}.$$
Again, due to the openness of $\realizableSet$ and similarly to \rmref{rem:CFLmod}, we need to modify $\limitervariable$ slightly in order to avoid placing the solution onto the boundary (if the limiter was active). Therefore we use 
\begin{align*}
\limitervariable = \begin{cases}
\hat{\limitervariable}, & \text{ if } \hat{\limitervariable} = 0,\\
\min(\hat{\limitervariable}+\limitertolerance,1), & \text{ if } \hat{\limitervariable} > 0,
\end{cases}
\end{align*}
where $0<\limitertolerance = 10^{-10}$ should be chosen small enough to ensure that the approximation quality is not influenced significantly.

Finally, we replace the original moment vector $\SGmomentvec(\timevar_\timeind,\x_{\cellind})$ with the limited vector $\SGmomentvec^\limitervariable(\timevar_\timeind,\x_{\cellind})$ given by
$$\limitersolution_\SGsumIndex(\timevar_\timeind,\x_{\cellind}) = \begin{cases}
\solution_0(\timevar_\timeind,\x_{\cellind}), \qquad & \text{if } \SGsumIndex=0,\\
(1-\limitervariable) \solution_\SGsumIndex(\timevar_\timeind,\x_{\cellind}), & \text{if } \SGsumIndex>0,
\end{cases}\qquad \SGsumIndex=0,\ldots,\SGtruncorder,$$
where $\limitervariable$ is chosen separately for each $\timevar_\timeind$ and $\x_{\cellind}$, ensuring that $\SGmomentvec^\limitervariable(\timevar_\timeind,\x_{\cellind})\in\numRealizableSet$ in all space-time cells.
\section{Hyperbolic Model Problems}\label{sec:systems}
In the following, we describe the hyperbolic model systems that we will use to test the \hyperbolicity{}-preserving stochastic Galerkin method and separately derive the parameter $\realizablePara$ from \eqref{eq:realizationPara2}.

\subsection{$\MN[1]$ model of radiative transport} \label{sec:M1RT}
We consider the kinetic radiative transfer equation \cite{schneider2016moment,Levermore1996,Lewis-Miller-1984}
\begin{equation}\label{radtranskin}
\dt \distribution + \velocity\,\dx \distribution  + \absorption \,\distribution = \scattering\, \Big(\frac{1}{2}\int_{-1}^1 \!\distribution\, \mathrm{d}\velocity - \distribution\Big),
\end{equation}
where $\distribution=\distribution(\timevar,\x,\velocity,\uncertainty)\in\R$ describes a particle distribution depending on time, $\x\in\Domain\subset\R$, the velocity $\velocity\in [-1,\,1]$ and the uncertainty $\uncertainty$. The equation models the propagation and interaction of particles through and with a medium, affected by absorption and scattering. The material parameters are the absorption and scattering coefficient, denoted by $\absorption$  and $\scattering$, respectively.

We define the moments with respect to the velocity $\velocity$ as
$$\moments{\SGsumIndex} := \int_{-1}^1 \!\velocity^{\SGsumIndex}\distribution\, \mathrm{d}\velocity,$$
where the moments $\moments{0}$ and $\moments{1}$ describe the local particle density and the mean velocity, respectively.

A system of equations for those moments can be obtained by projecting \eqref{radtranskin} onto the velocity basis $(1,\velocity)$
\begin{equation}\label{radtrans}
\left.
\begin{alignedat}{2}
\hspace*{1cm}
\dt \moments{0} &+ \dx \moments{1} &&= - \absorption \moments{0} ,\\
\dt \moments{1} &+ \dx \moments{2}  &&=- \absorption  \moments{1} -\scattering \moments{1}.\hspace*{1cm}
\end{alignedat}
\right\}
\end{equation}
The unknown second moment $\moments{2}= \moments{2}\big(\moments{0},\moments{1}\big)$ is closed via the implicit relation \cite{Ker76,BruHol01,AniPenSam91,Min78}
\begin{align}
\label{u1u0}
\frac{\moments{1}}{\moments{0}} &= \coth(\multiplierscomp{1} ) - \frac{1}{\multiplierscomp{1} } \,,\\
\frac{\moments{2}}{\moments{0}} &= \frac{\int_{-1}^1\! \velocity^2\,\exp(\multiplierscomp{1} \velocity) \mathrm{d}\velocity}{\int_{-1}^1\! \exp(\multiplierscomp{1}  \velocity) \mathrm{d}\velocity} = \frac{4}{\multiplierscomp{1}  - \multiplierscomp{1}  \exp(2\multiplierscomp{1} )}- \frac{2}{\multiplierscomp{1} } + \frac{2}{\multiplierscomp{1} ^2} + 1.\label{u2u0}
\end{align}
The first equation \eqref{u1u0} cannot be solved analytically for $\multiplierscomp{1}$, however, we can use a tabulation (see, e.g., \cite{Frank2005}) or a numerical fit \cite{Ducl2010} to calculate \begin{align}
\label{eq:M1Tabulation}
\frac{\moments{2}}{\moments{0}}\left(\frac{\moments{1}}{\moments{0}}\right),\end{align} 
as long as $\frac{\moments{1}}{\moments{0}}\in [-1,1]$. In this way, we obtain $\moments{2}(\solution)$ and the resulting model is the $\MN[1]$ model for \eqref{radtranskin}.

 \begin{lemma}\label{rthyp}
	The $\MN[1]$ system of radiative transfer is hyperbolic and the absolute values of the eigenvalues are bounded by 1.
\end{lemma}
\begin{proof}
	See, e.g., \cite{Levermore1996,Alldredge2015,BruHol01}, \cite[p. 67, 68]{schneider2016moment}.
\end{proof}

The \hypset{} for the $\MN[1]$ model is given by \cite{BruHol01,Ker76,Curto1991,Chidyagwai2017}
$$\realizableSet = \left\{ \big(\moments{0}, \moments{1}\big) ~\Big| ~ \big|\moments{1}\big| \leq \moments{0} \right\}.$$
It arises from the fact, that the underlying ansatz $\distribution$ is positive, resulting in 
$$\big|\moments{1}\big| = \Big|\int_{-1}^{1} \!\velocity\,\distribution\,\mathrm{d}\velocity\Big| \overset{\distribution\geqslant0}{\leq} \int_{-1}^{1} \!|\velocity|\,\distribution\,\mathrm{d}\velocity \overset{\velocity\in[-1,1]}{\leq} \int_{-1}^{1}\! \distribution\,\mathrm{d}\velocity = \moments{0}.$$

Since \eqref{radtrans} has a nonzero source term and is therefore no conservation law, we need to assure that we can still apply \thmref{LFrealizprev}.
\begin{lemma}Let $\solution(\timevar_{\timeind+1},\x_{\cellind})\in\realizableSet$ and $$\Delta \timevar < \frac{1}{\absorption + \scattering},$$
    then
	$$\widetilde{\solution}(\timevar_{\timeind+1},\x_{\cellind}) = \solution(\timevar_{\timeind+1},\x_{\cellind}) + \Delta \timevar \begin{pmatrix}
	-\absorption\\ -\absorption - \scattering
	\end{pmatrix} \solution(\timevar_{\timeind+1},\x_{\cellind}) ~ \in\realizableSet.$$ 
\end{lemma} 
\begin{proof}
	The proof is given in \cite[Thm. 3.3]{Alldredge2015}.
\end{proof}
Hence, we take
$$\Delta\timevar = \min\left( \cflnb \cdot\realizablePara\Delta\x , \,\cflnb\,\frac{\Delta \x}{\eigenvalue_{\min}}, \,\frac{1}{\absorption + \scattering}\right).$$

The choice of the \hyperbolicity{} parameter $\realizablePara$ is stated in the following lemma.
\begin{lemma}
	Assume $\solution\in\realizableSet$ and $\realizablePara=1$. Then for every moment model of \eqref{radtranskin}, we have $\solution\pm\realizablePara\flux(\solution) \in \realizableSet$.
\end{lemma}
\begin{proof}
	The proof can be found in \cite{Olbrant2012} and \cite[Lemma 3.2]{Alldredge2015}.
\end{proof}

\subsection{Euler equations}\label{SecEuler}
The one-dimensional compressible Euler equations for the flow of an ideal gas are given by
\begin{equation}\label{eulereq}
\left.
\begin{alignedat}{2}
\hspace*{2cm}
\dt \density &+ \dx \momentum &&= 0 ,\\
\dt \momentum &+ \dx \left(\frac{\momentum^2}{\density} + \pressure\right)  &&=0,\\
\dt \energy &+ \dx \left( (\energy + \pressure) \,\frac{\momentum}{\density}\right) &&  = 0,\hspace*{2cm}
\end{alignedat}
\right\}
\end{equation}
where $\density$ describes the density, $\momentum$ the momentum and $\energy$ the energy of the gas. The three equations model the conservation of mass, momentum and energy. The pressure $\pressure$ reads
$$\pressure = (\eulergamma-1)\left(\energy - \frac{1}{2}\frac{\momentum^2}{\density}\right),$$
with the adiabatic constant $\eulergamma>1$. The three eigenvalues of the Euler equations \eqref{eulereq} are given by
\begin{equation}\label{eulereig}
\left\{\frac{\momentum}{\density}-\sqrt{\eulergamma\,\frac{\pressure}{\density}}\,, ~\frac{\momentum}{\density}\,,~
\frac{\momentum}{\density}+\sqrt{\eulergamma\,\frac{\pressure}{\density}}\,\right\}.
\end{equation}

The eigenvalues are real-valued (i.e., the system \eqref{eulereq} is hyperbolic) for positive densities and pressures. Thus we obtain the following \hypset{}
$$	\realizableSet = \left\{ \solution = \begin{pmatrix}
\density\\\momentum\\\energy
\end{pmatrix}\,\Bigg|~\density>0,~\pressure = (\eulergamma-1)\left(\energy - \frac{1}{2}\frac{\momentum^2}{\density}\right) >0\right\}.$$

We now compute the \hyperbolicity{} parameter $\realizablePara$.

\begin{lemma}
	Given a vector $\solution$ and the flux function $\flux(\solution)$ for the Euler equations, i.e., 
	\begin{align}
	\solution = \begin{pmatrix}
	\density\\\momentum\\\energy
	\end{pmatrix}\in\realizableSet, \quad \flux(\solution) = \begin{pmatrix}
	\momentum\\\frac{\momentum^2}{\density} + \pressure \\[0.1cm](\energy + \pressure) \frac{\momentum}{\density}
	\end{pmatrix}\!,
	\end{align}
	the quantities $\solution\pm\realizablePara\flux(\solution)$ satisfy $\solution\pm\realizablePara\flux(\solution)\in \realizableSet$ if and only if 
	\begin{align}
	\realizablePara&<\min\left(|\realizablePara_+|,\,|\realizablePara_-|,\,\frac{\density}{\abs{\momentum}}\right),\nonumber\\[0.2cm]
	\realizablePara_\pm &:= -\frac{2\density\big(2\momentum  + \sqrt{2\energy\density - \momentum^2} \pm \eulergamma\sqrt{2\energy\density - \momentum^2}\,\big)}
	{\momentum^2\big(\eulergamma^2 - 2\eulergamma + 5\big) - 2\energy\density\big(\eulergamma^2 - 2\eulergamma + 1\big)}\,.\label{eq:Eulerbpm}
	\end{align}		
\end{lemma}
\begin{proof}
   We define
	$$\widehat{\solution}_\pm = \begin{pmatrix}
	\hat{\density}_\pm\\[0.05cm]  \hat{\momentum}_\pm\\[0.05cm]  \hat{\energy}_\pm
	\end{pmatrix}:= \solution\pm\realizablePara\flux(\solution) = \begin{pmatrix}
	\density \pm \realizablePara\momentum \\[0.05cm]  \momentum \pm \realizablePara \pressure \\[0.05cm] \energy \pm \realizablePara(\energy + \pressure) \frac{\momentum}{\density}
	\end{pmatrix} .$$
	Then, we need to determine for which values of $\realizablePara$ we have $\widehat{\solution}_\pm\in\realizableSet$.
	The condition on the positivity of $\hat{\density}_\pm$ implies $\realizablePara<\frac{\density}{\abs{\momentum}}$, while the pressure term belonging to $\solution_+$ reads
	\begin{align*}
	\frac{\hat{\pressure}_+}{\eulergamma-1} &= \left(\hat{\energy}_+ - \frac{1}{2}\frac{\hat{\momentum}_+^2}{\hat{\density}_+}\right) = \energy-\frac{{\left(\momentum+\realizablePara\left(\left(\eulergamma-1\right)\left(\energy-\frac{\momentum^2}{2\density }\right)+\frac{\momentum^2}{\density }\right)\right)}^2}{2\left(\density +\realizablePara\momentum\right)}+\frac{\realizablePara\momentum\left(\energy+\left(\eulergamma-1\right)\left(\energy-\frac{\momentum^2}{2\density }\right)\right)}{\density }.
	\end{align*}
	Solving $\frac{\hat{\pressure}_+}{\eulergamma-1} =0$ for $\realizablePara$ leads to $\realizablePara_+$ in \eqref{eq:Eulerbpm}. Note that an analogous result can be obtained for $\hat{\pressure}_-$.
\end{proof}
\section{Other UQ methods}\label{sec:othermethods}
In this section, we present two additional methods that aim to solve hyperbolic systems of equations with uncertain initial data, in particular an operator splitting approach for stochastic Galerkin \cite{Chertock2015} and the intrusive polynomial moment method (IPMM) \cite{Poette2009}. We will compare those methods to the results of the \hyperbolicity{}-preserving stochastic Galerkin scheme.

\subsection{Operator splitting with stochastic Galerkin}
The operator splitting method, introduced for the compressible Euler equations in \cite{Chertock2015}, splits a given system into subsystems and subsequently solves each of them with stochastic Galerkin.

As denoted in \cite{Poette2009}, the stochastic Galerkin approximation of the Euler equations (or generally nonlinear systems of conservation laws) can loose global hyperbolicity. 
The basic idea of the operator splitting method presented in \cite{Chertock2015} is to subdivide the system of equations into scalar nonlinear equations and linear systems, for which the stochastic Galerkin discretization is known to produce hyperbolic systems (compare to \cite{Chertock2015}).

\subsubsection{Operator splitting for the Euler equations}
According to \cite{Chertock2015}, the three subsystems of the Euler equations \eqref{eulereq} are given by
\begin{equation}\label{Esub1}
\left.
\begin{alignedat}{2}
\hspace*{2.2cm}
\dt \density &+ \dx \momentum &&= 0 ,\\
\dt \momentum &+ \dx \big( (\eulergamma-1)\energy + \splittingpara\momentum \big)  &&=0 ,\hspace*{2.2cm}\\
\dt \energy & - \splittingpara \dx \energy && = 0,
\end{alignedat}
\right\}\\[0.3cm]
\end{equation}
\begin{equation}\label{Esub2}
\left.
\begin{alignedat}{2}
\hspace*{1.75cm}
&\dt \density &&= 0, \hspace*{2.2cm}\\
&\dt \momentum + \dx \left(\frac{3-\eulergamma}{2} \frac{\momentum^2}{\density}  - \splittingpara \momentum \right) && = 0, \\
&\dt \energy  &&=0,
\end{alignedat}
\right\}\\[0.3cm]
\end{equation}
\begin{equation}\label{Esub3}
\left.
\begin{alignedat}{2}
&\dt \density &&=0 ,\\
&\dt \momentum && = 0,\\
&\dt \energy + \dx \left( \frac{\momentum}{\density} \left( \eulergamma\energy - \frac{\eulergamma-1}{2} \frac{\momentum^2}{\density}\right)+\splittingpara\energy\right) &&=0.\hspace*{2.2cm}
\end{alignedat}
\right\}
\end{equation}

The first subsystem \eqref{Esub1} is linear hyperbolic with eigenvalues $\pm\splittingpara$ and 0, the other subsystems are scalar hyperbolic. The choice of the splitting parameter $\splittingpara$ should ensure that the convection coefficients in \eqref{Esub2} and \eqref{Esub3} do not change their signs and that
\begin{equation}\label{splittingeig}-|\splittingpara| \leq \frac{\momentum}{\density}-\sqrt{\eulergamma\,\frac{\pressure}{\density}} < \frac{\momentum}{\density}+\sqrt{\eulergamma\,\frac{\pressure}{\density}} \leq |\splittingpara|\end{equation}
holds. This yields
\begin{equation}\label{eq:eulersplittingpara}
\splittingpara = \pm \sup_{\x\in\Domain}\left(\max\left(\left|\frac{\momentum}{\density}\right|+\sqrt{\eulergamma\,\frac{\pressure}{\density}},~
\eulergamma\left|\frac{\momentum}{\density}\right|,~
(3-\eulergamma)\left|\frac{\momentum}{\density}\right|\right)\right),
\end{equation}
whereas the supremum is taken over all spatial values of $\solution$ in the current time step. To avoid asymmetries, the sign of $\splittingpara$ is alternated in every time step. On each of the subsystems \eqref{Esub1}--\eqref{Esub3} we then consecutively apply stochastic Galerkin. The CFL condition uses the largest eigenvalue of the three SG Jacobians \eqref{SGJacobian} belonging to the different subsystems.

In contrast to the claim in \cite{Chertock2015}, we show that this approach does not preserve the hyperbolicity of the original system \eqref{eulereq}. To this end, we present an example showing that for every step size $\Delta\timevar$, we can find admissible initial conditions so that the solution of subsystem \eqref{Esub1} is already violating the \hyperbolicity{} requirements of \eqref{eulereq}. 
\begin{example}\label{splittingex}
	Consider a cell center $\x$ with adjacent cell centers $\xl$ and $\xr$. We define the initial state in those cells as
	\begin{alignat*}{3}
	\density(\timevar=0,\xl,\uncertainty) &:= \densityl = 0.5, \qquad &&\density(\timevar=0,\xr,\uncertainty) &&:= \densityr = 0.4,\\
	\momentum(\timevar=0,\xl,\uncertainty) &:= \momentuml = 2, \qquad &&\momentum(\timevar=0,\xr,\uncertainty) &&:= \momentumr = 2,\\
	\energy(\timevar=0,\xl,\uncertainty) &:= \energyl = \frac{1}{2} \frac{\momentuml^2}{\densityl} + \energydiff, \qquad &&\energy(\timevar=0,\xr,\uncertainty) &&:= \energyr =  \frac{1}{2} \frac{\momentumr^2}{\densityr} +\energydiff,
	\end{alignat*}
	with a constant $\energydiff>0$ and positive pressure $\pressure = (\eulergamma-1)\energydiff$. Furthermore, we set $\eulergamma=1.4$ and the truncation order to $\SGtruncorder=0$. According to \eqref{eq:eulersplittingpara} and using $\densityr < \densityl$, the splitting parameter $\splittingpara$ is determined by
	$$\splittingpara= \max\left( \left|\frac{\momentumr}{\densityr}\right| + \sqrt{\frac{\eulergamma(\eulergamma-1)\energydiff}{\densityr}}\,, \,(3-\eulergamma)\left|\frac{\momentumr}{\densityr}\right|\right).$$
	For $\energydiff<\frac{9}{\eulergamma}=6.4286$, this maximum is attained at the first term. We therefore assume this property and deduce $a= 8.$ 
	Hence, the CFL-condition reads
	$$\cflnb := \frac{\max\left(|\splittingpara|,\, \big| (3-\eulergamma)\frac{\momentumr}{\densityr}-\splittingpara\big|\right) \Delta\timevar }{\Delta\x} = \frac{\splittingpara\,\Delta\timevar}{\Delta\x}.$$

	Calculating one Lax-Friedrichs step for the stochastic Galerkin system of \eqref{Esub1} with  $\frac{\Delta\timevar}{\Delta\x}=\frac{\cflnb}{\splittingpara}$, we obtain
	\begin{align*}
	\density_0(\Delta\timevar,\x) &= 0.45,\\
	\momentum_0(\Delta\timevar,\x) &= 0.05\,\cflnb+2,\\
	\energy_0(\Delta\timevar,\x) &= \energydiff - \cflnb + 4.5,
	\end{align*}
	yielding
	$$\pressure(\Delta\timevar,\x) = (\eulergamma-1) \left( \energydiff - \frac{1}{360}\,\cflnb^2 - \frac{11}{9}\,\cflnb + \frac{1}{18}\right).$$
	The pressure is negative for any $\energydiff$ with
	$$\energydiff< \frac{1}{360}\,\cflnb^2 + \frac{11}{9}\,\cflnb - \frac{1}{18},$$
	whereas the term on the right hand side of this inequality is positive for 
	\begin{equation}\label{eq:cflnb}\cflnb > 6\,\sqrt{1345}-220 \approx 0.0454.\end{equation} Hence, for every CFL number which is larger than \eqref{eq:cflnb} we can find $\energydiff>0$ leading to a negative pressure, and therefore to a solution outside of the \hypset{}. When $\densityr$ approaches $\densityl$, we can arbitrarily reduce \eqref{eq:cflnb}, e.g., setting $\densityl=0.499$ requires a CFL number $\cflnb \leq  4.0048\cdot10^{-4}$ to obtain an admissible update.
\end{example}

Altogether we have given an example showing that the operator splitting presented in \cite{Chertock2015} does not necessarily preserve \hyperbolicity{}. Each of the subsystems separately leads indeed a hyperbolic SG discretization, however, they have different \hyperbolicity{} sets and thus do not always give \hyperbolic{} solutions in terms of the original system. For a negative pressure, we can also not ensure \eqref{splittingeig} since complex values would occur. This might lead to oscillations due to a wrong CFL condition. 

\subsubsection{Operator splitting for the $\MN[1]$ model}
We apply the previously described splitting approach to the $\MN[1]$ model of radiative transfer from \secref{sec:M1RT}. Following the outline of \cite{Chertock2015}, we split \eqref{radtrans} into the following subsystems\\
\begin{equation}\label{sub1}
\left.
\begin{alignedat}{2}
\hspace*{1.1cm}
\dt \moments{0} &+ \dx \moments{1} + \splittingpara\,\dx \moments{0}&&= - \absorption\moments{0} ,\\
\dt \moments{1}  &- \splittingpara\,\dx \moments{1}   &&=- \absorption \moments{1}  -\scattering \moments{1} ,\hspace*{1cm}
\end{alignedat}
\right\}\\[0.3cm]
\end{equation}
\begin{equation}\label{sub2}
\left.
\begin{alignedat}{2}
\hspace*{2.45cm}
&\dt \moments{0}- \splittingpara\,\dx \moments{0}&&= 0, \hspace*{3.45cm}\\
&\dt \moments{1}  &&=0,
\end{alignedat}
\right\}\\[0.3cm]
\end{equation}
\begin{equation}\label{sub3}
\left.
\begin{alignedat}{2}
\hspace*{1cm}
&\dt \moments{0} &&=0 ,\\
&\dt \moments{1}  + \dx \moments{2} + \splittingpara\,\dx \moments{1} &&=0.\hspace*{3.475cm}
\end{alignedat}
\right\}
\end{equation}

The first subsystem reduces to the linear terms of the original system. We therefore omit $\dx \moments{2}$ and obtain \eqref{sub1}, which is linear hyperbolic with eigenvalues $\pm\splittingpara$. \lemref{rthyp} states, that the absolute value of the eigenvalues for \eqref{radtrans} are bounded by 1. Similar to \eqref{splittingeig}, we then need to assure
$$  |\splittingpara| \geq 1.$$
Moreover, we choose $\splittingpara$ so that the convection coefficients in \eqref{sub2} and \eqref{sub3} do not change their signs. This property directly follows for the second subsystem and any choice of $\splittingpara$. For the third subsystem, we require
\begin{equation}
\label{convcoeff}
\frac{\partial \moments{2}}{\partial \moments{1}} + \splittingpara \geq 0.\end{equation}
According to \cite{Schlachter2017} we have for $\solution\in\realizableSet$
$$\frac{\partial \moments{2}}{\partial \moments{1}} = \frac{\moments{3}\moments{0} - \moments{2}\moments{1}}{\moments{0}\moments{2}-\left(\moments{1}\right)^2}  \in [-2,\,2],$$
where $\moments{3}$ can be calculated via tabulation, analogous to \eqref{eq:M1Tabulation}. Thus, we set
$$\splittingpara = \pm\sup_{\x\in\Domain} \left( \max\left(1,\,\bigg|\frac{\moments{3}\moments{0} - \moments{2}\moments{1}}{\moments{0}\moments{2}-\left(\moments{1}\right)^2}\bigg|\right) \right),$$
with the supremum taken over the values of the solution $\solution$ in each space cell for the current time step.
\begin{remark}\label{rem:M1ex}
Similarly to \exref{splittingex}, we can construct an initial state where the solution is violating \hyperbolicity{} after one Lax-Friedrichs time step of the first subsystem \eqref{sub1}. In this case, we are not able to calculate $\moments{2}$  for the second subsystem (cf. \eqref{eq:M1Tabulation}). 
\end{remark}

\subsection{Intrusive polynomial moment method}
Similar to stochastic Galerkin, the intrusive polynomial moment method (IPMM) is based on generalized polynomial chaos. It uses the entropy of the system in order to define a bijection between the solution and a new variable, the entropic variable. This procedure aims at a preservation of physical properties such as hyperbolicity and positivity, although it is noted in \cite{Poette2009} that the assumptions for this preservation are not always given in practice. The method is introduced in \cite{Poette2009} and described as a minimum entropy model in \cite{Kusch2015a}.

Assume that the system \eqref{conslaw} has the entropy - entropy flux pair $(\entropy,\widehat{\entropy})$, which satisfies the entropy inequality
$$\dt \entropy(\solution) + \dx \widehat{\entropy}(\solution) \leq 0,$$
with 
$$\entropy,\widehat{\entropy} \,: \,\R^\dimension \rightarrow \R,$$
and where $\entropy$ is a strictly convex function. The entropic variable is defined by
$$\entropicVarSum :=  \nabla_\solution \entropy (\solution) := \nabla\entropy(\solution).$$
Since $\entropy$ is strictly convex, the map between $\solution$ and $\nabla \entropy(\solution)$ is one-to-one, i.e., $\solution = \InvGradEntropy\!(\entropicVarSum)$. 

Approximating $\entropicVarSum$ as a truncated gPC expansion
$$\entropicVarSum =\begin{pmatrix}
\InvVara\\\vdots\\\InvVar_\IPMMtruncorder
\end{pmatrix} = \IPMMapproach,$$
we obtain
$$\solution_\SGeqIndex = \intRS \InvGradEntropy\!\left(\IPMMapproach\right) \!\xiBasisPoly{\SGeqIndex} \xiPDFdxi,\qquad \SGeqIndex=0,\ldots,\IPMMtruncorder.$$
This leads to the following IPMM model
\begin{equation} \label{IPMMsystem}
\dt\begin{pmatrix}
	\intRS \InvGradEntropy\!\left(\IPMMapproach\right)\! \xiBasisPoly{0} \xiPDFdxi\\ \vdots\\ \intRS \InvGradEntropy\!\left(\IPMMapproach\right)\! \xiBasisPoly{\IPMMtruncorder} \xiPDFdxi
	\end{pmatrix}
+ \dx \begin{pmatrix}
	\intRS \flux\! \left(  \InvGradEntropy\!\left(\IPMMapproach\right)\right)\! \xiBasisPoly{0} \xiPDFdxi\\ \vdots\\ \intRS \flux\! \left(  \InvGradEntropy\!\left(\IPMMapproach\right) \right)\! \xiBasisPoly{\IPMMtruncorder} \xiPDFdxi
	\end{pmatrix} \\[0.3cm]
= 0.
\end{equation}
After performing one step of Lax-Friedrichs for the system \eqref{IPMMsystem}, we can calculate $\entropicVar_{0}(\timevar_{\timeind+1},\x),..,\entropicVar_{\IPMMtruncorder}(\timevar_{\timeind+1},\x)$ out of
$\solution_0(\timevar_{\timeind+1},\x),..,\solution_{\IPMMtruncorder}(\timevar_{\timeind+1},\x)$ via a Newton scheme that solves
$$\solution_{\SGeqIndex} - \intRS \InvGradEntropy\!\left(\IPMMapproach\right)\! \xiBasisPoly{\SGeqIndex} \xiPDFdxi = 0, \qquad \SGeqIndex=0,\ldots,\IPMMtruncorder.$$
\begin{rem}\label{rem:alphaconvex}
	According to \cite[Remark 2.3.3]{Alldredge2015}, we need certain conditions to obtain convergence of the Newton scheme and to derive $\entropicVar$ out of $\SGmomentvec$. In particular, 
	$$  -\SGapproach + \intRS \InvGradEntropy\!\left(\IPMMapproach\right) \IPMMapproach\,\xiPDFdxi
	- \intRS \entropy\left( \InvGradEntropy\!\left(\IPMMapproach\right) \right) \!\xiPDFdxi$$ 
	is required to be $\alpha$-convex. In practice, this is supposed to be ensured for large enough truncation orders $\IPMMtruncorder$.
\end{rem}

\subsubsection{Euler equations}
For the Euler equations, we use the entropy as stated in \cite{Poette2009}
$$\entropy(\solution) = -\density \,\ln\!\left(\density^{-\eulergamma}\left(\energy - \cfrac{1}{2}\,\cfrac{\momentum^2}{\density}\,\right)\right),$$
which yields
$$\nabla\entropy(\solution) = \begin{pmatrix} 
-\ln\!\bigg(\cfrac{2\energy\density - \momentum^2}{2\density^{\eulergamma+1}}\bigg) + \eulergamma - \cfrac{\momentum^2}{2\density\energy - \momentum^2}\,\\[0.4cm]
\cfrac{2\density\momentum}{2\density\energy-\momentum^2}\\[0.3cm]
-\cfrac{2\density^2}{2\density\energy - \momentum^2}
\end{pmatrix},$$
as well as 
$$\InvGradEntropy(\entropicVarSum) = \begin{pmatrix}
\exp\!\bigg(\cfrac{2\InvVara\InvVarc - 2\InvVarc\ln(-\InvVarc) - 2\InvVarc\eulergamma - \InvVarb^2}{2\InvVarc(\eulergamma-1)}\,\bigg) \\[0.4cm]
-\cfrac{\InvVarb}{\InvVarc}\, \exp\!\bigg(\cfrac{2\InvVara\InvVarc - 2\InvVarc\ln(-\InvVarc) - 2\InvVarc\eulergamma - \InvVarb^2}{2\InvVarc(\eulergamma-1)}\,\bigg) \\[0.4cm]
\cfrac{\InvVarb^2-2\InvVarc}{2\InvVarc^2}\, \exp\!\bigg(\cfrac{2\InvVara\InvVarc - 2\InvVarc\ln(-\InvVarc) - 2\InvVarc\eulergamma - \InvVarb^2}{2\InvVarc(\eulergamma-1)}\,\bigg) 
\end{pmatrix}.$$

\subsubsection{$\MN[1]$ model of radiative transfer}
For the $\MN[1]$ model of radiative transfer, we deduce
$$\InvGradEntropy(\entropicVarSum) = \begin{pmatrix}
\int_{-1}^1 \exp(\InvVara + \velocity \InvVarb) \mathrm{d}\velocity\\[0.15cm]
\int_{-1}^1 \velocity\,\exp(\InvVara + \velocity \InvVarb) \mathrm{d}\velocity \end{pmatrix}.$$
Details on this calculation can be found in \cite{Schlachter2017}.
\section{Numerical Results}\label{sec:results}

\subsection{$\MN[1]$ model of Radiative Transfer}
In the following, we test the UQ methods on the $\MN[1]$ model of radiative transfer \eqref{radtrans} with uncertainty $\uncertainty\sim\mathcal{U}(-1,1)$. According to \exref{Legendrebasis}, we take $\xiPDF(\uncertainty) = \frac{1}{2}$ and $\xiBasisPoly{\SGsumIndex}$ as the $\SGsumIndex$th normalized Legendre polynomial. We consider the plane source test \cite{Garrett2014,ganapol2001homogeneous} with stochastically disturbed width of the initial Gaussian, where $x\in\Domain = [-0.5,\,0.5]$, $ \absorption=0$, $\scattering = 1$ and the initial conditions are given by
\begin{align}\label{initialRT}
\moments{0}(\timevar=0,\x,\uncertainty) &= \max\!\left(10^{-4},\, \frac{50^2}{8\pi\,(\uncertainty+2)^2} \, \exp\!\left(-\frac{1}{2}\,\frac{50^2\,\x^2}{(\uncertainty+2)^2}\right)\right),\\
\moments{1}(\timevar=0,\x,\uncertainty) &= 0.
\end{align}
The particles are initially concentrated around the origin and will spread out to the left and right.
We apply the \hyperbolicity{}-preserving SG method (\hSG{}), operator splitting and IPMM to this problem and compare the results with Monte Carlo using $\nbsamples = 100\,000$ samples.

At first, we set the truncation order to $\SGtruncorder=2$ and use $200$ cells as well as $\nbxiQuadNodes=10$ quadrature points in $\uncertainty$\footnote{We did not find the numerical solution of this test case to be very sensitive to the choice of $\nbxiQuadNodes$.}. Then we increase $\SGtruncorder$ to $9$ and the number of cells to $500$. The expected values and standard deviations of the particle density $\moments{0}$ at $\timevar = 0.45$ are calculated via \eqref{SGEV}--\eqref{SGSD} and shown in \figref{fig:RT_dx200_K3_Q10} and \figref{fig:RT_dx500_K11_Q10}. The limited stochastic Galerkin scheme and IPMM yield similar outcomes, slightly in favor of IPMM. They both give a good approximation of the Monte Carlo reference solution, with improved quality for higher truncation orders.

In the operator splitting, we observe situations as in \rmref{rem:M1ex}, where the solution is leaving the \hypset{} and where we are not able to calculate $\moments{2}$ via the tabulation \eqref{eq:M1Tabulation}. In this case, we need to modify the algorithm and redefine in every tabulation step
$$\frac{\moments{1}}{\moments{0}} = \begin{cases}
\frac{\moments{1}}{\moments{0}}, \qquad &\text{if } (\moments{0},\moments{1})^T \in\realizableSet,\\
1, &\text{if } \frac{\moments{1}}{\moments{0}} > 1,\\
-1, &\text{if } \frac{\moments{1}}{\moments{0}} < -1.
\end{cases}$$
The splitting scheme is only converging to the solution of the original system as $\Delta\x\rightarrow0$. This is illustrated in \figref{fig:RT_dx500_K11_Q10}, where the splitting gets closer to the reference solution as in  \figref{fig:RT_dx200_K3_Q10}, since the spatial domain is divided into a finer grid. However, it is still very imprecise compared to the other results. Note that the authors in \cite{Chertock2015} used a Strang splitting together with a higher-order scheme in space and time to overcome this drawback.
\begin{figure}
	\centering
	\externaltikz{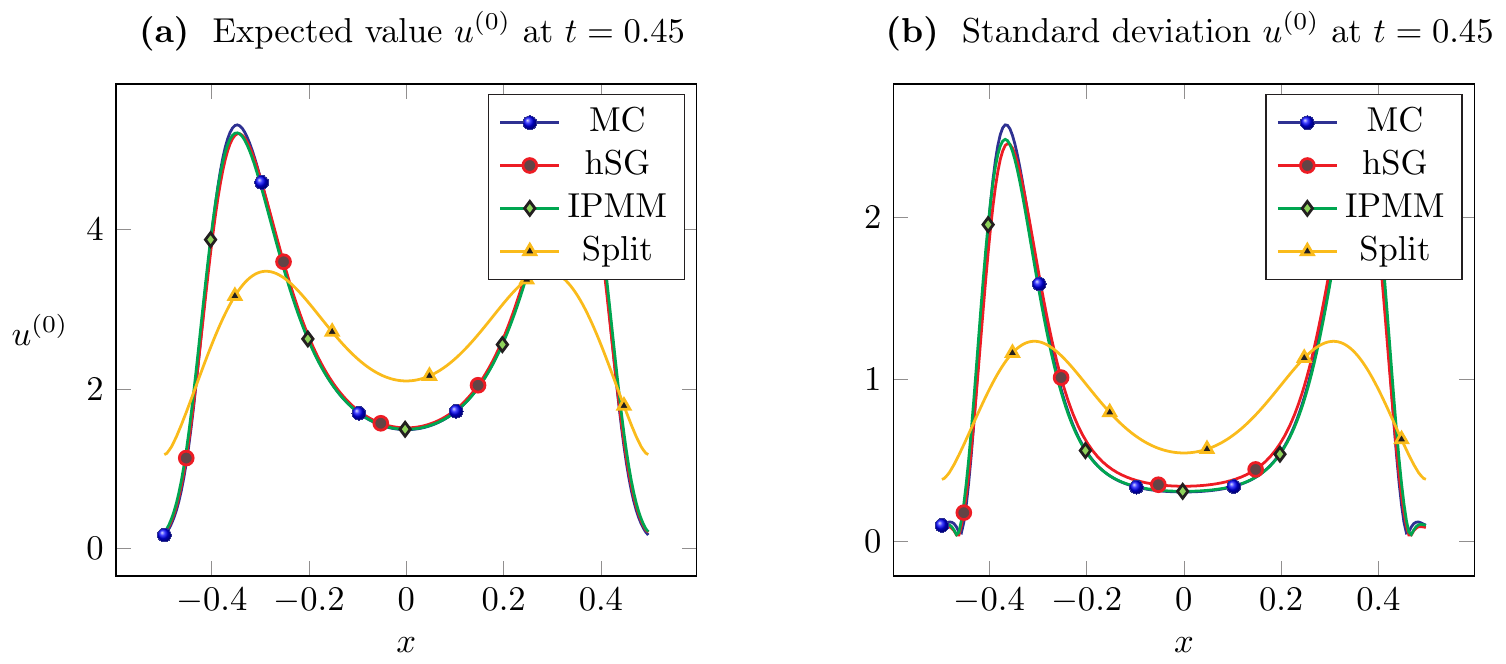}{\input{Images/RT_dx200_K3_Q10.tex}}
	\caption{$\MN[1]$ model of radiative transfer with $\SGtruncorder=2$, $200$ cells and $\nbxiQuadNodes=10$.}
	\label{fig:RT_dx200_K3_Q10}
\end{figure}

\begin{figure}
	\centering
	\externaltikz{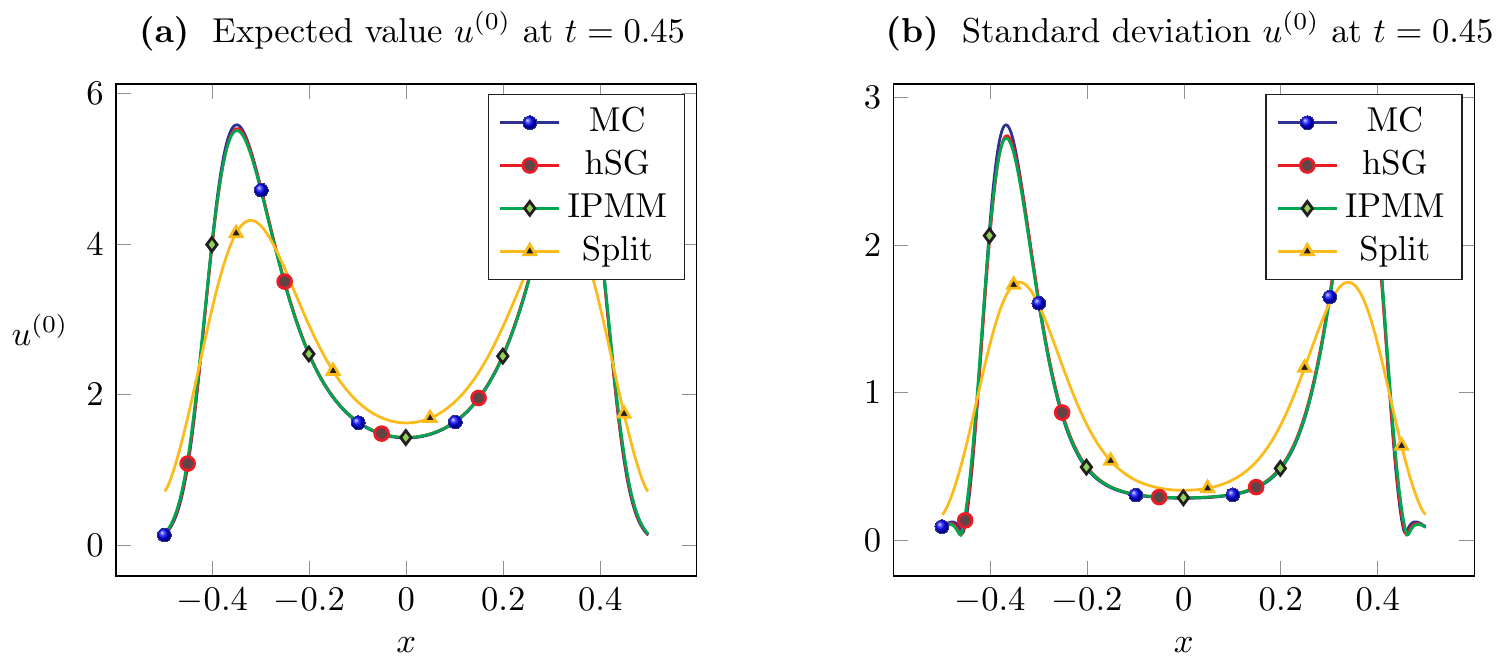}{\input{Images/RT_dx500_K11_Q10.tex}}
	\caption{$\MN[1]$ model of radiative transfer with $\SGtruncorder=9$, $500$ cells and $\nbxiQuadNodes=10$.}
	\label{fig:RT_dx500_K11_Q10}
\end{figure}

\figref{fig:M1_Theta} shows the activity of the \hyperbolicity{} limiter during the \hSG{} method. It is mostly active along the wave fronts of the density which is moving to the left and right. This is where the solution lies closest to the boundary of the \hypset{}. While increasing the truncation order $\SGtruncorder$, the activity of the limiter is decreasing. This is also verified in Table \ref{thetatable}, where the percentage of limited cells throughout the calculation is decreasing and attaining zero as $\SGtruncorder$ reaches 9. In addition to the percentage usage of the limiter, the maximal value of the limiter variable $\limitervariable$ decreases. These outcomes are not surprising since a larger polynomial order yields a better approximation of the (assumed to be \hyperbolic{}) solution. Note that the usual stochastic Galerkin scheme would already fail in the first time step since the solution is leaving the \hypset{} and the tabulation for $\moments{2}$ cannot be performed.

\begin{figure}
	\centering
	\externaltikz{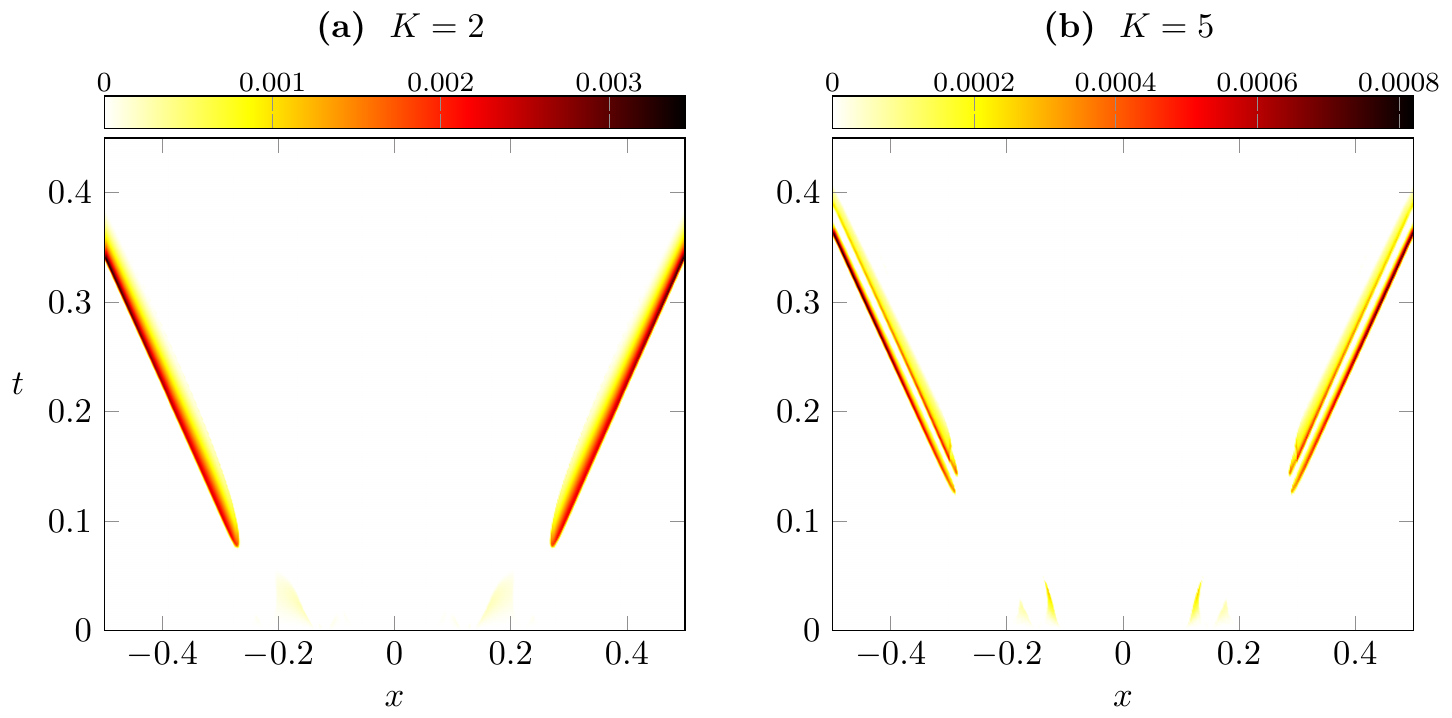}{\input{Images/Theta_M1.tex}}
	\caption{Values of the limiter variable $\limitervariable$ in the \hyperbolicity{}-preserving stochastic Galerkin scheme for the $\MN[1]$ model of radiative transfer with $500$ cells and two truncation orders $\SGtruncorder$. The accuracy of $\limitervariable$ is set to $10^{-5}$. We do not show the values of $\limitervariable$ in the first time step since they are much larger than the others (cf. Table \ref{thetatable}).}
	\label{fig:M1_Theta}
\end{figure}

\begin{table}[htbp]
	\centering
	\begin{tabular}{llllllllll}
		\hline
		\text{Truncation order $\SGtruncorder$} & 1 & 2 & 3 & 4 & 5 & 6 & 7 & 8 & 9 \\
		\hline
		\text{\% of limited cells over all $\timevar$}&7.9949  &  6.5148  &  4.7072  &    3.9291    &    4.4574   &    3.1544      &   2.1705       &  0.0118   &  0  \\
		\text{maximal value of }$\limitervariable$ &0.5647   & 0.3704   & 0.3479    &0.0545  &  0.0152  &  0.0052  &  0.0016   & 0.0002 &0 \\
		\text{maximal value of }$\limitervariable$ for $\timevar>0$ & 0.0076  &  0.0035  &  0.0058 &   0.0015  &  0.0008  &  0.0004 &   0.0003 &0 & 0\\		
		\hline
	\end{tabular}
	\caption{Usage of the limiter variable $\limitervariable$ in the \hyperbolicity{}-preserving stochastic Galerkin scheme for the $\MN[1]$ model of radiative transfer with $500$ cells and different truncation orders. The accuracy of $\limitervariable$ is set to $10^{-5}$.}
	\label{thetatable}
\end{table}

\newpage

\begin{figure}
	\centering
	\externaltikz{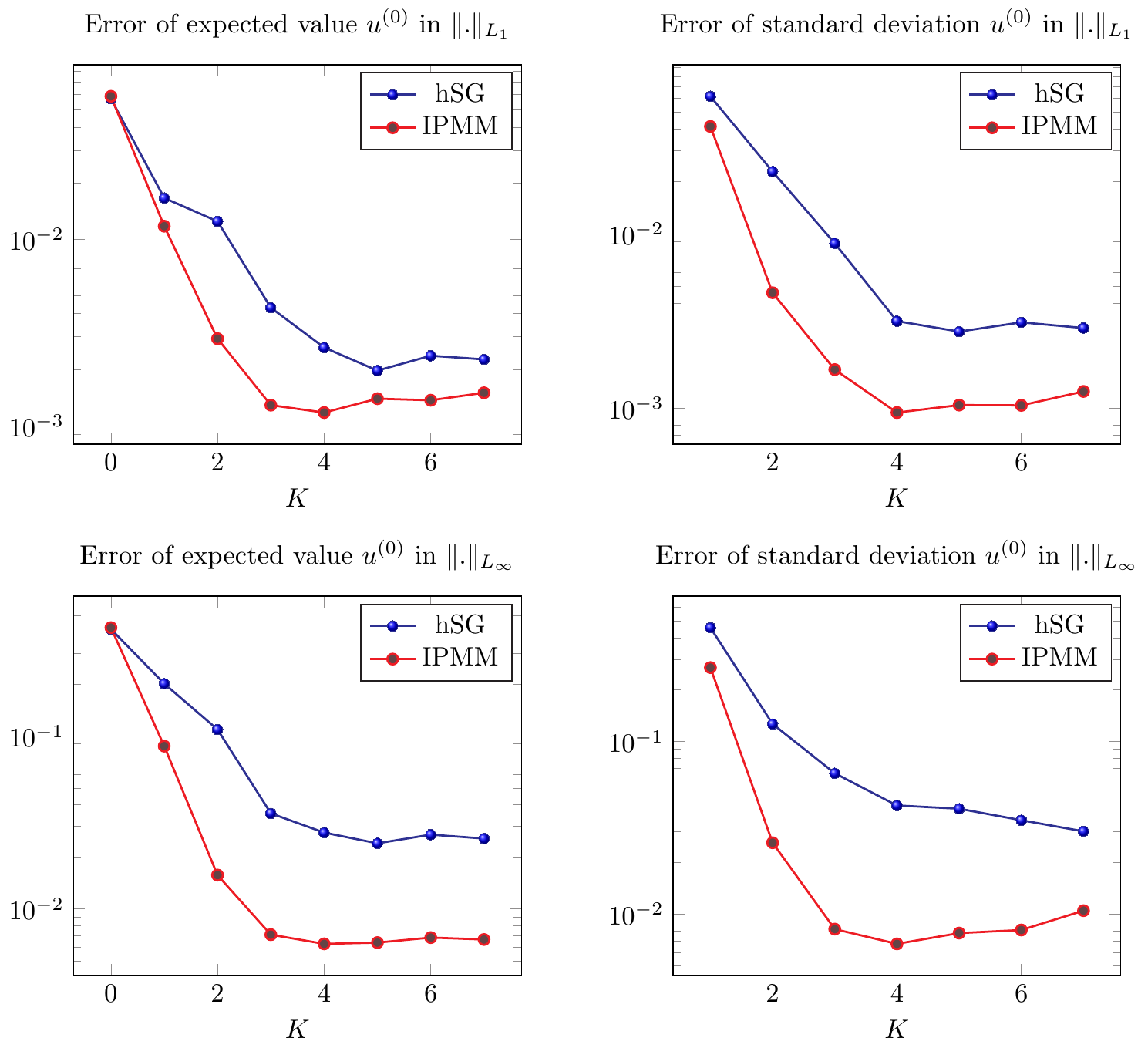}{\input{Images/Error_M1.tex}}
	\caption{Error plots of $\moments{0}$ for the $\MN[1]$ model of radiative transfer with $10000$ cells at $\timevar = 0.08$ and in a logarithmic scale. The error is given for two different norms $\|.\|_{L_1}\!$ \eqref{L1} and $\|.\|_{L_\infty}\!$ \eqref{Linfty}. Because of the form \eqref{SGSD}, computing the standard deviation is only reasonable for $\SGtruncorder>0$. The exact values are given in \tabref{errortable}.}
	\label{fig:Error_M1}
\end{figure}{\tiny }

\figref{fig:Error_M1} and \tabref{errortable} demonstrate the error of the density from \hSG{} and IPMM in the discrete $L_1$ and $L_\infty$ norm, calculated by their differences to Monte Carlo and evaluated at a given time $\timevar^*$. 
We define those norms as
\begin{align}
\big\|h - g\big\|_{L_1(\Domain)} &= \sum_{\cellind=1}^\ncells \abs{\cell{\cellind}} \big|h(\timevar^*,\x_\cellind) - g(\timevar^*,\x_\cellind)\big|, \label{L1}\\[0.2pt]
\big\|h - g\big\|_{L_\infty(\Domain)} &= \max_{\cellind=1,\ldots,\ncells} \big|h(\timevar^*,\x_\cellind) - g(\timevar^*,\x_\cellind)\big|.\label{Linfty}
\end{align} 
Then, $g$ is chosen as the expected value \eqref{SGEV} or standard deviation \eqref{SGSD} of the density $\moments{0}$ in Monte Carlo and $h$ as the corresponding values in the limited stochastic Galerkin or IPMM method.

The IPMM error is indeed smaller than the error of the hyperbolicity-preserving stochastic Galerkin method. However, IPMM comes with a far more complex algorithm, in our cases the CPU time was about 2--3 times larger than for \hSG{} (depending on the truncation order).
In this context, \hSG{} shows a very acceptable error since it only requires a small modification of the classical SG algorithm. For $\SGtruncorder\geq4$, the inaccuracies of the underlying spatial discretization are disturbing our results. In order to see the expected spectral convergence, a higher-order method for the discretization in space and time is required. 

\begin{table}[htbp]
	\centering
	\begin{tabular}{lll | llllllll}
		\hline
		\multicolumn{3}{c|}{Truncation order $\SGtruncorder$} & 0& 1 & 2 & 3 & 4 & 5 & 6 & 7  \\
		\hline
		$\mathbb{E}(\moments{0})$ & $\|.\|_{L_1}$ & \hSG{} &   0.0569 & 0.0167&    0.0125 &   0.0043  &  0.0026 &   0.0020&    0.0024  &  0.0023    \\
	         &  &IPMM & 0.0587   & 0.0118   & 0.0029  & 0.0013&   0.0012    &0.0014&    0.0014  &  0.0015   \\
	         	\hline
		     & $\|.\|_{L_\infty}$ & \hSG{} & 0.4161   & 0.2010  &  0.1092 &   0.0357 &   0.0276  &  0.0240 &   0.0269 &   0.0255     \\
		     &  & IPMM &0.4244   & 0.0875   & 0.0157 &   0.0071   & 0.0063    &0.0064 &   0.0068  &  0.0067   \\
		     	\hline
		$s(\moments{0})$ & $\|.\|_{L_1}$ & \hSG{} & -- & 0.0617  &  0.0228  &  0.0088  &  0.0032 &   0.0027  &  0.0031  &  0.0029   \\
		     & & IPMM & -- & 0.0414  &  0.0046 &   0.0017 &   0.0009 &   0.0010  &  0.0010 &   0.0012\\
		     	\hline
		     & $\|.\|_{L_\infty}$ & \hSG{} & -- & 0.4576  &  0.1265  &  0.0655  &  0.0426 &   0.0408  &  0.0350 &   0.0301  \\
		     &  & IPMM & --   & 0.2684  &  0.0259 &   0.0082  &  0.0067  &  0.0078  &  0.0081   & 0.0105  \\
		\hline
	\end{tabular}
	\caption{Errors in expected value and standard deviation of the density $\moments{0}$ for the $\MN[1]$ model of radiative transfer with $10000$ cells at $\timevar = 0.08$. The error is given for two different norms $\|.\|_{L_1}\!$ \eqref{L1} and $\|.\|_{L_\infty}\!$ \eqref{Linfty} as well as for the limited stochastic Galerkin method and IPMM. Because of the form \eqref{SGSD}, computing the standard deviation is only reasonable for $\SGtruncorder>0$.}
	\label{errortable}
\end{table}

\subsection{Euler Equations}
In this section, we apply the UQ methods to the Euler equations \eqref{eulereq} from \secref{SecEuler}. We again consider $\uncertainty\sim\mathcal{U}(-1,1)$, set the spatial domain to $\Domain = [0,\,1]$ and the adiabatic constant to $\eulergamma = 1.4$. Moreover, we take two different initial conditions which are demonstrated in the following subsections.

\subsubsection{Uncertain Sod test case}
Consider the first set of initial conditions given by
\begin{equation} \label{eq:initialEulerSod}
\left.
\hspace*{2cm}
\begin{alignedat}{1}
 \density(\timevar=0,\x,\uncertainty) &= \begin{cases}
1, \qquad &\x < 0.5 + 0.05\uncertainty, \hspace*{2cm}\\
0.125, \qquad &\x\geq 0.5 + 0.05\uncertainty,
\end{cases}\\
\momentum(\timevar=0,\x,\uncertainty) &= 0,\\
\energy(\timevar =0,\x,\uncertainty) &= \begin{cases}
2.5, \qquad &\x < 0.5 + 0.05\uncertainty,\\
0.25, \qquad &\x\geq 0.5 + 0.05\uncertainty.
\end{cases}
\end{alignedat}
\right\}
\end{equation}
This test case, studied in \cite{Poette2009,Chertock2015}, represents a modification of the Sod Riemann problem, where the position of the discontinuity is depending on the uncertainty $\uncertainty$.

We divide $\Domain$ into $500$ cells and apply each of the three UQ methods to the Euler equations, using a truncation order $\SGtruncorder=2$. Furthermore, we increase the number of quadrature nodes to $\nbxiQuadNodes=100$\footnote{Since the initial discontinuity depends on $\uncertainty$, a higher quadrature rule is necessary compared to the smooth dependence on $\uncertainty$ in \eqref{initialRT}.}. The methods are compared to Monte Carlo with $100\,000$ samples.

The expected value in \figref{fig:Euler_a_a} and \figref{fig:Euler_a_b} indicate a very good agreement between Monte Carlo, the limited SG and IPMM. The standard deviation shown in \figref{fig:Euler_a_c} and \figref{fig:Euler_a_d} is slightly smaller for \hSG{} than for IPMM, yet they are both close to Monte Carlo.

 The expected value of the splitting scheme has a similar structure compared to the other solutions but still gives the poorest approximation. This can especially be seen in the standard deviation. The method might be improved by using more cells, however, in \figref{fig:Euler_a_d} the solution even shows oscillations around $\x\approx 0.78$.\\ Moreover, a negative pressure occurs while computing the splitting parameter $\splittingpara$, meaning that we are leaving the \hyperbolicity{} set. This observation coincides with the statement of \exref{splittingex}. According to \eqref{splittingeig}, we require
 $$-|\splittingpara| \leq \frac{\momentum}{\density}-\sqrt{\eulergamma\,\frac{\pressure}{\density}} < \frac{\momentum}{\density}+\sqrt{\eulergamma\,\frac{\pressure}{\density}} \leq |\splittingpara|,$$
 where $\sqrt{\eulergamma\frac{\pressure}{\density}}$ is complex for negative values of $\pressure$. In our algorithm we have ignored those values for the calculation of $\splittingpara$, resulting in oscillations due to the violated CFL condition. Thus, we will not show the method in the next test case.
  
 \begin{figure}
 	\centering
 	\settikzlabel{fig:Euler_a_a} \settikzlabel{fig:Euler_a_b} \settikzlabel{fig:Euler_a_c} \settikzlabel{fig:Euler_a_d}
 	\externaltikz{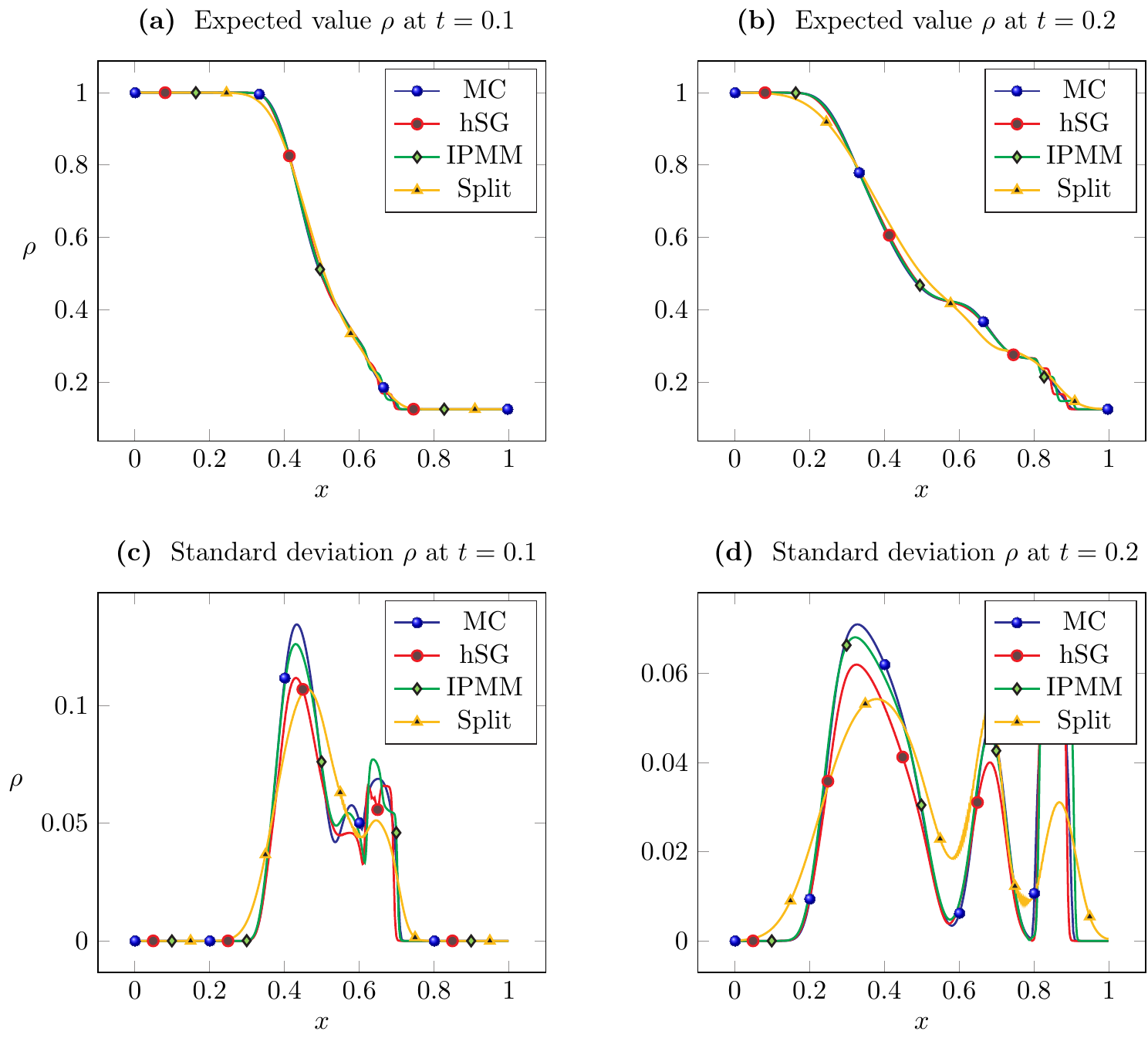}{\input{Images/Euler_a_dx2000_K3_Q10.tex}}
 	\caption{Euler equations with initial state \eqref{eq:initialEulerSod}, $\SGtruncorder=2$, 500 cells and $\nbxiQuadNodes=100$.}
 	\label{fig:Euler_a_dx200_K3_Q10}
 \end{figure}

\subsubsection{Uncertain Riemann problem with shock}
Next, we consider the second set of initial conditions for the Euler equations
\begin{equation} \label{eq:initialEulerLO}
\left.
\hspace*{2cm}
\begin{alignedat}{1}
\density(\timevar=0,\x,\uncertainty) &= \begin{cases}
1, \qquad &\x < 0.5 + 0.07\uncertainty, \hspace*{2cm}\\
0.125, \qquad &\x\geq 0.5 + 0.07\uncertainty,
\end{cases}\\
\momentum(\timevar=0,\x,\uncertainty) &= 0,\\
\energy(\timevar =0,\x,\uncertainty) &= \begin{cases}
0.25, \qquad &\x < 0.5 + 0.07\uncertainty,\\
2.5, \qquad &\x\geq 0.5 + 0.07\uncertainty.
\end{cases}
\end{alignedat}
\right\}
\end{equation}
They are inducing a numerically more complex situation where we found the solutions to be more likely to leave the \hypset{}. We use $\SGtruncorder=9$, 500 cells, $\nbxiQuadNodes=100$ and show the result in \figref{fig:Euler_b}.  The expected values of the density for Monte Carlo and IPMM in \figref{fig:Euler_b_a} coincide very well, whereas the \hSG{} solution slightly differs from these values around the shock at $\x=0.3$. The standard deviation in \figref{fig:Euler_b_b} shows similar approximations of IPMM and \hSG{} compared to the reference solution of Monte Carlo.

\begin{figure}
	\centering
	\settikzlabel{fig:Euler_b_a} \settikzlabel{fig:Euler_b_b} 
	\externaltikz{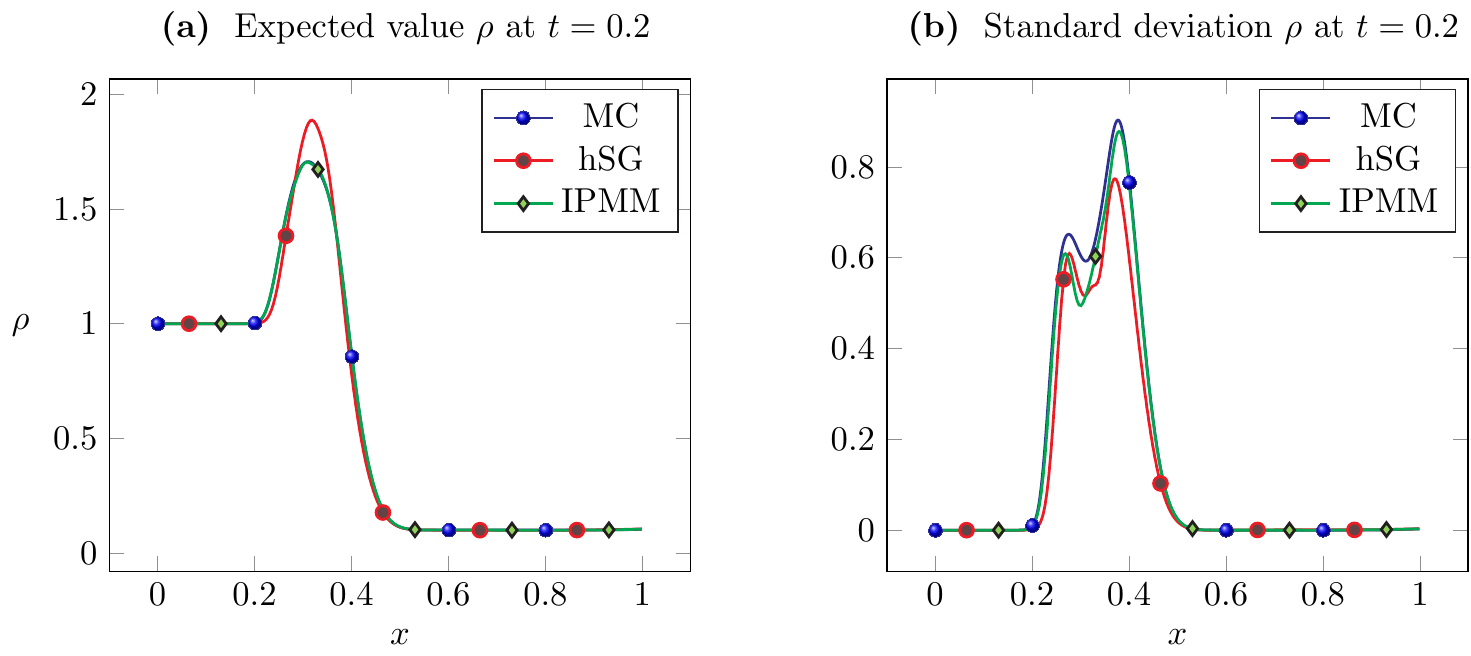}{\input{Images/Euler_b_dx200_K3_Q10.tex}}
	\caption{Euler equations with initial state \eqref{eq:initialEulerLO}, $\SGtruncorder=9$, $500$ cells and $\nbxiQuadNodes=100$.}
	\label{fig:Euler_b}
\end{figure}

\tabref{Euler_thetatable} and \figref{fig:Euler_Theta} demonstrate the limiter usage during the performance of \hSG{} in each of the two initial states. 
As expected, the limiter is more active for \eqref{eq:initialEulerLO}. More precisely, in \figref{fig:Euler_Theta_a} it is used in 0.63\% of the cells over time with a maximal limiter variable of $0.0044$, whereas in \figref{fig:Euler_Theta_b} we deduce activity in almost every time step (in 1.57\% of the cells) with maximum $0.0082$. In \tabref{Euler_thetatable}, we again observe a reduced usage as the truncation order increases. 

\begin{table}[htbp]
	\centering
	\begin{tabular}{ll | lllllllll}
		\hline
		\multicolumn{2}{c|}{Truncation order $\SGtruncorder$} & 1 & 2 & 3 & 4 & 5 & 6 & 7 & 8 & 9 \\
		\hline
		\text{\% of limited}&IC 1 & 0.0332 &   0.6338  &  0.3652  &  0.0897 &   0.0197 &   0.0197 &   0.0162    &0.0154  &  0.0010  \\
		  cells  over all $\timevar$ &IC 2& 2.0654  &  1.5688   & 1.0264    &0.7714   & 0.4719 &   0.2872  &      0.2425  &  0.1952    &    0.1531  \\
		 \hline
		\text{maximal value  } &IC 1 &0.3581  &0.2608   & 0.2336    &0.2240  &  0.2172  &  0.2144 &  0.2028   & 0.2083 & 0.1965 \\
		of $\limitervariable$ &IC 2 &0.3581   &  0.2608   &  0.2353  &0.2240  &  0.2182  &  0.2132  &  0.2079 & 0.2083 & 0.2096 \\
		\hline
		\text{maximal value } &IC 1 & 0.0003  &  0.0044  &  0.0083 &   0.0054  &  0.0006  &  0.0001 &   0 &0 & 0\\		
       of $\limitervariable$ for $\timevar>0$ &IC 2 & 0.0043 &  0.0082  &   0.0114 &   0.0145  &  0.0171  &  0.0169 &   0.0212 & 0.0244 & 0.0244\\				\hline
	\end{tabular}
	\caption{Usage of the limiter variable $\limitervariable$ in the \hSG{} scheme for the Euler equations with $500$ cells, different truncation orders and the two initial conditions \eqref{eq:initialEulerSod} and \eqref{eq:initialEulerLO}, denoted by IC 1 and IC 2, respectively. The accuracy of $\limitervariable$ is set to $10^{-5}$.}
	\label{Euler_thetatable}
\end{table}

\begin{figure}
	\centering
	\settikzlabel{fig:Euler_Theta_a} \settikzlabel{fig:Euler_Theta_b}
	\externaltikz{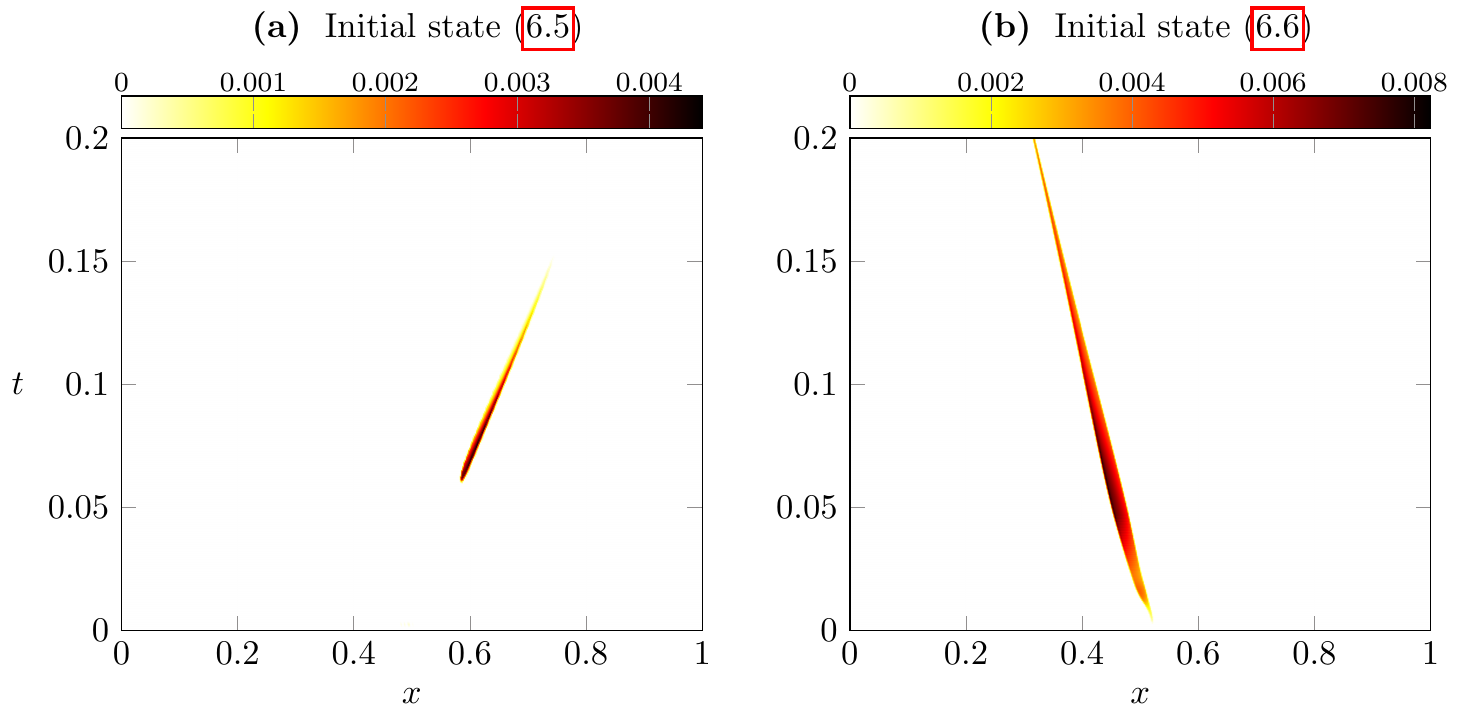}{\input{Images/Theta_Euler.tex}}
	\caption{Values of the limiter variable $\theta$ in the \hyperbolicity{}-preserving stochastic Galerkin scheme for the Euler equations with $500$ cells, $K=2$ and the two initial conditions. The accuracy of $\limitervariable$ is set to $10^{-5}$. We do not show the values of $\limitervariable$ in the first time step since they are much larger as the others.}
	\label{fig:Euler_Theta}
\end{figure}

\section{Conclusions and Outlook}

We have derived a modification of the classical stochastic Galerkin scheme that maintains the hyperbolicity of the original deterministic system under the assumption of \hyperbolic{} initial conditions. It provides good results that almost reach the quality of the intrusive polynomial moment method while being notably simpler to derive (no need to know the entropy of the system) and computationally much cheaper. 

Until now, only a simple first-order discretization in space and time has been applied to the modified SG scheme. Future work should incorporate the use of higher-order schemes, like the discontinuous-Galerkin scheme \cite{Alldredge2015,Cockburn1989a,Cockburn1990}, to further increase the efficiency of the approximation.

Due to the inherent analogy of kinetic theory and Uncertainty Quantification, some further ideas might be transferable. One example is the class of positive $\PN$ models \cite{Hauck2010a}, that are derived using a modified entropy (related to IPMM), which might give similar results as our modified SG scheme without the need of using a \hyperbolicity{} limiter. Furthermore, the idea of the kinetic scheme \cite{Schneider2015b,Hauck2010} might be adoptable, simplifying the \hyperbolicity{} limiter in cases where the domain of \hyperbolicity{} is not known (or expensive to compute).
\section*{Acknowledgements}
Funding by the Deutsche Forschungsgemeinschaft (DFG) within the RTG GrK 1932 ``Stochastic Models for Innovations in the Engineering Science'' is gratefully acknowledged.

\bibliographystyle{siam}
\bibliography{library,bibliography}

\begin{thebibliography}{10}

\bibitem{Abgrall2017}
{\sc R.~Abgrall and S.~Mishra}, {\em {Uncertainty quantification for systems of
  conservation laws}}, Seminar f{\"{u}}r Angewandte Mathematik, ETH
  Z{\"{u}}rich, 18 (2017), pp.~507--544.

\bibitem{acharjee2006uncertainty}
{\sc S.~Acharjee and N.~Zabaras}, {\em {Uncertainty propagation in finite
  deformations -- A spectral stochastic Lagrangian approach}}, Computer methods
  in applied mechanics and engineering, 195 (2006), pp.~2289--2312.

\bibitem{Alldredge2015}
{\sc G.~Alldredge and F.~Schneider}, {\em {A realizability-preserving
  discontinuous Galerkin scheme for entropy-based moment closures for linear
  kinetic equations in one space dimension}}, Journal of Computational Physics,
  295 (2015), pp.~665--684.

\bibitem{AniPenSam91}
{\sc A.~M. Anile, S.~Pennisi, and M.~Sammartino}, {\em {A thermodynamical
  approach to Eddington factors}}, Journal of Mathematical Physics, 32 (1991),
  p.~544.

\bibitem{Bru02}
{\sc T.~A. Brunner}, {\em {Forms of approximate radiation transport}},
  SAND2002-1778, Sandia National Laboratory,  (2002).

\bibitem{BruHol01}
{\sc T.~A. Brunner and J.~P. Holloway}, {\em {One-dimensional Riemann solvers
  and the maximum entropy closure}}, Journal of Quantitative Spectroscopy and
  Radiative Transfer, 69 (2001), pp.~543--566.

\bibitem{Cameron1947}
{\sc R.~H. Cameron and W.~T. Martin}, {\em {The Orthogonal Development of
  Non-Linear Functionals in Series of Fourier-Hermite Functionals}}, Annals of
  Mathematics, 48 (1947), pp.~385--392.

\bibitem{Chertock2015}
{\sc A.~Chertock, S.~Jin, and A.~Kurganov}, {\em {An operator splitting based
  stochastic Galerkin method for the one-dimensional compressible Euler
  equations with uncertainty}}, Preprint,  (2015), pp.~1--21.

\bibitem{Chidyagwai2017}
{\sc P.~Chidyagwai, M.~Frank, F.~Schneider, and B.~Seibold}, {\em {A
  Comparative Study of Limiting Strategies in Discontinuous Galerkin Schemes
  for the {$M_1$} Model of Radiation Transport}},  (2017), pp.~1--24.

\bibitem{Cockburn1990}
{\sc B.~Cockburn, S.~Hou, and C.-W. Shu}, {\em {The Runge-Kutta local
  projection discontinuous Galerkin finite element method for conservation
  laws. IV. The multidimensional case}}, Mathematics of Computation, 54 (1990),
  pp.~545--581.

\bibitem{Cockburn1989a}
{\sc B.~Cockburn and C.-W. Shu}, {\em {TVB Runge-Kutta local projection
  discontinuous Galerkin finite element method for conservation laws. II.
  General framework}}, Mathematics of Computation, 52 (1989), p.~411.

\bibitem{Curto1991}
{\sc R.~Curto and L.~Fialkow}, {\em {Recursiveness, positivity, and truncated
  moment problems}}, Houston J. Math, 17 (1991), pp.~603--636.

\bibitem{Ducl2010}
{\sc R.~Duclous, B.~Dubroca, and M.~Frank}, {\em {A deterministic partial
  differential equation model for dose calculation in electron radiotherap}},
  Physics in Medicine and Biology, 55 (2010), p.~3843.

\bibitem{Frank2005}
{\sc M.~Frank}, {\em {Partial Moment Entropy Approximation to Radiative Heat
  Transfer}}, Pamm, 5 (2005), pp.~659--660.

\bibitem{ganapol2001homogeneous}
{\sc B.~D. Ganapol, R.~S. Baker, J.~A. Dahl, and R.~E. Alcouffe}, {\em
  {Homogeneous infinite media time-dependent analytical benchmarks}}, tech.
  rep., Tech. Rep. LA-UR-01-1854. Los Alamos National Laboratory, 2001.

\bibitem{Garrett2014}
{\sc C.~K. Garrett and C.~D. Hauck}, {\em {A Comparison of Moment Closures for
  Linear Kinetic Transport Equations: The Line Source Benchmark}}, Transport
  Theory and Statistical Physics,  (2013).

\bibitem{ghanem2003stochastic}
{\sc R.~G. Ghanem and P.~D. Spanos}, {\em {Stochastic finite elements: a
  spectral approach}}, Courier Corporation, 2003.

\bibitem{Hauck2010a}
{\sc C.~Hauck and R.~McClarren}, {\em {Positive {$P_N$} Closures}}, SIAM
  Journal on Scientific Computing, 32 (2010), pp.~2603--2626.

\bibitem{Hauck2010}
{\sc C.~D. Hauck}, {\em {High-order entropy-based closures for linear transport
  in slab geometry}}, Communications in Mathematical Sciences, 9 (2011),
  pp.~187--205.

\bibitem{Ker76}
{\sc D.~S. Kershaw}, {\em {Flux Limiting Nature's Own Way: A New Method for
  Numerical Solution of the Transport Equation}}, tech. rep., LLNL Report
  UCRL-78378, 1976.

\bibitem{Kusch2015a}
{\sc J.~Kusch}, {\em {Uncertainty Quantification for Hyperbolic Equations}},
  RWTH Aachen University,  (2015), pp.~1--23.

\bibitem{Levermore1996}
{\sc C.~D. Levermore}, {\em {Moment closure hierarchies for kinetic theories}},
  Journal of Statistical Physics, 83 (1996), pp.~1021--1065.

\bibitem{Lewis-Miller-1984}
{\sc E.~E. Lewis and J.~{W. F. Miller}}, {\em {Computational Methods in Neutron
  Transport}}, John Wiley and Sons, New York, 1984.

\bibitem{lucor2007stochastic}
{\sc D.~Lucor, C.~Enaux, H.~Jourdren, and P.~Sagaut}, {\em {Stochastic design
  optimization: Application to reacting flows}}, Computer Methods in Applied
  Mechanics and Engineering, 196 (2007), pp.~5047--5062.

\bibitem{Min78}
{\sc G.~N. Minerbo}, {\em {Maximum entropy Eddington factors}}, J. Quant.
  Spectrosc. Radiat. Transfer, 20 (1978), pp.~541--545.

\bibitem{Olbrant2012}
{\sc E.~Olbrant, C.~D. Hauck, and M.~Frank}, {\em {A realizability-preserving
  discontinuous Galerkin method for the M1 model of radiative transfer}},
  Journal of Computational Physics, 231 (2012), pp.~5612--5639.

\bibitem{Poette2009}
{\sc G.~Po{\"{e}}tte, B.~Despr{\'{e}}s, and D.~Lucor}, {\em {Uncertainty
  quantification for systems of conservation laws}}, Journal of Computational
  Physics, 228 (2009), pp.~2443--2467.

\bibitem{Schlachter2017}
{\sc L.~Schlachter}, {\em {Uncertainty Quantification for Hyperbolic
  Equations}}, TU Kaiserslautern,  (2017).

\bibitem{Schneider2016a}
{\sc F.~Schneider}, {\em {Kershaw closures for linear transport equations in
  slab geometry II: high-order realizability-preserving discontinuous-Galerkin
  schemes}}, Journal of Computational Physics, 322 (2016), pp.~920--935.

\bibitem{schneider2016moment}
\leavevmode\vrule height 2pt depth -1.6pt width 23pt, {\em {Moment models in
  radiation transport equations}}, Verlag Dr. Hut, 2016.

\bibitem{Schneider2014}
{\sc F.~Schneider, G.~W. Alldredge, M.~Frank, and A.~Klar}, {\em {Higher Order
  Mixed-Moment Approximations for the Fokker--Planck Equation in One Space
  Dimension}}, SIAM Journal on Applied Mathematics, 74 (2014), pp.~1087--1114.

\bibitem{Schneider2015b}
{\sc F.~Schneider, G.~W. Alldredge, and J.~Kall}, {\em {A
  realizability-preserving high-order kinetic scheme using WENO reconstruction
  for entropy-based moment closures of linear kinetic equations in slab
  geometry}}, Kinetic and Related Models,  (2015).

\bibitem{Toro2009}
{\sc E.~F. Toro}, {\em {Riemann Solvers and Numerical Methods for Fluid
  Dynamics}}, 2009.

\bibitem{vreugdenhil2013numerical}
{\sc C.~B. Vreugdenhil}, {\em {Numerical methods for shallow-water flow}},
  vol.~13, Springer Science {\&} Business Media, 2013.

\bibitem{Wiener1938}
{\sc N.~Wiener}, {\em {The homogeneous chaos.}}, Amer. J. Math, 60 (1938),
  pp.~897--936.

\bibitem{xiu2002stochastic}
{\sc D.~Xiu, D.~Lucor, C.-H. Su, and G.~E. Karniadakis}, {\em {Stochastic
  modeling of flow-structure interactions using generalized polynomial chaos}},
  Journal of Fluids Engineering, 124 (2002), pp.~51--59.

\bibitem{Zhang2010}
{\sc X.~Zhang and C.-W. Shu}, {\em {On positivity preserving high order
  discontinuous Galerkin schemes for compressible Euler equations on
  rectangular meshes}}, Journal of Computational Physics, 229 (2010),
  pp.~8918--8934.

\bibitem{Zhang2011b}
{\sc X.~Zhang and C.-W. Shu}, {\em {Maximum-principle-satisfying and
  positivity-preserving high-order schemes for conservation laws: survey and
  new developments}}, Proceedings of the Royal Society A: Mathematical,
  Physical and Engineering Sciences, 467 (2011), pp.~2752--2776.

\end{thebibliography}

\end{document}